\documentclass[11pt,oneside]{amsart}
\usepackage{amscd,amssymb}

\usepackage{graphicx}
\theoremstyle{plain}
\newcommand{\F}{\mathbb{F}}

\newcommand{\cA}{\mathcal{A}}

\newcommand{\X}{\mathcal{X}}

\newcommand\be{\begin{equation}}
\newcommand\ee{\end{equation}}
\newcommand\benn{\begin{equation*}}
\newcommand\eenn{\end{equation*}}
\newcommand\bea{\begin{eqnarray}}
\newcommand\eea{\end{eqnarray}}
\newcommand\beann{\begin{eqnarray*}}
\newcommand\eeann{\end{eqnarray*}}

\newcommand{\bk}{\backslash}

\newcommand{\R}{\mathbb{R}}

\newcommand{\Or}{\ensuremath{{\mathcal O}}}

\newcommand{\ga}{\alpha}     
\newcommand{\gd}{{\delta_\Gamma}}     
\newcommand{\vep}{\varepsilon} 
\newcommand{\supp}{\operatorname{supp}}

\renewcommand{\>}{\right\rangle}

\newtheorem{Thm}[equation]{Theorem}

\newtheorem{Cor}[equation]{Corollary}
\newtheorem{Prop}[equation]{Proposition}

\newtheorem{Lem}[equation]{Lemma}
\newtheorem{Def}[equation]{Definition}
\numberwithin{equation}{section}

\newcommand{\q}{\mathbb{Q}}

\newcommand{\e}{\epsilon}
\newcommand{\z}{\mathbb{Z}}
\renewcommand{\q}{\mathbb{Q}}
\newcommand{\n}{\mathbb{N}}
\renewcommand{\c}{\mathbb{C}}
\newcommand{\br}{\mathbb{R}}

\newcommand{\ba}{\backslash}

\newcommand{\G}{\Gamma}

\newcommand{\g}{\gamma}

\renewcommand{\P}{\mathcal P}

\newcommand{\la}{\langle}
\newcommand{\ra}{\rangle}

\newcommand{\SL}{\operatorname{SL}}
\newcommand{\GL}{\operatorname{GL}}

\newcommand{\Spin}{\operatorname{Spin}}

\renewcommand{\O}{\operatorname{O}}
\newcommand{\SO}{\operatorname{SO}}
\newcommand{\SU}{\operatorname{SU}}

\newcommand{\I}{\operatorname{I}}

\newcommand{\Isom}{\operatorname{Isom}}

\newcommand{\PSL}{\op{PSL}}
\newcommand{\bi}{\begin{itemize}}

\newcommand{\ei}{\end{itemize}}

\newcommand{\op}{\operatorname}

\newcommand{\vs}{\vskip 5pt}

\newcommand{\bH}{\mathbb H}

\begin{document}

\title[Apollonian circle packings]{Apollonian circle packings and closed horospheres on hyperbolic $3$-manifolds}

\author{Alex Kontorovich and Hee Oh \\ (with appendix by Hee Oh and Nimish Shah)}

\address{Mathematics department,
Brown University, Providence, RI}
\email{alexk@math.brown.edu}
\thanks{Kontorovich is supported by an NSF Postdoc, grant DMS 0802998.}
\address{Mathematics department, Brown university, Providence, RI
and Korea Institute for Advanced Study, Seoul, Korea}
\email{heeoh@math.brown.edu}
\thanks{Oh is partially supported by NSF grant DMS 0629322.}

\address{Department of Mathematics, The Ohio State University, Columbus, OH}
\email{shah@math.ohio-state.edu}
\begin{abstract} We show that for a given bounded Apollonian circle packing $\mathcal P$,
there exists a constant $c>0$ such that the number of circles 
of curvature at most $T$  is asymptotic to $c\cdot T^\alpha$
as $T\to \infty$. Here $\alpha\approx 1.30568(8)$ is the residual dimension of the packing. 
For $\mathcal P$ integral,
let $\pi^\P(T)$ denote the number of circles with prime curvature less than $T$.
 Similarly let $\pi_2^\P(T)$ be the number of pairs of tangent circles with prime curvatures less than $T$. We obtain the upper bounds $\pi^\P(T)\ll T^\ga/\log T$ and $\pi_2^\P(T)\ll T^\ga/(\log T)^2$, which are sharp up to constant multiple.
The main ingredient of our proof is the effective equidistribution of expanding closed horospheres in the unit tangent bundle of a geometrically finite hyperbolic $3$-manifold $\G\ba \bH^3$ under the assumption that the critical exponent of $\G$ exceeds one.
\end{abstract}


\maketitle
\tableofcontents
\section{Introduction}
A set of four mutually tangent circles in the plane with distinct points of tangency is called a {\it Descartes configuration}.
Given a Descartes configuration, one  can construct four new circles, each of which is tangent to three of the given ones.
 Continuing
to repeatedly fill the interstices between mutually tangent circles with further tangent circles,
we arrive
at an infinite circle
packing. It is  called an {\it Apollonian circle packing}, after the great geometer Apollonius of Perga (262-190 BC).

See Figure \ref{f1} showing the first three generations of this procedure, where each circle
is labeled with its curvature (that is, the reciprocal of its radius).
 Unlike the inner circles, the bounding circle is oriented so that its ``outward'' normal vector points into the packing. In Figure \ref{f2}, the outermost circle has curvature $-1$ (the sign conveys its orientation).

The astute reader would do well to peruse the lovely series of papers by Graham, Lagarias, Mallows, Wilks, and Yan on this beautiful topic, especially \cite{GrahamLagariasMallowsWilksYanI} and  \cite{GrahamLagariasMallowsWilksYanI-n}, as well as the recent letter of Sarnak
to Lagarias \cite{SarnakToLagarias} which inspired this paper.

\begin{figure}
 \includegraphics [width=1.5in]{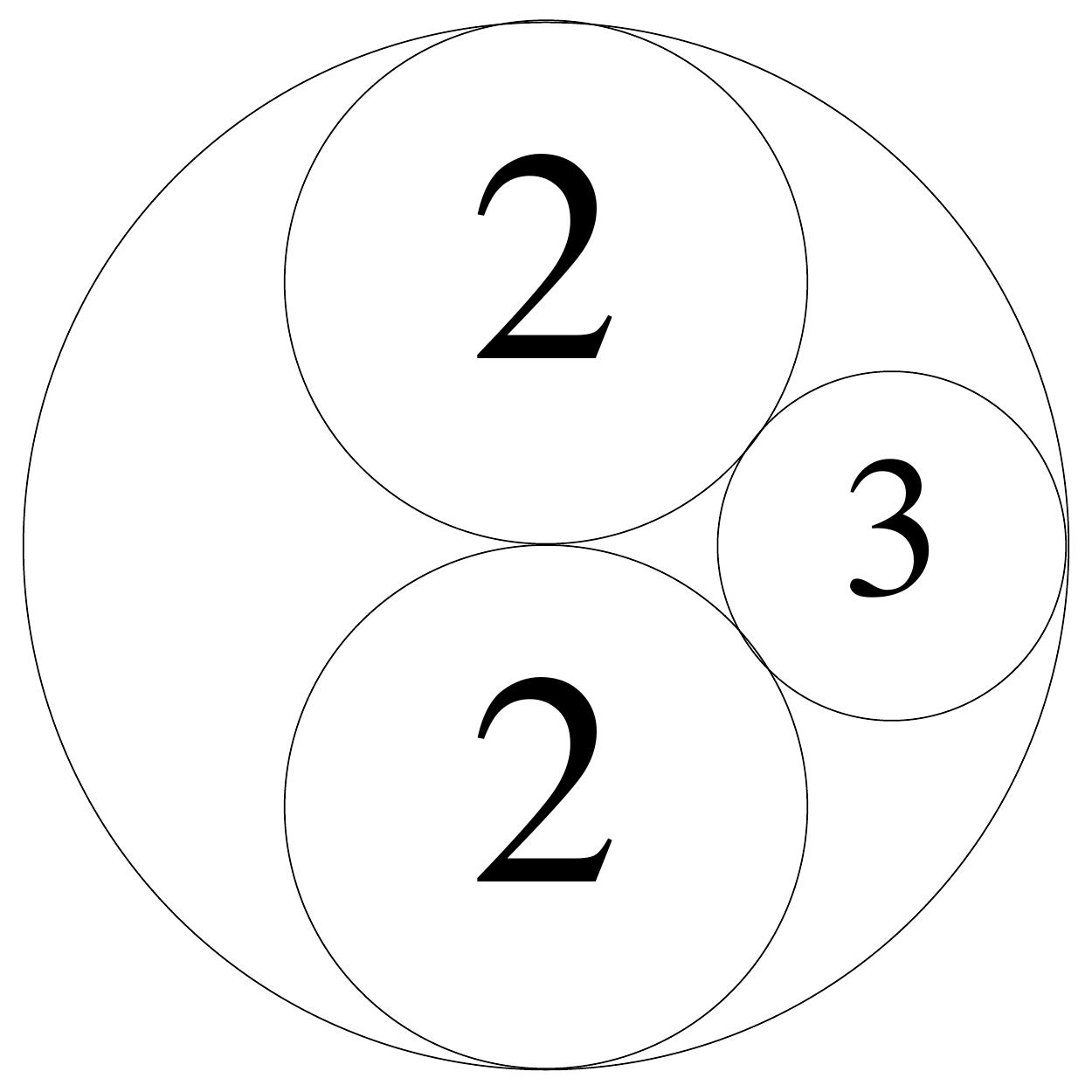}
\includegraphics [width=1.5in]{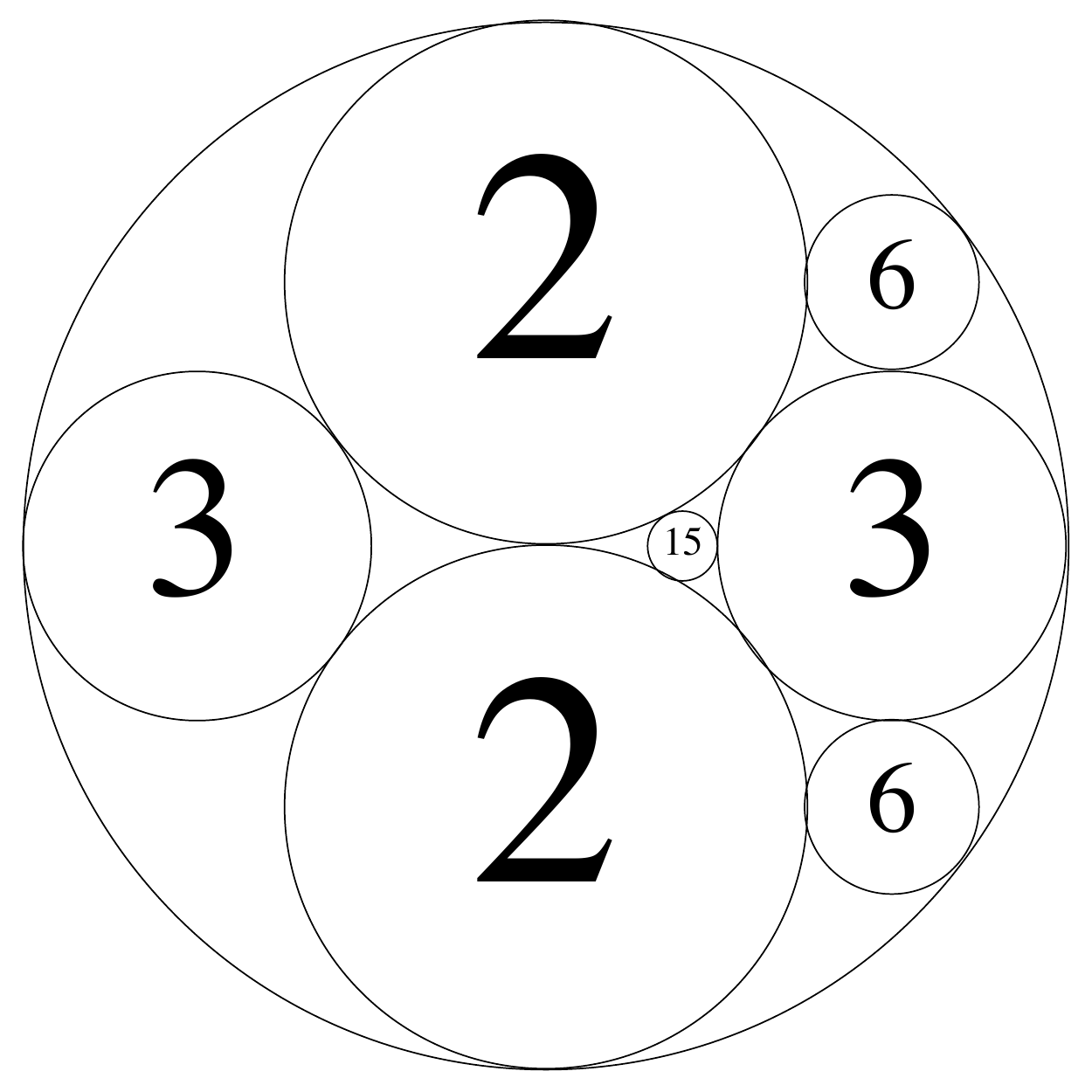}
 \includegraphics [width=1.5in]{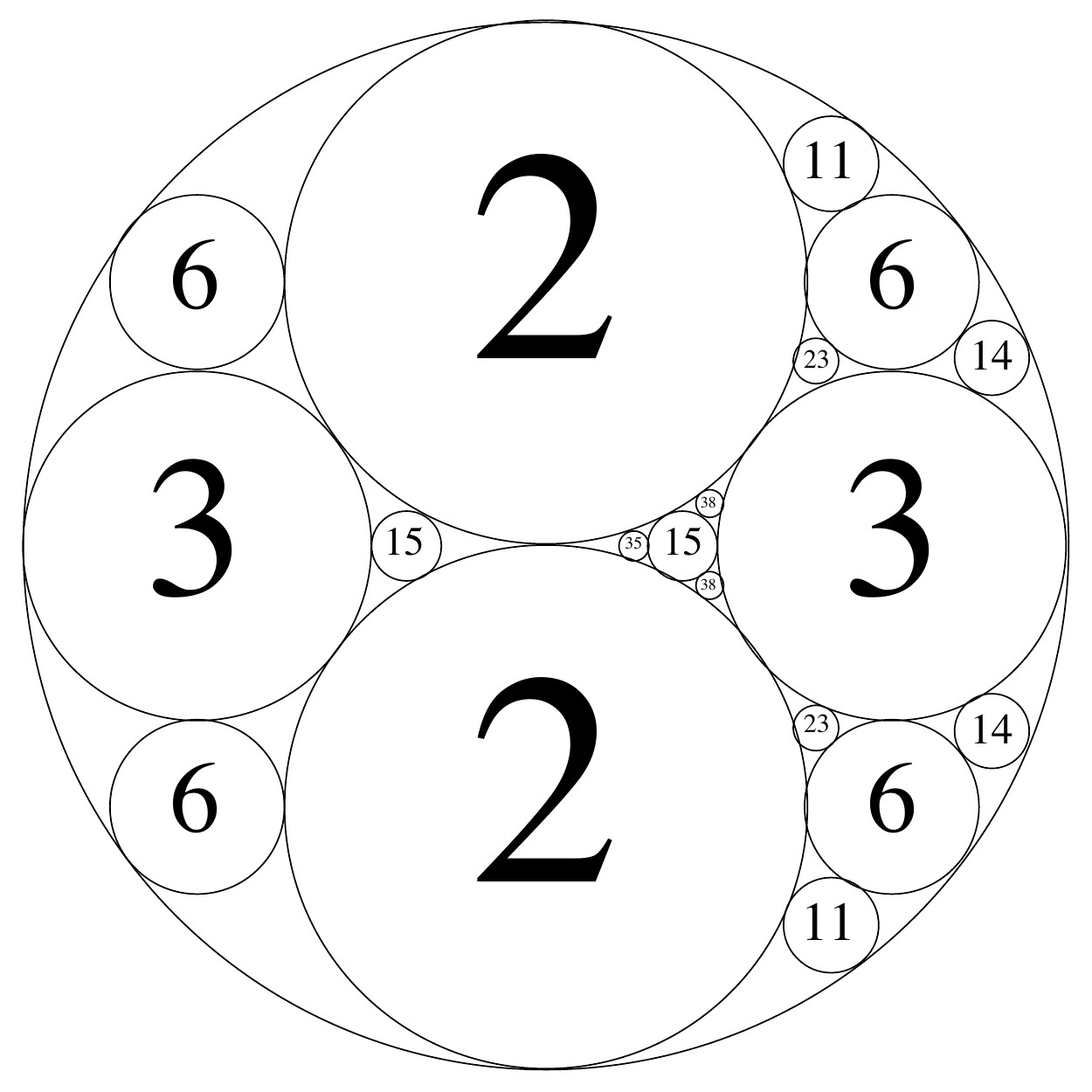}
\caption{Generations 1-3.}
\label{f1}
\end{figure}

\begin{figure}
 \includegraphics [width=2in]{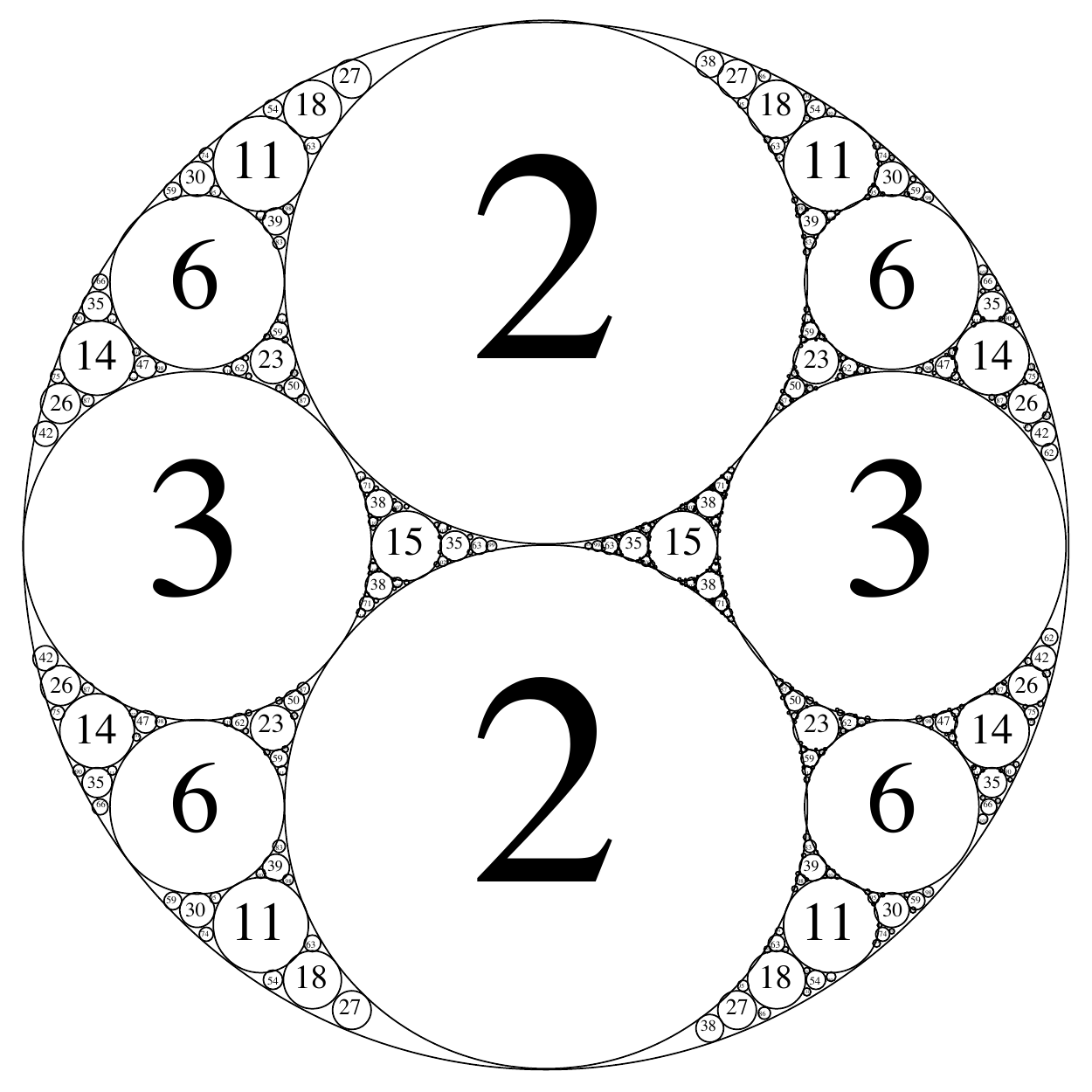}
\caption{A bounded Apollonian circle packing.}
\label{f2}
\end{figure}

\vs

\noindent{\bf Counting circles in an Apollonian packing:}
Let $\mathcal P$ be either a bounded Apollonian circle packing or
an unbounded one which is congruent to the packing in Figure \ref{InfPack}.


 For $\mathcal P$ bounded, denote by $N^{\mathcal P}
(T)$ the number of circles in $\mathcal P$ in the packing whose curvature is at most $T$,
i.e., whose radius is at least $1/T$.
For $\mathcal P$ congruent to the packing in Figure \ref{InfPack},
 one alters the definition of $N^\P (T)$ to count circles in a fixed period.


\begin{figure}
 \includegraphics[width=2in] {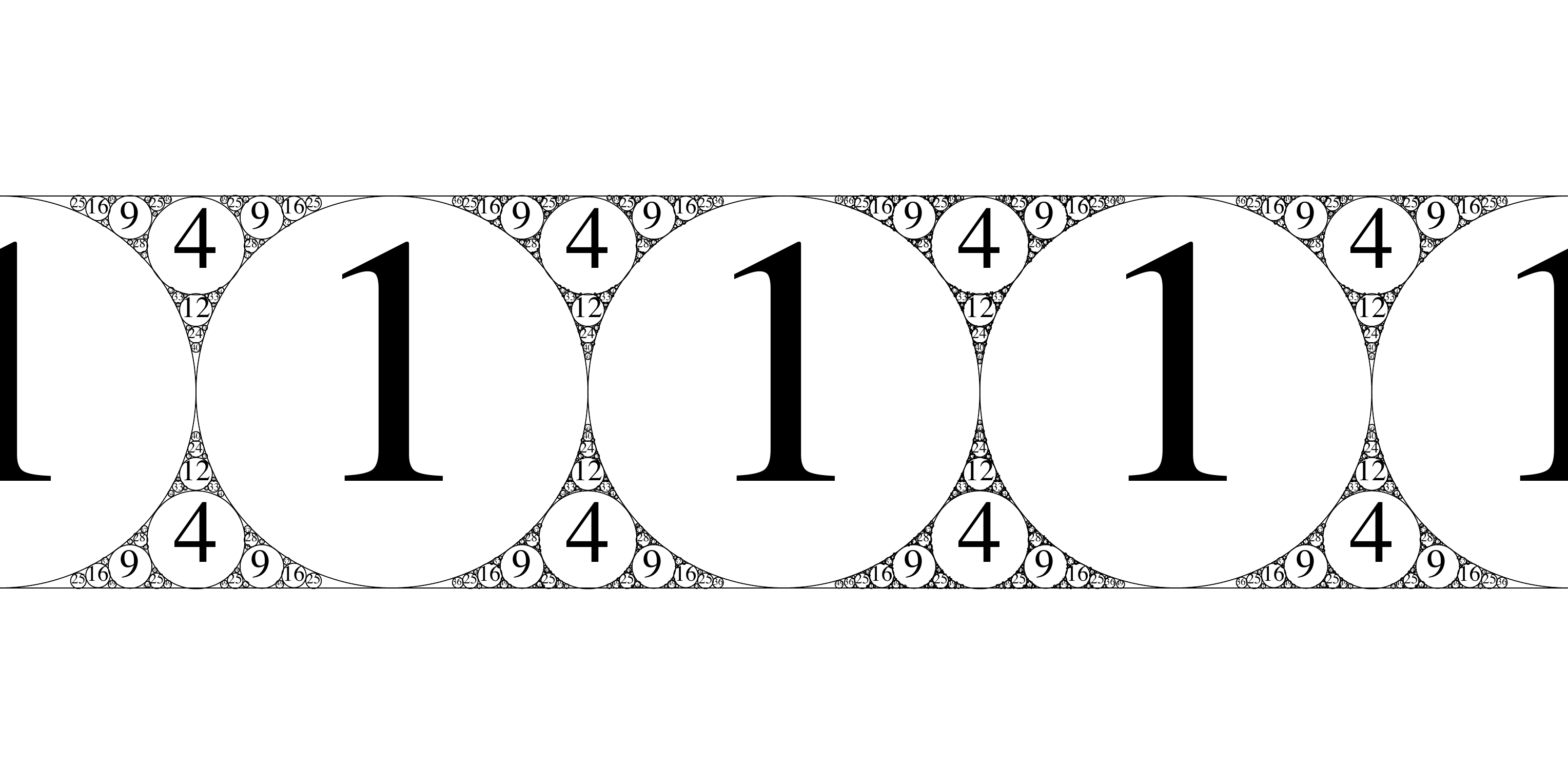}
\caption{An Apollonian circle packing bounded by two parallel lines.}
\label{InfPack}
\end{figure}

It is easy to see that $N^{\mathcal P}(T)$ is  finite for any given $T>0$.
The main goal of this paper is to obtain asymptotic formula for $N^{\mathcal P}(T)$
 as $T$ tends to infinity.
To describe our results,  recall that
the residual set of $\mathcal P$ is defined to be
 the subset of the plane remaining after the removal of all of the interiors
of circles in $\mathcal P$ (where the circles are oriented so that the interiors are disjoint).
Let $\alpha=\alpha_{\mathcal P}$ denote the Hausdorff dimension of the residual set of $\mathcal P$. As any Apollonian packing can be moved to any other by a M\"obius
transformation, $\alpha$ does not depend on $\mathcal P$.
The value of $\ga$, though not known exactly, has been
numerically computed to be $1.30568(8)$ by McMullen \cite{McMullen1998}.

Boyd \cite{Boyd1982} showed in 1982
 that
$$\lim_{T\to \infty}\frac{\log N^{\mathcal P}(T)}{\log T}=\alpha.$$
This confirmed Wilker's prediction \cite{Wilker1977} which
was based on
computer experiments.

Regarding an asymptotic formula for $N^{\mathcal P}(T)$,
it was not clear from the literature  whether one should conjecture a strictly polynomial
growth rate.
In fact, Boyd's numerical experiments led him to wonder whether ``perhaps a relationship such as $N^{\mathcal P}(T)\sim c\cdot T^\alpha (\log (T/c'))^\beta$ might be more appropriate'' (see \cite{Boyd1982} page 250).

In this paper, we show that
 $N^{\mathcal P}(T)$ has purely polynomial asymptotic growth. By $f(T)\sim g(T)$
with $T\to \infty$,
we mean that $\lim_{T\to \infty}\frac{f(T)}{g(T)}=1$
\begin{Thm}\label{main1}
Given an Apollonian circle packing $\mathcal P$ which is either bounded or congruent to Figure \ref{InfPack},
there exists $c=c(\mathcal P)>0$ such
that as $T\to \infty$,
 $$N^{\mathcal P}(T)\sim  c\cdot T^{\alpha} .$$
\end{Thm}



\begin{figure}
 \includegraphics[width=2in] {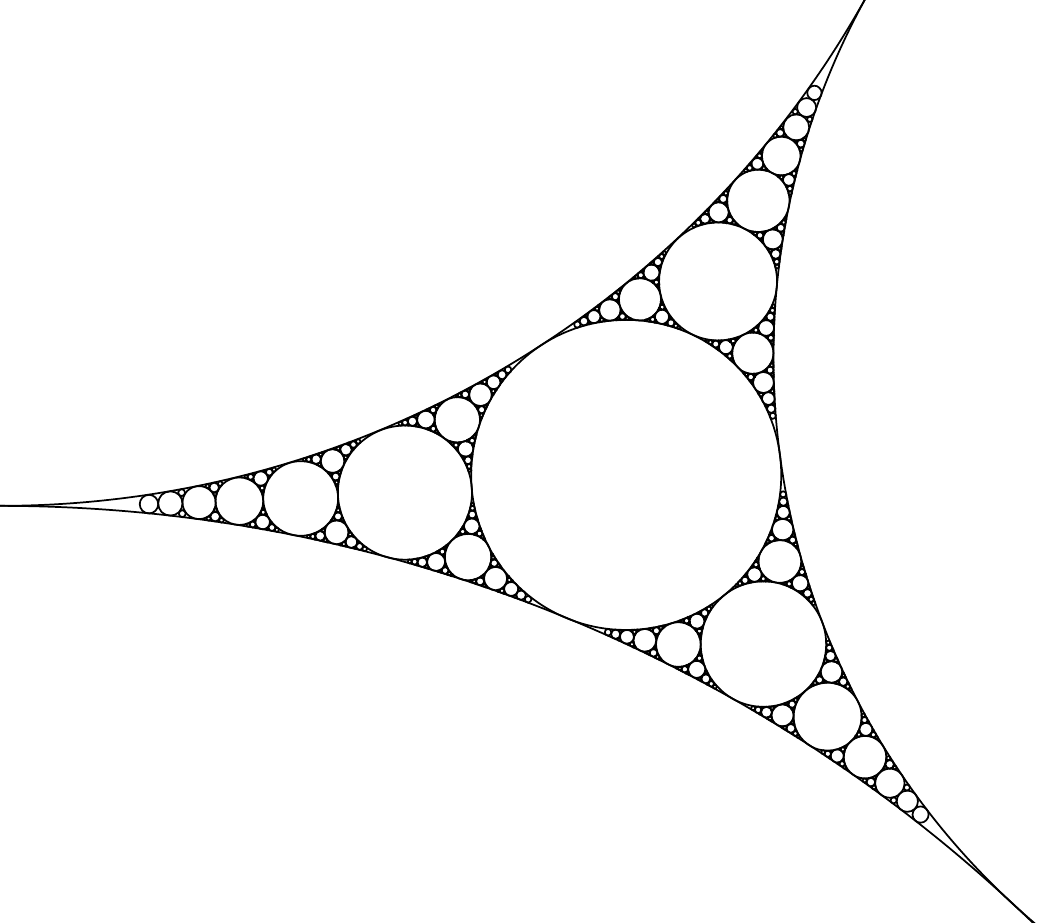}
\caption{Circles in a curvilinear triangle.}
\label{CurviLin}
\end{figure}

In \cite{Boyd1982}, Boyd actually considered the more general problem of counting those circles in a packing which are contained in a curvilinear triangle
${\mathcal R}$;
let $N^{\mathcal R}(T)$
count the circles
having curvature at most $T$
(see Fig. \ref{CurviLin}).
For this question, it does not matter whether or not the full packing is bounded; the counting function $N^{\mathcal R}(T)$ is always well-defined.  Since two such triangles are bi-Lipschitz equivalent, it follows from Theorem \ref{main1} that
there exist constants $c_1, c_2>0$ such that for all $T\gg 1$,
$$
c_1 \cdot T^\alpha \le N^{\mathcal R}(T) \le c_2 \cdot T^\alpha.
$$

Though we believe that the asymptotic formula $N^{\mathcal R}(T)\sim c \cdot T^\ga$ always holds, our techniques cannot yet establish this
in full generality.


\vs

\noindent{\bf Primes and twin primes in an integral packing:}
A quadruple $(a,b,c,d)$ of the curvatures of 
four circles in a Descartes configuration is called a {\it Descartes quadruple}.
The Descartes circle theorem (see e.g. \cite{Coxeter1968}) states that
 any Descartes quadruple $(a,b,c,d)$ satisfies
the quadratic equation\footnote{ Arguably the most elegant formulation of this theorem is the following excerpt from the poem ``The Kiss Precise'' by Nobel Laureate Sir Fredrick Soddy \cite{Soddy1936}:

{\tiny
\quad Four circles to the kissing come. / The smaller are the bender. /

\quad The bend is just the inverse of /  The distance from the center. /

\quad Though their intrigue left Euclid dumb /  There's now no need for rule of thumb. /

\quad Since zero bend's a dead straight line / And concave bends have minus sign, /

\quad The sum of the squares of all four bends / Is half the square of their sum.
}}:
\begin{equation}\label{dc}a^2+b^2+c^2+d^2=\frac{1}{2}{(a+b+c+d)}^2 .
\end{equation}


Given any three mutually tangent circles with distinct points of tangency and curvatures $a,b$ and $c$, there are exactly two
circles which are tangent to all of the given ones, having curvatures  $d$ and $d'$, say.
It easily follows from \eqref{dc} that
\be\label{dprime}
d+d'=2(a+b+c).
\ee
In particular, this shows that if a Descartes quadruple  $(a,b,c,d)$ corresponding to
the initial four circles in the packing $\mathcal P$ is integral,
then every circle in $\mathcal P$
also has integral curvature,
as first observed by Soddy in 1936 \cite{Soddy1937}. Such a packing
is called  {\it integral}. 

It is natural to inquire about the Diophantine properties of an
 integral Apollonian packing, such as how many circles in
$\mathcal P$ have prime curvatures. By rescaling, we may assume that $\mathcal P$ is {\it primitive}, that is,
the greatest common divisor of the curvatures is one.
We call a circle  {\it prime} if
its curvature is a prime number.
A pair of prime circles
which are tangent to each other will be
called {\it twin prime circles}.
It is easy to see that a primitive integral packing is either bounded or the one
pictured in Fig. \ref{InfPack}.

For $\mathcal P$ bounded,
denote by $\pi^{\mathcal P}(T)$ the
number of  prime circles in $\mathcal P$  of curvature at
most $T$, and  by $\pi^{\mathcal P}_2(T)$
 the number of twin prime circles in $\mathcal P$ of
 curvatures at most $T$.
 For $\mathcal P$ congruent to the packing in Figure \ref{InfPack},
 one alters the definition of $\pi^\P (T)$ and $\pi^{\mathcal P}_2(T)$ to count prime circles in a fixed period.

 Sarnak showed in \cite{SarnakToLagarias}
 that there are infinitely many prime and twin prime circles 
in any primitive integral
packing $\mathcal P$, and that
  $$\pi^{\mathcal P}(T)\gg\frac{T}{{(\log T)}^{3/2}} .$$

Using the recent results of
Bourgain, Gamburd and Sarnak in \cite{BourgainGamburdSarnak2008}
and \cite{BourgainGamburdSarnak2008-p} on the uniform spectral
gap property of Zariski dense subgroups of $\SL_2(\z [i])$, together with the Selberg's upper bound sieve,
we prove:
\begin{Thm}\label{thmTwo}
Given a
primitive integral
Apollonian circle packing $\mathcal P$,
\begin{enumerate}\item $ \pi^{\mathcal P}(T)\ll  \frac{T^{\alpha}}{\log T}$;\item
$ \pi^{\mathcal P}_2(T)\ll  \frac{T^{\alpha}}{(\log T)^2}. $
\end{enumerate} \end{Thm}

\noindent{\bf Remarks:}
\begin{enumerate}
\item The number of pairs of tangent circles in $\mathcal P$ of curvatures at most $T$
is equal to $3N_{\mathcal P}(T)$ up to an additive constant (see Lemma
\ref{reduction}). Therefore,
in light of Theorem \ref{main1}, the upper bounds
in Theorem \ref{thmTwo} are only off by a constant multiple from the expected
asymptotics.

\item A suitably modified version of Conjecture 1.4 in \cite{BourgainGamburdSarnak2008},
 a generalization of Schinzel's hypothesis,
implies that for some $c, c_2>0$,
$$\pi^{\mathcal P}(T) \sim c \cdot \frac{T^{\alpha}}{\log T}\quad\text{and}\quad
\pi^{\mathcal P}_2(T) \sim c_2 \cdot \frac{T^{\alpha}}{(\log T)^2}
$$
 (see the discussion in \cite[Ex D]{BourgainGamburdSarnak2008}).
The constants $c$ and $c_2$ are detailed in \cite{EFuchs}.
\end{enumerate}

\vs

\noindent{\bf Orbital counting of a Kleinian group in a cone:}
By Descartes' theorem \eqref{dc}, any Descartes quadruple $(a,b,c,d)$ lies on the cone
\break
$Q(x)=0$,
where
 $Q$ denotes the
quadratic form
$$Q(a,b,c,d)=2(a^2+b^2+c^2+d^2)-{(a+b+c+d)}^2
.
$$
In light of \eqref{dprime}, one can ``flip'' the quadruple $(a,b,c,d)$ into $(a,b,c,d')$ via left-multiplication by $S_4$, where
$$S_1=\begin{pmatrix}  -1&2&2&2 \\0&1&0&0\\
0&0&1&0\\ 0&0&0&1 \end{pmatrix},\quad
 S_2=\begin{pmatrix} 1&0&0&0\\
 2&-1& 2&2 \\
0&0&1&0\\ 0&0&0&1 \end{pmatrix}, $$
$$
 S_3=\begin{pmatrix} 1&0&0&0\\
 0&1&0&0\\ 2&2& -1 &2 \\
 0&0&0&1 \end{pmatrix}, \quad  S_4=\begin{pmatrix} 1&0&0&0\\
 0&1&0&0 \\ 0&0&1&0\\ 2&2&2&-1 \end{pmatrix}. $$
Let $\mathcal A$ denote
the so-called Apollonian group generated by these reflections, that is $\mathcal{A}=\<S_1,S_2,S_3,S_4\>$, and let
 $\O_{Q}$ be the orthogonal group preserving $Q$. One can check that $\mathcal A<\O_Q(\z)$ and that $Q$ has
signature $(3,1)$. Therefore $\mathcal{A}$ is a Kleinian group; moreover $\mathcal A$ is of infinite index in $\O_Q(\z)$.


 For a fixed packing $\mathcal P$,
there is a labeling by the Apollonian group $\mathcal A$ of all the
 (unordered) Descartes quadruples in $\mathcal P$.
Moreover the counting problems for $N^{\mathcal P}(T)$ and $N^{\mathcal P}_2(T)$ for $\mathcal P$ bounded
 can be reduced to counting elements in the orbit
 $\xi_{\mathcal P}\cdot \mathcal A^t  \subset \br^4$
of maximum norm at most $T$, where $\xi_{\mathcal P}$ is the unique root quadruple of
$\mathcal P$ (see Def. \ref{root} and Lemma \ref{reduction}). For $\mathcal P$ congruent to Figure
\ref{InfPack}, the same reduction holds with $\xi_P$ given by $(0,0, c,c)$ where $c$ is the curvature of the largest circle in $\P$.

We prove the following more general counting theorem:
Let $\iota: \PSL_2(\c)\to \SO_{F}({\br})$ be a real linear representation,
where $F$ is a real quadratic form in $4$ variables with signature $(3,1)$. Let $\G<\PSL_2(\c)$ be
a geometrically finite Kleinian group. The limit set $\Lambda(\G)$
of $\G$ is the set of accumulation points of
$\G$-orbits in the ideal boundary $ \partial _\infty(\bH^3)$ of
the hyperbolic space $\bH^3$.
 We assume that the Hausdorff dimension $\delta_\G$
of $\Lambda(\G)$  is strictly bigger
 than one.
\begin{Thm}\label{count}
Let $v_0\in \br^4$ be a non-zero vector lying in the cone $F=0$
 with a discrete orbit $ v_0 \G \subset \br^4$.
Then for any norm $\|\cdot \|$ on $\br^4$,  as $T\to \infty$,
$$\# \{v\in v_0\G: \| v\|<T\}\sim c \cdot T^{\delta_{\G}}
 $$
where $c>0$ is explicitly given  in Theorem \ref{reduction4}.
\end{Thm}

There are two main difficulties preventing existing counting methods from tackling
the above asymptotic formula. The first is that $\G$ is not required to be a lattice in $\PSL_2(\c)$
(recall that the Apollonian group $\mathcal A$ has  infinite index in $\O_Q(\z)$),
so Patterson-Sullivan theory enters in the spectral decomposition of
the hyperbolic manifold $\G\ba \bH^3$.
The second difficulty noted by Sarnak in
\cite{SarnakToLagarias} stems from the fact that the stabilizer
of $v_0$ in $\G$ may not have enough unipotent elements;
in the application to Apollonian packings, the stabilizer is
indeed either finite or a rank one abelian subgroup, whereas the stabilizer of $v_0$ in
the ambient group $G$ is a compact extension of a rank two abelian subgroup.

In the similar situation of an infinite area hyperbolic {\it surface},
 that is when $\G<\PSL_2(\br)$,
 the counting problem in a cone with respect to a Euclidean norm was solved in Kontorovich's thesis \cite{Kontorovich2007}, under the assumption
that the stabilizer
of $v_0$ in $\G$ is co-compact in
the stabilizer
of $v_0$ in $\PSL_2(\br)$.
The methods developed here are quite different. Our approach to the counting problem is via the equidistribution of expanding closed horospheres on hyperbolic 3-manifolds (see the next subsection for
more detailed discussion). Our proof of the equidistribution works for hyperbolic surfaces
as well, and in particular solves the counting problem in
 \cite{Kontorovich2007} for any norm and without assumptions on the stabilizer.
In \cite{KontorovichOh2008},
we apply the methods of this paper to the problem of thin orbits of Pythagorean triples having few prime factors.




\vs
\noindent{\bf Equidistribution of expanding horospheres in
hyperbolic $3$-manifolds:}  Let $\G$ be a geometrically finite
 torsion-free discrete subgroup of $\PSL_2(\c)$.
The main ingredient of our proof of Theorem \ref{count}
 is the equidistribution of expanding  closed
 horospheres in $\G\ba \bH^3$.

The group $G=\PSL_2(\c)$ is the group of orientation preserving
isometries of the hyperbolic space $\bH^3=\{(x_1, x_2, y): y>0\}$.  The invariant measure on $\bH^3$
for the action of $G$
and the Laplace operator are  given
respectively by $$\frac{dx_1dx_2 dy}{y^3}\quad\text{and}\quad \Delta=-y^2\left(
\frac{\partial^2}{\partial y^2}+
\frac{\partial^2}{\partial {x_1^2}}+
\frac{\partial^2}{\partial {x_2^2}} \right)+y\frac{\partial}{\partial y}.$$

Set
\begin{align}\label{nak} N&=\{n_x:=\begin{pmatrix} 1& x \\
 0& 1\end{pmatrix}:x\in \c \}, \\
A&=\{a_y:=\begin{pmatrix} \sqrt y & 0\\ 0& \sqrt y^{-1}\end{pmatrix}: y>0\} ,\notag \\
 K&=\{g\in G: \bar g^t g=I\} \quad\text{and} \notag \\
  M&=\{\begin{pmatrix} e^{i\theta} & 0\\ 0& e^{-i\theta}\end{pmatrix}: \theta\in \br\} \notag.\end{align}

By the Iwasawa decomposition $G=NAK$,
 any element $g\in G$ can be written uniquely as $g=n_x a_y  k$
  with  $n_x\in N$, $a_y\in A$, and $k\in K$.
Via the map
$$n_x a_y (0,0, 1)=(x, y) ,$$
the hyperbolic space $\bH^3$ and
its unit tangent bundle $\op{T}^1(\bH^3)$
can be identified with the quotients $G/K$ and $G/M$ respectively.

 Denoting
by $[u]$ the image of $u\in G$ under the quotient map
$G\to G/M$, the
 horospheres
correspond to $N$-leaves $[u]N =[uN]$ in $G/M$;
note this is well-defined as $M$ normalizes $N$.
For a closed horosphere $\G\ba \G [u] N$,
the translates $\G\ba \G [u] N a_y$ represent closed horospheres
$\G \ba \G [u a_y] N$ which are expanding as $y\to 0$.

In the case of a hyperbolic surface $\Gamma\ba \bH^2$ of finite area,
it is a theorem of Sarnak's \cite{Sarnak1981} that
such long horocycle flows are equidistributed with respect to the Haar measure.
The
equidistribution of expanding horospheres
 for any {\it finite volume} hyperbolic manifold can
also be proved using the mixing of geodesic flows. This approach appears already
in Margulis's 1970 thesis \cite{Margulisthesis}; see also \cite{EskinMcMullen1993}.





In what follows we describe our equidistribution result for expanding closed horospheres
on any geometrically finite hyperbolic $3$ manifold.

Assuming $\delta_\G>1$, Sullivan \cite{Sullivan1979} showed, generalizing the work of Patterson \cite{Patterson1976},
that there exists a unique positive eigenfunction $\phi_0$
of the Laplacian operator $\Delta$ on $\G\ba \mathbb H^3$ of
lowest eigenvalue $\delta_\G(2-\delta_\G)$ and of unit $L^2$-norm, that is,
 $\int_{\G\ba \bH^3} \phi_0(x_1, x_2,y )^2 \frac{1}{y^3} {dx_1dx_2dy}=1$.
Moreover the base eigenvalue-value $\delta_\G(2-\delta_\G)$ is
isolated in the $L^2$-spectrum of $\Delta$ by Lax-Phillips \cite{LaxPhillips}.

Note that the closed leaf  $\G \ba \G  N$ inside $\op{T}^1(\G\ba \bH^3)$
is an embedding of one of the following: a complex plane, a cylinder, or a torus.
 As $\phi_0>0$, it is a priori not clear whether the integral
$$\phi_0^N(a_y):=\int_{n_x\in (N\cap \G )\ba  N}\phi_0(x,y) dx$$ converges.
We show that for any $y>0$,
 the integral $\phi_0^N(a_y)$
does converge and is of the form
$$\phi_0^N(a_y)=c_{\phi_0} y^{2-\delta_\G} + d_{\phi_0} y^{\delta_\G} $$
for some constants $c_{\phi_0}>0$ and $d_{\phi_0}\ge 0$.
By $f(y)\sim g(y)$ with $y\to 0$, we mean that $\lim_{y\to 0}\frac{f(y)}{g(y)}=1$. The measure
$dn$ on $N$ is Lebesgue: $dn_x=dx_1dx_2$.
\begin{Thm}\label{lh} Let $\G< G$ be a geometrically finite torsion-free discrete subgroup with
$\delta_\G>1$ and let $\G\ba \G N$ be closed.
 There exists $\e>0$ such that for any $\psi\in C_c^\infty (\G\ba G)^K=C_c(\G\ba \bH^3)$,
and for all small $y>0$
$$\int_{(N\cap \G) \ba N } \psi(\G \ba \G n a_y ) \; dn = c_{\phi_0}\cdot
 \langle
 \psi,\phi_0\rangle_{L^2(\G\ba \bH^3)} \cdot y^{2-\delta_\G } (1+O(y^\e)) $$
where the implied constant depends only on the Sobolev norm of $\psi$.
\end{Thm}

Thus as $y\to 0$, the integral
of any function $\psi \in C_c(\G\ba G)^K$  along
the orthogonal translate $\G\ba\G Na_y$ converges to $0$ with the speed of order $y^{2-\delta_{\G}}$. It also follows
that for $\psi\in C_c(\G\ba G)^K$, as $y\to 0$,
$$\int_{(N\cap \G) \ba N } \psi(\G \ba \G n a_y) \; dn \sim
 \la \psi, \phi_0\ra \cdot \int_{ (N\cap \G) \ba N } \phi_0(\G \ba \G n a_y) \; dn. $$

We denote by $\tilde \Omega_\G$ the set of vectors $(p, \vec v)$ in
the unit tangent bundle $\op{T}^1( \bH^3)$
such that the end point of the geodesic ray tangent to $\vec v$
belongs to the limit set $\Lambda(\G)$ and by
$\hat \Omega_\G$ its image under the projection of
$\op{T}^1(\bH^3)$ to $\op{T}^1(\G\ba \bH^3)$.

Roblin \cite{Roblin2003},
 generalizing the work of Burger \cite{Burger1990}, showed that, up to a constant multiple,
there exists a unique Radon measure
 $\hat \mu$  on $\op{T}^1(\G\ba \bH^3)$ invariant for the horospherical foliations
which is supported on $\hat \Omega_\G$ and gives zero measure
to all closed horospheres.

In the appendix \ref{app} written jointly
by Shah and the second named author, the following theorem is deduced from Theorem \ref{lh},
based on the aforementioned measure classification of
 Burger and Roblin.
In view of the isomorphism $\op{T}^1(\G\ba \bH^3)=\G\ba G /M,$
the following theorem
shows that the orthogonal translations of closed horospheres in the expanding direction
are equidistributed in $\op{T}^1(\G\ba \bH^3)$ with respect to the Burger-Roblin
measure $\hat \mu$.

\begin{Thm}\label{lhtwo}  For any $\psi\in C_c( \G\ba G)^M$, as $y\to 0$,
$$
\int_{ (N\cap \G) \ba N } \psi(\G \ba \G n a_y) \; dn \sim
 c_{\phi_0}\cdot  y^{2-\delta_\G} \cdot \hat \mu(\psi)  $$
where $\hat \mu$ is normalized so that $\hat \mu (\phi_0)=1$.
\end{Thm}

Theorem \ref{lhtwo} was proved by Roblin \cite{Roblin2003} when $(N\cap \G)\ba N$ is
compact with a different interpretation of the constant $c_{\phi_0}$. His proof
does not yield an effective version as in Theorem \ref{lh} but works for any
$\delta_\G>0$.

We also remark that the quotient type equidistribution results for {\it non-closed} horocycles
for geometrically finite surfaces were established by Schapira \cite{Schapira2005}.

We conclude the introduction by giving a brief outline of
the proof of Theorem \ref{lh}. By Dal'bo (Theorem \ref{dalbo}),
we have the following classification of
closed horospheres in terms of its base point in the
boundary $\partial_\infty(\bH^3)$: $\G\ba \G N$ is closed if and only if either $\infty\notin \Lambda(\G)$ or
$\infty$ is a bounded parabolic fixed point (see Def. \ref{boundedparabolic}).
This classification is used repeatedly in our analysis establishing the following facts:
\begin{enumerate}
\item For any bounded subset $B \subset (N\cap \G)\ba N$ which
 properly covers $(N\cap \G)\ba (\Lambda(\G)-\{\infty\})$ (cf. Def. \ref{proper}),
$$\int_{(N\cap \G) \ba N } \phi_0( n a_y) \; dn =
\int_{B} \phi_0(n a_y) \, dn  + O(y^{\delta_\G}).$$
\item  Denoting by $\rho_{B, \e}\in C_c(\G\ba G)$
 the $\e$-approximation of $B$ in the transversal direction,
$$\int_{B} \phi_0( n a_y) \; dn= \la a_y \phi_0, \rho_{B,\e} \ra + O(\e y^{2-\delta_\G}) +O(y^{\delta_\G}).$$
\item For $\psi\in C_c(\G\ba \bH^3)$,
there exists a compact subset $B=B(\text{supp}(\psi))\subset (N\cap \G)\ba N$
such that for all $0<y<1$
$$\int_{(N\cap \G)\ba N } \psi (na_y)\, dn =
\int_{B} \psi(n a_y) \, dn .$$
\end{enumerate}

These facts allow us to focus on
the integral of $\psi$ over a compact region, say, $B$, of
 $(N\cap \G)\ba N$ instead of the whole space.
In approximating $\int_{B}\psi(na_y)\, dn$ with
$\la a_y \psi, \rho_{B,\e}\ra$, the usual argument based on the contracting property
along the stable horospheres is not sufficient:
the error terms combine with those coming from the spectral gap to overtake the main term!
We develop a recursive argument which improves the error upon each iteration, and halts in finite time, once the main term is dominant. Using the spectral theory
of $L^2(\G\ba G)$ along with the assumption $\delta_\G>1$, we
get a control of the main term of $\la a_y \psi, \rho_{B,\e}\ra$ as $\la \psi, \phi_0\ra \cdot
\la a_y \phi_0, \rho_{B,\e}\ra $.

For the application to Apollonian packings, we need to consider the max norm, which necessitates
the extension of our argument  to the unit tangent bundle, that is, the deduction of Theorem
\ref{lhtwo} from \ref{lh} as done in Appendix \ref{app}.

 We finally remark that the power savings error term in (1) of Theorem \ref{lh} is crucial to prove Theorem \ref{thmTwo}.

\vs

After this paper was submitted, the asymptotic formula for counting circles
in a curvilinear triangle of any Apollonian packing
has been obtained in \cite{OhShahcircle}.
See also \cite{OhShahGFH} for similar counting results for hyperbolic and spherical
Apollonian circle packings. We also refer to \cite{OhICM} for a survey on recent
progress on counting circles.

\vs
\noindent{\bf Acknowledgments.}
We are grateful to Peter Sarnak for introducing us to this problem and for helpful discussions.
 We also thank Yves Benoist, Jeff Brock and Curt McMullen
for useful conversations.

\section{Reduction to orbital counting}\label{red}
\subsection{Apollonian group}\label{red1}
In a quadruple of mutually tangent circles, the curvatures $a,b,c,d$
satisfy the Descartes equation:
$$2(a^2+b^2+c^2+d^2)={(a+b+c+d)}^2$$
as observed by Descartes in 1643 (see \cite{Coxeter1968} for a proof).

Any quadruple $(a,b,c,d)$ satisfying this equation is called a
Descartes quadruple. A set of four mutually tangent circles with
disjoint interiors is called a Descartes configuration.


We denote by $Q$ the Descartes quadratic form given by
$$Q(a,b,c,d)=a^2+b^2+c^2+d^2-\frac{1}{2}{(a+b+c+d)}^2 .$$

Hence $v=(a,b,c,d)$ is a Descartes quadruple if and only if
$ Q (v)=0$.
The orthogonal group corresponding to $Q$ is
given by  $$\O_Q=\{g\in \GL_4 : Q(vg^t)=Q(v) \text{ for all $v\in \br^4$} \} .$$
One
 can easily check that
the Apollonian group $\mathcal A:=\la S_1, S_2, S_3, S_4\ra $ defined in the introduction
 is a subgroup of $\O_Q(\z):=\O_Q\cap \GL_4(\z)$.


\begin{Def}\label{dnp} {\rm
\begin{enumerate}
 \item For $\mathcal P$ bounded, denote by $N^{\mathcal P}
(T)$ the number of circles in $\mathcal P$ in the packing whose curvature is at most $T$,
i.e., whose radius is at least $1/T$.
Denote by $N^{\mathcal P}_2(T)$ the number of pairs of tangent circles in $\mathcal P$
of curvatures at most $T$.

\item For $\mathcal P$ congruent to the packing in Figure \ref{InfPack},
  $N^\P (T)$ denotes the number of circles between two largest
tangent circles including the lines and the largest circles.
Similarly, $N^{\mathcal P}_2(T)$ denotes the number of unordered pairs of tangent circles between two
largest tangent circles including the pairs containing  the lines and the largest circles.

\item For $\mathcal P$ bounded, denote by $\pi^{\mathcal P}(T)$ the
number of  prime circles in $\mathcal P$  of curvature at
most $T$.  Denote by $\pi^{\mathcal P}_2(T)$
 the number of twin prime circles in $\mathcal P$ of
 curvatures at most $T$.

\item For $\mathcal P$ congruent to the packing in Figure \ref{InfPack},
 one alters the definition of $\pi^\P (T)$ and $\pi^{\mathcal P}_2(T)$ to count prime circles in a fixed period.
\end{enumerate}

} \end{Def}

We will interpret $N^{\mathcal P}(T)$ and $N_2^{\mathcal P}(T)$ as orbital counting
functions on $\xi \mathcal A^t$ for a carefully chosen
Descartes quadruple $\xi$ of $\mathcal P$.


\begin{Def}\label{root}{\rm  A Descartes quadruple $v=(a,b,c,d)$ with $a+b+c+d>0$ is
a root quadruple if $a\le 0\le b\le c\le d$ and $a+b+c\ge d$.}\end{Def}

If $\mathcal P$ is bounded,
Theorem 3.2 in
\cite{GrahamLagariasMallowsWilksYanI-n} shows that $\mathcal P$ contains a unique Descartes root quadruple $\xi:=(a,b,c,d)$
with $a<0$.

\begin{Thm}\label{gm} \cite[Thm 3.3]{GrahamLagariasMallowsWilksYanI-n}
 The set of curvatures occurring in $\mathcal P$, counted with multiplicity,
consists of the four entries in $\xi$, together with the largest entry in each vector $\xi \gamma^t$ as $\gamma$ runs over all non-identity elements of the Apollonian group $\mathcal A$.
\end{Thm}
In \cite{GrahamLagariasMallowsWilksYanI-n}, this is stated only for
an integral Apollonian packing, but the same proof works for any bounded packing.
Let $w^{(n)}$ be a non-returning walk away from the root quadruple $\xi$ along the Apollonian group, i.e., $w^{(n)}=\xi S_{i_1}^t\cdots S_{i_n}^t$ with $S_{i_k}\ne S_{i_{k+1}}$, $1\le k\le n-1$.
Then the key observation in the proof of above theorem is that
$w^{(n)}$ is obtained from $w^{(n-1)}$ by changing one entry, and moreover the new entry
inserted is always the largest entry in the new vector.

This theorem yields that for $T\gg 1$,
$$N^{\mathcal P}(T)=\# \{\gamma\in \mathcal A: \|\xi \gamma^t\|_{\max} <T\} +3 .$$

Consider the repeated generations of $\mathcal P$ with initial $4$ circles given by the root
quadruple. Then a geometric version of Theorem \ref{gm} is that for $n\ge 1$,
each reduced word $\gamma=S_{i_n}\cdots S_{i_1}$ of length $n$
corresponds to exactly one new
 circle, say $C_\gamma$, added at the $n$-th generation and the curvature of $C_\gamma$ is the maximum
among the entries of the quadruple $\xi \gamma^t$ (cf. \cite[Section 4]{GrahamLagariasMallowsWilksYanI}).
Thus the correspondence $\phi:\gamma \mapsto C_\gamma $ establishes a bijection between
the set of all non-identity elements of $\mathcal A$ and the set of all circles not in the zeroth generation. Hence
 the set $\{C_{\gamma} | \gamma \ne e ,\|\xi \gamma^t\|<T\}$ gives
 all circles (excluding those $4$ initial circles) of curvature at most $T$.

For each $\gamma \ne e$ in $\mathcal A$,
set $$\phi_2(\gamma)=\{ \{C_\gamma, C_\gamma(1) \}, \{C_\gamma, C_\gamma (2) \}, \{C_\gamma, C_\gamma(3) \} \}$$
where $C_\gamma(i)$, $i=1,2,3$, are the three circles corresponding to the quadruple $\xi \gamma^t$ besides $C_\gamma$. Noting that each $(C_\gamma, C_\gamma(i))$ gives a pair of tangent circles,
we claim that every pair of tangent circles arises as one of the triples in the image of $\phi_2$,
provided one of the circles in the pair does not come from the zeroth stage.
If $C$ and $D$ form such a pair and are not
from the initial stage, then they are different generations and it is
obvious that no two circles in the same generation touch each other.
If, say, $D$ is generated earlier than $C$,
 then for the element $\gamma\in \mathcal A$
giving $C=C_\gamma$, which is necessarily a non-identity element,
$D$ must be one of $C_\gamma (i)$'s. This is because it is clear from
the construction of the packing that every circle is tangent only {\it three} circles from previous generations.

Therefore $\phi_2$ yields a one to three correspondence from $\mathcal A\setminus\{e\}$ to
the set of all unordered pairs of tangent circles in $\mathcal P$, at least one circle of whose pair does
not correspond to the root quadruple.
Since there are $6$ pairs arising from the initial $4$ circles, we deduce
that
for $T\gg 1$,
\begin{align}\label{ntwo} N^{\mathcal P}_2(T) &=  3\cdot \# \{\gamma\in \mathcal A: \|\xi \gamma^t\|_{\max} <T\}
+3 . \end{align}

The above argument establishing \eqref{ntwo} was kindly explained to us by Peter Sarnak.



If $\mathcal P$ lies between two parallel lines, that is, congruent to Figure \ref{InfPack},
there exists the unique $c>0$ such that
$\mathcal P$ contains a Descartes quadruple $\xi:=(0,0,c,c)$.

In this case,
the stabilizer of $\xi$ in $\mathcal A ^t$ is
 generated by two reflections $S_3^t$ and $S_4^t$.
One can directly verify that
for all $T\gg 1$,
  $$N^{\mathcal P}(T)= { \#\{v \in \xi \mathcal A^t  :\|v \|_{\max} <T\}} +3 ;$$
and $$N^{\mathcal P}_2(T) =  3\cdot \# \{ v \in \xi\mathcal A^t: \|v \|_{\max} <T\}+3.$$


\begin{Lem}\label{reduction}
 Let $\mathcal P$ be either bounded or congruent to Fig. \ref{InfPack}. \begin{enumerate}
\item  For all $T\gg 1$,
$$N^{\mathcal P}(T)=\begin{cases} \# \op{Stab}_{\mathcal A^t}(\xi)\cdot \#\{v \in \xi \mathcal A^t  :
 \|v \|_{\max} <T\} + 3 &\text{for $\mathcal P$ bounded}\\
{ \#\{v \in \xi \mathcal A^t  :\|v \|_{\max} <T\}} +3&\text{otherwise} .\end{cases}
$$
\item  For all $T\gg 1$, $$ N^{\mathcal P}_2(T) = 3 N^{\mathcal P}(T) -6. $$
 \item The orbit $\xi \cA^t$ is discrete in $\br^4$.
\item For all $T\gg 1$,
$$\pi^{\mathcal P}(T) \ll \sum_{i=1}^4 \#\{v=(v_1,v_2,v_3,v_4) \in \xi \mathcal A^t  :
 \|v \|_{\max} <T, v_i\text{ is prime}\} .$$

\item For all $T\gg 1$,
 $$\pi_2^{\mathcal P}(T) \ll \sum_{1\le i\ne j\le 4} \#\{v \in \xi \mathcal A^t  :
 \|v \|_{\max} <T, v_i, v_j \text{ are primes}\}.$$

\end{enumerate}
\end{Lem}
\begin{proof}
The first two claims follow immediately from
the discussion above, noting that the stabilizer of $\xi$ in $\mathcal A^t$ is finite for $\mathcal P$ bounded.
The third claim follows from the fact that
$N^{\mathcal P}(T)<\infty $ for any $T>0$.
For claims (4) and (5), note that for $\mathcal P$ bounded,
\begin{align*} \pi^\mathcal P(T)& \le
3+\#\{\gamma\in \mathcal A: \|\xi \gamma^t\|_{\max}  \text{ is prime} <T\} \\
&\ll \sum_{i=1}^4 \#\{v=(v_1,v_2,v_3,v_4) \in \xi \mathcal A^t  :
 \|v \|_{\max} <T, v_i\text{ is prime}\},
\end{align*}
and
\begin{align*} \pi_2^\mathcal P(T)& \le
5+\#\{\gamma\in \mathcal A: \|\xi \gamma^t\|_{\max}
\text{ is prime} <T, \text{one more entry of $\xi\gamma^t$ is prime}\} \\
&\ll \sum_{i=1}^4 \sum_{j\ne i} \#\{v=(v_1,v_2,v_3,v_4) \in \xi \mathcal A^t  :
 \|v \|_{\max} <T, v_i, v_j \text{ are primes}\}.
\end{align*}
The claim (4) and (5) for $\mathcal P$ congruent to Fig. \ref{InfPack} can be shown similarly.

\end{proof}

We remark that there are bounded packings which are not multiples of integral packings: for instance, $\xi=(3 - 2\sqrt3, 1, 1, 1)$ is a Descartes root quadruple which defines a bounded Apollonian packing. This is obvious from the viewpoint of geometry, but it is not at all clear {\it a priori} that the orbit $\xi\mathcal A^t$ should be discrete.
 Note also that there are other unbounded packings:
by applying a suitably chosen M\"obius transformation to a given packing, one can arrive at a packing which spreads uncontrollably to the entire plane,
 or one which is fenced off along one side by a single line.






\subsection{The residual set}
We consider the upper half-space model for the hyperbolic space: $$\mathbb H^3=
\{(x_1,x_2 ,y)\in \br^3:  y>0\}$$ with metric given by $\frac{dx_1^2+dx_2^2+ dy^2}{y^3}$.
A discrete subgroup of $\Isom(\bH^3)$ is called a Kleinian group.

The ideal boundary $\partial_\infty (\bH^3)$ of $\bH^3$
can be identified with the set of geodesic rays emanating from a fixed point $x_0\in \bH^3$.
The topology on  $\partial_\infty (\bH^3)$ is defined via the angles between
corresponding rays: two geodesic rays are close if and only if the angle between
 the corresponding rays is small.

In the upper half-space model, we can identify
$\partial_\infty (\bH^3)$ with the extended complex plane
 $\{(x_1, x_2, 0)\}\cup\{\infty\} =\c\cup\{\infty\}$, which is homeomorphic
to the sphere $\mathbb S^2$.
The space $\bH^3$ has the natural compactification $\bH^3\cup \partial_\infty(\bH^3)$
(cf. \cite[3.2]{Kapovichbook}).

\begin{Def} {\rm For a Kleinian group $\G$,
the limit set $\Lambda(\G)$ of $\G$ consists of limit points of an orbit $\G z$, $z\in \bH^3$ in the
ideal boundary $\hat \c =\c \cup \{\infty\}$. We denote by $\delta_{\G}$ the Hausdorff dimension of
$\Lambda(\G)$.}
\end{Def}
By Sullivan, $\delta_\G$ is equal to the critical exponent of $\G$ for geometrically
finite $\G$.
In this subsection,
we realize the action of $\mathcal A$ on Descartes quadruples arising from $\mathcal P$
as the action of
a subgroup, $G_\mathcal A(\mathcal P)$, of the M\"obius transformations on 
 $\hat \c$
in a way that the residual set of $\mathcal P$ coincides with the limit set of $G_{\mathcal A}(\mathcal P)$.
The residual set $\Lambda(\mathcal P)$ is defined to be the closure of all the circles in $\mathcal P$,
or equivalently, the complement in $\hat \c$ of the interiors
 of all circles in the packing $\mathcal P$ (where the circles are oriented so that the interiors are
disjoint).

An oriented Descartes configuration is a Descartes configuration in which the orientations of the circles are
compatible in the sense that either the interiors of all four oriented circles are disjoint or the
interiors are disjoint when all the orientations are reversed.
Given an ordered configuration $\mathcal D$ of four oriented circles
$(C_1,C_2,C_3,C_4)$ with curvatures $(b_1, b_2, b_3, b_4)$ and centers
$\{(x_i, y_i): 1\le i\le 4\}$, set
$$W_{\mathcal D}:=\begin{pmatrix} \bar b_1 & b_1 & b_1x_1& b_1y_1\\
 \bar b_2 & b_2 & b_2x_2& b_2y_2\\
  \bar b_3 & b_3 & b_3x_3& b_3y_3\\
 \bar b_4 & b_4 & b_4x_4& b_4y_4
\end{pmatrix},$$
where $\bar b_i$ is the curvature of the circle which is the reflection of $C_i$ through the unit circle centered at the origin, i.e., $\bar b_i=b_i (x_i^2+y_i^2) - b_i^{-1}$ if $b_i\ne 0$.
If one of the circles, say $C_i$, is a line, we interpret the center $(x_i,y_i)$ as the outward unit normal vector and set $b_i=\bar b_i=0$.

Then by \cite[Thm. 3.2]{GrahamLagariasMallowsWilksYanI}, for any ordered and oriented Descartes configuration $\mathcal D$,
 the map $g\mapsto W_{\mathcal D}^{-1} g W_{\mathcal D}$ gives an isomorphism
$$\psi_{\mathcal D}:\O_Q(\br)\to \O_{Q_W}(\br)$$ where
$Q_W$ is the Wilker quadratic form:
$$Q_W=\begin{pmatrix} 0 & -4 &0&0\\
-4 & 0&0 &0\\ 0&0&2&0\\ 0&0&0&2\end{pmatrix} .$$

On the other hand, if $\text{M\"ob}(2)$ denotes the group of M\"obius transformations and $\op{GM}^*(2):=\text{M\"ob}(2)\times \{\pm I\}$ denotes the extended
M\"obius group, we have the following by Graham et al:
\begin{Thm} \cite[Thm. 7.2]{GrahamLagariasMallowsWilksYanI} \label{wd}
There exists a unique isomorphism $$\pi: \op{GM}^*(2) \to \O_{Q_W}(\br)$$ such  that
for
any ordered and oriented Descartes configuration $\mathcal D$,
\begin{enumerate}
\item $  W_{\gamma (\mathcal D)} =W_{\mathcal D} \pi(\gamma )^{-1}
\quad\text{for any $\gamma\in\;\; $\rm{M\"ob}(2)} .$
\item $W_{-\mathcal D}=-W_{\mathcal D}$.
\end{enumerate}
\end{Thm}

Let $\mathcal D_0=(C_1, C_2, C_3, C_4)$ denote the
ordered and oriented Descartes configuration corresponding to the root quadruple
$\xi$ of $\mathcal P$. We obtain the following isomorphism:
 $$\Phi_{\mathcal D_0}:=\pi^{-1}\circ \psi_{\mathcal D_0}:\O_{Q}(\br) \to \op{GM}^*(2).$$

 Denote by $\mathfrak s_i:=\mathfrak s_i(\mathcal D_0)$
 the M\"obius transformation
  given by the inversion in the circle, say, $\hat C_i$,
 determined by the three intersection points
of the circles $C_j$, $j\ne i$.
Figure \ref{picDual} depicts the root quadruple $(C_1,\dots,C_4)$ as solid-lined circles and the corresponding dual quadruple $(\hat C_1,\dots,\hat C_4)$ as dotted-lined circles.
\begin{figure}
 \includegraphics [width=2in]{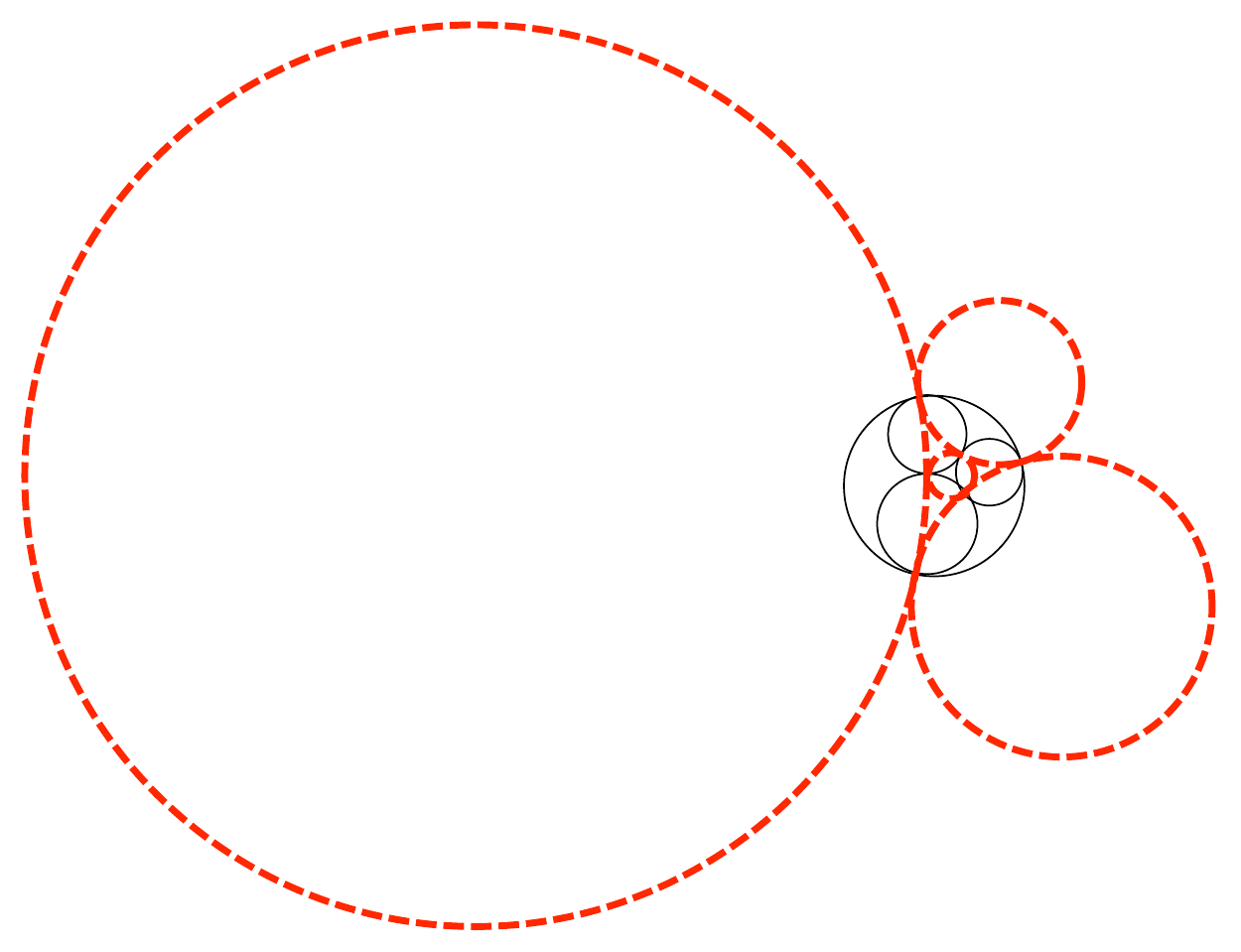}
\caption[Action of $\mathcal A$ as M\"obius transformations]{Action of $\mathcal A$ as M\"obius transformations}
\label{picDual}
\end{figure}

Note that $\mathfrak s_i$ fixes $C_j$, $j\ne i$
and moves $C_i$ to the unique other circle that is tangent to $C_j$'s, $j\ne i$.

We set
 $$G_{\mathcal A}(\mathcal P):=\langle \mathfrak s_1, \mathfrak s_2, \mathfrak s_3,\mathfrak s_4\rangle .$$

\begin{Lem} For each $1\le i\le 4$,
$$\Phi_{\mathcal D_0} (S_i)=\mathfrak s_i;$$
hence
 $$ \Phi_{\mathcal D_0} (\mathcal A) =G_{\mathcal A}(\mathcal P).$$
\end{Lem}
\begin{proof}
If $\gamma_i:=\Phi_{\mathcal D_0} (S_i)$,
then by \eqref{wd},
$$W_{\gamma_i(\mathcal D_0)}=W_{\mathcal D_0} (\psi_{\mathcal D_0} S_i)^{-1}
=  S_i W_{\mathcal D_0}.$$
On the other hand, by \cite[3.25]{GrahamLagariasMallowsWilksYanI},
$$W_{\mathfrak s_i (\mathcal D_0)}=S_i W_{\mathcal D_0} .$$
Therefore
$$\gamma_i (\mathcal D_0)=\mathfrak s_i (\mathcal D_0) ;$$
and hence $\gamma_i=\mathfrak s_i$.
\end{proof}

Each inversion $\frak s_i$ extends uniquely to an isometry of the hyperbolic space $\bH^3$, corresponding to
the inversion with respect to the hemisphere whose boundary is $\hat C_i$. The intersection of the exteriors
of these hemispheres is a fundamental domain
 for the action of $G_{\mathcal A}(\mathcal P)$ on $\bH^3$.

\begin{figure}
 \includegraphics [width=2in]{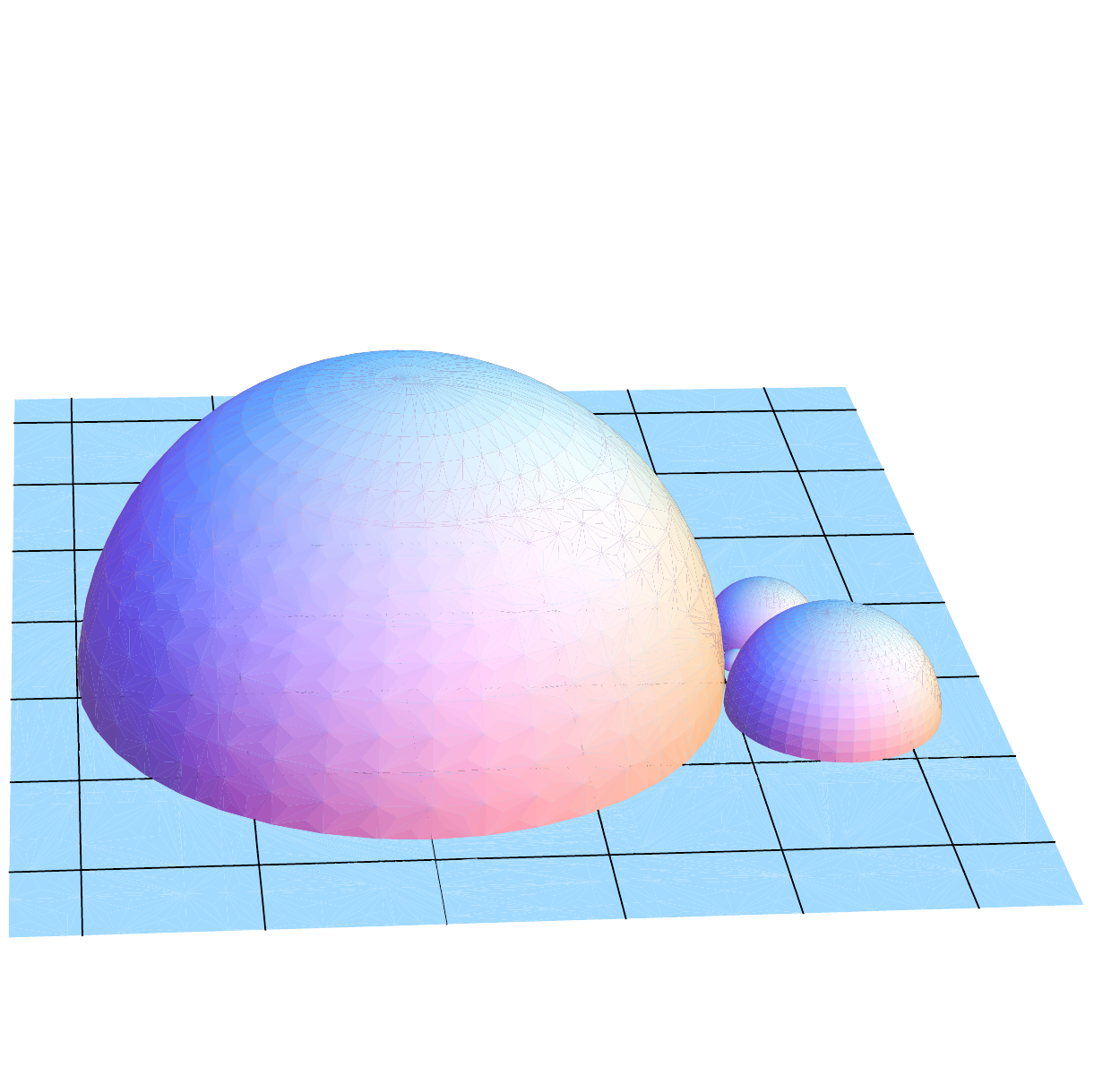}
\caption[Fundamental domain]{Fundamental domain}
\label{picFunDom}
\end{figure}

The Hausdorff dimension of $\Lambda(\mathcal P)$, say $\alpha$,  is independent of $\mathcal P$.
Hirst showed in \cite{Hirst1967} that
$\alpha$ is strictly between one and two.
For our purpose, we only need to know that $\alpha>1$,
though much more precise estimates were made by Boyd in \cite{Boyd1982} and
 McMullen \cite{McMullen1998}.

\begin{Prop}
\begin{enumerate}
\item  $G_{\mathcal A}(\mathcal P)$ is geometrically finite and discrete.
\item We have $\Lambda {(G_{\mathcal A}(\mathcal P))}=\Lambda (\mathcal P) $, and hence
$\alpha=\delta_{G_{\mathcal A}(\mathcal P)}$.
\end{enumerate} 
\end{Prop}
\begin{proof} For simplicity set $\G:=G_{\mathcal A}(\mathcal P)$.
The group $\G$ is geometrically finite (that is,
it admits a finite sided Dirichlet domain) by \cite[Thm.13.1]{Kapovichbook}
and discrete since $\mathcal A$ is discrete and $\Phi_{\mathcal D_0}$ is a topological isomorphism.

 Clearly, $\G$ is non-elementary.
It is well-known that $\Lambda (\G)$
is the same as the set of all accumulation points in the orbit $x_0$
under $\Gamma$ for any (fixed) $x_0\in \hat \c$.
On the other hand, by \cite[Thm 4.2]{GrahamLagariasMallowsWilksYanI}, $\Lambda(\mathcal P)$
is equal to the closure of all tangency points of circles in $\mathcal P$
and is invariant under $\G$.
This immediately yields
$$\Lambda (\Gamma)\subset \Lambda (\mathcal P).$$
If $x_0\in \Lambda(\mathcal P)$, then
any neighborhood, say $U$, of $x_0$ contains infinitely many circles.
Since $\cup_{1\le i\le 4} \G(C_i)$ is the set of all circles in
$\mathcal P$, there exist $j$ and an infinite sequence $\gamma_i\in \G$ such that $\gamma_i (C_j)\subset U$ for all $i$.
Therefore $x_0\in \Lambda(\G)$.
This proves that $ \Lambda (\mathcal P)\subset \Lambda (\Gamma).$\end{proof}


Set $\Gamma_{\mathcal P}$  to be the subgroup
 of holomorphic elements, that is,
$$\Gamma_{\mathcal P}:=G_{\mathcal A}(\mathcal P)\cap \PSL_2(\c),$$
since $\text{M\"ob}(2)$ is the semidirect product of complex conjugation with the subgroup  $\text{M\"ob}_+(2)=\PSL_2(\c)$
of orientation preserving
transformations.
Then the above lemma holds with $\Gamma_{\mathcal P}$ in place of $G_{\mathcal A}(\mathcal P)$,
as both properties are inherited by a subgroup of finite index.
By the well known Selberg's lemma, we can further replace $\G_{\mathcal P}$ a torsion-free
subgroup of finite index.

\subsection{Reduction to orbital counting for a Kleinian group}
Let $G:=\PSL_2(\c)$ and $\Gamma <G$ be a geometrically finite, torsion-free Kleinian subgroup.
Suppose we are given a real linear representation $\iota: G\to \SO_{F}({\br})$
where $F$ is a real quadratic form in $4$ variables with signature $(3,1)$.
As we prefer to work with a right action, we consider $\br^4$ as the set of row vectors and the action is given by $vg:=\iota(g) v^t$ for $g\in G$ and $v\in \br^4$.

By Lemma \ref{reduction} and the discussions in the previous subsection, Theorem \ref{main1} follows from:
\begin{Thm}\label{reduction2}  Suppose $\delta_\G >1$.
Let $v_0\in \br^4$ be a non-zero vector lying in the cone $F=0$
 with a discrete orbit $ v_0 \G \subset \br^4$.
Then for any norm $\|\cdot \|$ on $\br^4$, there exists $c>0$ such that
$$\# \{v\in v_0\G: \| v\|<T\}\sim  c \cdot T^{\delta_{\G}}
 .$$
\end{Thm}

It also follows from
Lemma \ref{reduction} that this theorem implies
$$N_2^{\mathcal P}(T)\sim (3 c)\cdot  T^{\alpha}$$
for the same $c$ as in Theorem \ref{main1}.


\section{Geometry of closed horospheres on $\op{T}^1(\G\ba \bH^3)$}\label{closedh}
Let $G=\PSL_2(\c)$ and $\bH^3=\{(x_1, x_2, y): y>0\}$. We use the notations
$N$, $A$, $K$, $M$, $n_x$ and $a_y$ as defined in \eqref{nak}.
 Throughout this section, we suppose that
 $\G<G$ is a torsion-free, non-elementary (i.e., $\Lambda(\G)$ consists of
 more than $2$ points) and geometrically finite Kleinian group.

\begin{Def}\label{boundedparabolic}{\rm
\begin{enumerate}
\item A point $\xi\in \Lambda(\G)$ is called
a {\it parabolic fixed point} if $\xi$ is fixed by a parabolic isometry of $\G$ (that is, an element
of trace $\pm 2$). The {\it rank} of a parabolic fixed point $\xi$ is defined to be
rank of the abelian
group $\G_{\xi}$ which stabilizes $\xi$.

\item A parabolic fixed point $\xi$ is called a {\it bounded parabolic}
 if it is of rank $2$ or
 if
 there exists a pair of two open disjoint discs in $\hat \c$ whose union, say $U$, is precisely $\G_\xi$-invariant
  (that is, $U$ is invariant by $\G_{\xi}$ and $\gamma(U)\cap U\ne \emptyset$ for all $\gamma\in
  \Gamma $ not belonging to $\G_{\xi}$).
Or equivalently, $\xi$ is bounded parabolic if $\G_{\xi}\backslash (\Lambda(\G)-\{\xi\})$ is compact.
\item A point $\xi\in \Lambda(\G)$ is called a
{\it point of approximation} if there exist a sequence $\gamma_i\in \G$
and $z\in \bH^3$ such that $\gamma_i z\to \xi$ and $\gamma_i z$
is within a bounded distance from the geodesic ray ending at $\xi$.
\end{enumerate}}
\end{Def}

Since $\G$ is geometrically finite, it is known by the work of Beardon and Maskit
 \cite{BeardonMaskit1974}
that
any limit point $\xi\in \Lambda(\G)$ is either a point of approximation or
a bounded parabolic fixed point.

We denote by $\op{Vis}$ the visual map from $\op{T}^1( \bH^3)
$ to the ideal boundary $\partial_\infty(\bH^3)$
which maps the vector $(p, \vec v)$ to the end of the geodesic ray tangent to $\vec v$.

\begin{Def}\label{og} {\rm \begin{itemize}
                            \item
 We denote by $\tilde \Omega_\G\subset
\op{T}^1( \bH^3)$ the set of vectors $(p, \vec v)$ whose
 the image under the visual map belongs to $\Lambda(\G)$.
\item We denote by $\hat{\Omega}_\G$ the image of $\tilde \Omega_\G$ under the projection $
\op{T}^1( \bH^3)\to \op{T}^1( \G\ba \bH^3)$.
                           \end{itemize}}\end{Def}


We note that $\tilde \Omega_\G$ is a closed set
consisting of horospheres.
For comparison, if $\G$ has finite co-volume, then $\Lambda_\G=\c\cup\{\infty\}$ and  $\hat \Omega_\G=
\op{T}^1( \G\ba \bH^3)$.

For $u\in \PSL_2(\c)$, we denote by $[u]$ its image in $\PSL_2(\c)/M$,
which we may consider as a vector in $\op{T}^1(\bH^3)$.
The horosphere determined by the vector $[u]$ in $\G\ba \op{T}^1(\bH^3)$ corresponds to
$\G\ba \G uNM /M$, which will be simply denoted by $\G\ba \G [u]N$.
We note that $u(\infty)$ is identified with $\op{Vis}[u]$.

\begin{Thm}[Dal'Bo]\cite{Dalbo2000} \label{dalbo}  Let $u\in \PSL_2(\c)$.
\begin{enumerate}
\item If
$u(\infty)\in \Lambda(\G)$ is a point of approximation, then $\G \ba \G [u] N$ is dense in $\hat \Omega_\G$.
\item
The orbit  $\G\ba \G  u N$ is closed in $\G\ba G$ if and only if
either $u(\infty)\notin \Lambda(\G)$ or
$u(\infty)$ is a bounded parabolic fixed point for $\G$.
\end{enumerate}
\end{Thm}

The first claim was proved in \cite[Prop. C and Cor. 1]{Dalbo2000} under the condition that
the length spectrum of $\G$ is not discrete in $\br$.  As remarked there, this condition holds
for non-elementary Kleinian groups by \cite{GuivarchRaugi1986}.
The second claim follows from \cite[Prop. C and Cor. 1]{Dalbo2000} together with
 the following lemma.
\begin{Lem}\label{stab}
\begin{enumerate}\item  We have $\G\cap NM=\G\cap N$.
\item The orbit $\G\ba \G  N M $ is closed if and only if $\G\ba \G N$ is closed.
\end{enumerate} \end{Lem}
\begin{proof}
 If $\gamma\in\G \cap  NM$, but
not in $N$, then
$\operatorname{Tr}^2(\gamma)$ is real and $0\le \operatorname{Tr}^2(\gamma )<4$
where $\operatorname{Tr}(\gamma )$ denotes the trace of $\gamma$. Hence $\gamma$
is an elliptic element.
As $\G$ is discrete, $\gamma$ must be of finite order. However $\G$ is torsion-free,
which forces $\gamma=e$. This proves the first claim.
 The second claim follows easily since the first claim implies
the inclusion map $\G\cap N \ba N \to \G\cap NM \ba NM$ is proper.
\end{proof}

It follows that if $\G\ba \G N$ is closed in $\G\ba \PSL_2(\c)$,
then $\G\ba \G N$ is either
\begin{enumerate}
\item the embedding of a plane, if $\infty\notin \Lambda(\G)$;
\item the embedding of a cylinder, if $\infty$ is a bounded parabolic
fixed point of rank one; or
\item the embedding of a torus, if $\infty$ is a bounded parabolic fixed point
of rank two.
\end{enumerate}

For $X\subset \bH^3\cup\partial_\infty(\bH^3)$,
$\bar X$ denotes its closure in $\bH^3\cup\partial_\infty(\bH^3)$.
\begin{Prop} \label{ff} Assume that $\G\ba \G N$ is closed.
 There exists a finite-sided fundamental polyhedron $\mathcal F\subset
\bH^3$ for the action of $\G$, and also a fundamental domain $\mathcal F_N\subset \c$ for the action of $N\cap\G$ such that
 for some $r\gg 1$ and for some finite subset $I_\G\subset \G$,
\begin{equation*}
\{(x_1, x_2, y)\in \bH^3: x_1+ix_2\in \mathcal F_N, \; x_1^2+x_2^2+ y^2 >r\}\subset \mathcal \cup_{\gamma\in I_\G} \gamma \mathcal F .
\end{equation*}

\end{Prop}
\begin{proof}
 Choose a finite-sided fundamental polyhedron $\mathcal F$ for
$\G$ with $\infty\in \bar {\mathcal F}$.
If $\infty \notin
 \Lambda(\G)$, then $\infty$ lies in the interior of
$\cup_{\gamma\in I_\G} \gamma \mathcal F$ for some finite $I_\G\subset \Gamma$.
As the exteriors of hemispheres form
 a basis of neighborhoods
of $\infty$ in $\bH^3\cup\partial_\infty(\bH^3)$, and $\mathcal F_N=\c$,
the claim follows.

Now suppose that $\infty$ is a bounded parabolic fixed point of rank one.
Let $n_{v_1}\in \G\cap N$ be a generator for $\G\cap N$ for $v_1\in \c$,
and fix a vector $v_2\in \c$
 perpendicular to $v_1$.
The set  $\mathcal F_N:=\{s_1v_1+s_2v_2\in \c : 0\le s_1<1\}$ is a fundamental
domain in $\c$ for the action of $\G\cap N$.
By replacing $\mathcal F$ if necessary,
 we may assume that $\mathcal F \subset \mathcal F_N\times \br_{>0}$.
Then by \cite[Prop. A. 14 in VI]{Maskit1988}, for all large $c>1$, the sets
$$S(c):=\{ (s_1v_1+s_2v_2, y)\in \bH^3: y>c\text{ or }  |s_2| > c\}$$
is precisely invariant by $\G\cap N$ and
$\mathcal F - S(c)$ is bounded away from $\infty$.
Since $\mathcal F\cap S(c)\ne \emptyset$ as $\infty\in \bar{\mathcal F}$ and
$S(c)$ is precisely invariant by $N\cap \G$, it follows that
\begin{equation}\label{tc} \{(x_1, x_2, y) \in S(c): x_1+ix_2 \in \mathcal F_N \}
\subset \mathcal F . \end{equation}

Since $$\{(x_1, x_2, y): x_1+ix_2 \in \mathcal F_N,
 x_1^2+x_2^2 +y^2 >r\}\subset S(c)$$ for some large $c>0$,
 the claim follows in this case.

If $\infty$ is a bounded parabolic of rank two,
$S(c):=\{(x_1, x_2, y): y>c\}$ is precisely invariant by $\G\cap N$
and $\mathcal F-S(c)$ is bounded away from $\infty$ by
\cite[Prop. A. 13 in VI]{Maskit1988}. Choose $\mathcal F_N$ so that
$\mathcal F\subset \mathcal F_N\times \br_{>0}$.
Then \eqref{tc} holds for the same reason, and
since $\mathcal F_N$ is bounded,
$$\{(x_1, x_2, y): x_1+ix_2 \in \mathcal F_N,
 x_1^2+x_2^2 +y^2 >r\}\subset S(c)$$ for some large $c>0$.
This proves the claim.
\end{proof}

Considering the action of
 $N\cap \G$ on $\partial_\infty(\bH^3)\setminus \{\infty\}=\c $,
we denote by $\Lambda_N(\G)$
 the image of $\Lambda(\G)\setminus \{\infty\}$
 in the quotient $(N\cap \G)\ba \c$, that is,
 $$\Lambda_N(\G):=\{ [n_x]\in (N\cap \G)\ba N: x\in \Lambda(\G)\setminus \{\infty\} \}.$$

\begin{Prop} \label{lg} If $\G\ba \G N$ is closed, the set $\Lambda_N(\G)$
is bounded.
\end{Prop}
\begin{proof} This is clear if $\infty$ is a bounded parabolic of rank two,
as $N\cap \G\ba N$ is compact. If $\infty\notin \Lambda(\G)$,
then $\Lambda(\G)$ is a compact subset of $\c$, and hence the claim follows.
In the case when $\infty$ is a bounded parabolic of rank one,
 the claim follows from the well-known fact that $\Lambda(\G)$ lies in a strip of finite width
(see \cite[Pf. of Prop. 8.4.3]{Thurstonbook}).
This can also be deduced from Proposition \ref{ff} using the fact that
the intersection of the convex core of $\G\ba \bH^3$ and the thick part of the manifold
$\G\ba \bH^3$
is compact for a geometrically finite group.
\end{proof}


\begin{Prop}\label{mp}  Assume that $\G\ba \G N$ is closed.
 For any compact subset $J \subset \G\ba G$, the following set is bounded:
 \begin{equation}\label{nj} N(J):=\{[n]\in (N\cap \G)\ba  N\ : \G\ba \G nA\cap  J \ne\emptyset\}.\end{equation}
\end{Prop}
\begin{proof} Let $\mathcal F_N$ and $\mathcal F$ be as in Proposition \ref{ff}.
If $\mathcal F_N$ is bounded, there is nothing to prove.
Hence
by Theorem \ref{dalbo},
we may assume that either $\infty\notin \Lambda(\G)$ or $\infty$ is a
bounded parabolic fixed point of rank one.

To prove the proposition, suppose on the contrary that
there exist  sequences  $n_{j}\in \mathcal F_N \to \infty$,  $a_{j}\in A$, $\gamma_j\in \G$
and $w_j\in J $ such that
 $n_{j} a_{j}= \gamma_j w_j$.
As $J$ is bounded, we may assume $w_j\to w\in G$ by passing to a subsequence.
Then $ \gamma_j^{-1} n_{j}a_j \to w$. Let $\gamma_0\in \G$ be such that
 $\gamma_0 w (0,0,1)\in\bar {\mathcal F}$.
By the geometric finiteness of $\G$, there exists a finite union, say
$\mathcal F'$, of translates of $\mathcal F$
such that $\gamma_0w_j (0,0,1)\in \mathcal F'$ for all large $j$.

As $n_j\to \infty$,
the Euclidean norm of $n_ja_j(0,0,1)$ goes to infinity, and hence
by Proposition \ref{ff},
 $$ n_ja_j(0,0,1)\in \mathcal F '\quad\text{for all large $j$} ,$$
by enlarging $\mathcal F'$ if necessary.

Therefore for all large $j$,
$$\gamma_0 \gamma_j^{-1} n_j a_j (0,0,1)=\gamma_0 w_j (0,0,1) \in
\mathcal F'\cap \gamma_0\gamma_j ^{-1}(\mathcal F ') .$$
Since $\mathcal F'\cap \gamma_0\gamma_j^{-1} (\mathcal F ')\ne \emptyset$ for  only finitely many
  $\gamma_j$'s,
we conclude that $\{\gamma_j\}$ must be a finite set.
As $n_ja_j=\gamma_j w_j\in \gamma_j J$,
$n_ja_j$ must be a bounded sequence,
 contradicting
  $n_j\to \infty$.
\end{proof}

Note that $N(J)$ is defined so that for all $y>0$,
$$(N\cap \G)\ba  N a_y\cap J\subset N(J) .$$

\begin{Cor} \label{suppcom} Let $\psi\in C_c(\G\ba G)$ have support $J$.
\begin{enumerate}
\item The set $N(J)$ defined
in \eqref{nj} is bounded.
 \item For any function $\eta\in C_c(N\cap \G\ba N)$
with $\eta|_{ N(J)}\equiv 1$,
we have for all $y>0$,
$$\int_{(N\cap \G) \ba N}\psi (na_y)dn =\int_{(N\cap \G)\ba N }\psi(n a_y)\; \eta(n)\; dn .$$
\end{enumerate}
\end{Cor}
\begin{proof} The first claim is immediate from the above proposition.
For (2), it suffices to note that
$\psi(na_y)\equiv0$ for $n$ outside of $N(J)$, and hence
$$
\int_{(N\cap \G)\ba N}\psi (na_y)\  \eta(n)\ dn =\int_{ N(J) }\psi(n a_y)\ \eta(n) \ dn .
$$
Using $\eta \equiv 1$ on $N(J)$, the claim follows.
\end{proof}

\section{The base eigenfunction $\phi_0$}\label{phizeroo}

In this section, we assume that $\G <G=\PSL_2(\c)$ is a
geometrically finite torsion-free discrete subgroup
and that the Hausdorff dimension $\delta_\G$ of the limit set $\Lambda(\G)$ is strictly bigger than one. Assume also that $\G\ba \G N$ is closed.


 By Sullivan \cite{Sullivan1979}, there
exists a positive $L^2$-eigenfunction $\phi_0$, unique up to a scalar multiple,
 of the Laplace operator
$\Delta$ on $\G\ba \bH^3$ with smallest eigenvalue $\delta_\G (2-\delta_\G)$ and which
 is square-integrable, that is,
 $$\int_{\G\ba \bH^3} \phi_0(x,y)^2 \frac{1}{y^3} dx dy <\infty .$$

 In this section, we study
the properties of $\phi_0$ along closed horospheres. For the base
point $o=(0,0,1)$, we denote by $\nu_o$ a weak limit as $s\to
\delta_\G$ of the family of measures on $\overline \bH^3$:
\begin{equation*}
\nu(s):=\frac{1}{\sum_{\gamma\in \G} e^{-sd(o, \gamma
o)}}\sum_{\gamma\in \G} e^{-sd(o, \gamma o)} \delta_{\gamma
o} .\end{equation*}

The measure $\nu_o$ is indeed the unique weak limit of
$\{\nu(s)\}$ as $s\to \delta_\G$. Sullivan \cite{Sullivan1979}
showed that $\nu_o$ is the unique finite measure (up to a constant
multiple) supported on $\Lambda(\G)$ with the property: for any $\gamma\in \G$,
\begin{equation}\label{conformal}\frac{d(\gamma_*\nu_o)}{d\nu_o}(\xi)=e^{-\delta_\G
\beta_\xi(\gamma(o), o)}\end{equation} where
$\beta_\xi(z_1,z_2):=\lim_{z\to \xi} d(z_1,z)-d(z_2,z)$ is the
Busemann function and $\gamma_*\nu_o(A):=\nu_o(\gamma^{-1}(A))$. The
measure $\nu_o$, called the Patterson-Sullivan measure,
 has no atoms.


The base eigenfunction
$\phi_0$ can be explicitly written as the integral of the Poisson kernel against $\nu_o$:
\begin{align}\label{PSS} \phi_0(x+iy)&=
\int_{u\in \Lambda(\G)} e^{-\delta_\G  \beta_u (x+iy, o) } d\nu_o(u) \\
&=\int_{u\in\Lambda(\G)\setminus \{\infty\}}
 \left({(\|u\|^2+1) y \over \|x-u\|^2+y^2}\right)^{\delta_\G} d\nu_o(u) ,
\end{align}
where $\Lambda(\G)$ is identified as a subset of $\c\cup\{\infty\}$.
The main goal of this section is to study  the average of $\phi_0$ along the
 translates $\G\ba \G N a_y$.

\begin{Def}\label{PS} {\rm For a given $\psi\in C (\G\ba G)^K$, define
$$\psi^N(a_y):=\int_{n\in (N\cap \G)\ba N} \psi(n a_y) \; dn .$$}
\end{Def}

We will show that the integral $\phi_0^N(a_y)$
above
converges, and moreover that all of the action in this integral takes
place only over the following bounded set (see Proposition \ref{lg}):
 $$\Lambda_N(\G)=\{[n_x]\in (N\cap \G)\ba N : x\in \Lambda(\G)\setminus \{\infty\}\} .$$

If $\infty$ is a bounded parabolic fixed point of rank one,
then
there exists
a vector $v_1\in \br^2$ such that $n_{v_1}$ is a generator for $N\cap \G$.
Fixing vector $v_2$ perpendicular to $v_1$,
 we can decompose $x\in\R^2$ as  $x=x_1v_1+x_2v_2$.
\begin{Def}\label{proper} {\rm We say that an open subset $B\subset  (N\cap \G)\ba N$
{\it properly covers} $\Lambda_N(\G)$ if the following holds:
\begin{enumerate}
 \item if $\infty\notin \Lambda(\G)$, then
$\e_0(B):=\inf_{u\in \Lambda_N(\G), x\notin B} \|x-u\|>0 $
\item if $\infty$ is a bounded parabolic of rank one, then
$$\e_0(B) :=\inf_{x \notin B, u\in \Lambda_N (\G)} |u_2-x_2| >0 .$$
\item if $\infty$ is a bounded parabolic of rank two, then $B=(N\cap \G)\ba N$.
\end{enumerate} } \end{Def}


\begin{Prop} \label{ppo} \begin{enumerate}
\item $\phi_0^N(a_y)\gg y^{2-\delta_\G}$ for all $0<y\ll 1$.
\item For any open subset $B\subset (N\cap \G)\ba N$ which properly covers
$\Lambda_N(\G)$, and all small $y>0$,
\be\label{phi0B}
\phi_0^N(a_y)=
\int_B \phi_0(n_xa_y)\; dx +  O_{\e_0(B)} (y^{\delta_\G }).
\ee
\end{enumerate}
\end{Prop}
\begin{proof}
Choose a fundamental domain $\mathcal F_N\subset \c $ for $(N\cap \G)\ba N$. We first show that
 $\phi_0^N(a_y)\gg y^{2-\delta_\G}$ for all $0<y\ll 1$.
For a subset $Q\subset \mathcal F_N$,
the notation $Q^c$ denotes the complement of $Q$ in $ \mathcal F_N $.
Observe that for a subset $Q\subset \mathcal F_N$,
\begin{align}\label{phizero}
&\int_{x\in Q}\phi_0(n_xa_y) \;dx\\
& =
\int_{u\in\Lambda(\G)}
(\|u\|^2+1) ^{\delta_\G}
\int_{x\in Q}
\left({y \over \|x-u\|^2+y^2}\right)^{\delta_\G}
dx
d\nu_o(u) \notag\\
&= \int_{u\in\Lambda (\G) }
\left( \int_{x\in  Q-u} (\|u\|^2+1) ^{\delta_\G} \left({y \over \|x\|^2+\|y\|^2 } \right)^{\delta_\G} dx
\right) d\nu_o(u)\notag
\end{align}
where the interchange of orders is justified since everything is nonnegative.

We observe that by changing the variable $w=\frac{x}{y}$,
\begin{align}\label{whole}\int_{x\in \br^2}
 \left({y \over \|x\|^2+\|y\|^2 } \right)^{\delta_\G} dx
&=  y^{2-\delta_\G}
 \int_{w\in \br^2} \left({1 \over \|w\|^2+1} \right)^{\delta_\G} dw  \\
&= y^{2- \delta_\G} 2\pi \int_{r>0}\frac{r}{{(r^2+1)}^{\delta_\G} }\; dr
\notag \\
&=\frac{ \pi }{\delta_\G-1} y^{2-\delta_\G} .\notag
\end{align}

\noindent{\bf{Case I: $\infty\notin \Lambda(\G)$.}}
In this case, we have
$$\omega_0=\int_{u\in\Lambda (\G)}
(\|u\|^2+1) ^{\delta_\G} d\nu_o(u) \ll \nu_o(\Lambda(\G)) <\infty .$$
Therefore for $Q=\c$, we obtain from \eqref{phizero} and \eqref{whole} that
 \begin{align}\label{first} \phi_0^N(a_y)&=  \int_{u\in\Lambda (\G)}
(\|u\|^2+1) ^{\delta_\G}  d\nu_o(u)\cdot
 \int_{x\in \br^2}
 \left({y \over \|x\|^2+\|y\|^2 } \right)^{\delta_\G} dx\\
&=\frac{\omega_0 \cdot \pi }{\delta_\G-1} y^{2-\delta_\G} \notag.\end{align}

This proves (1). To prove (2),
let $B$ be an open subset of $\c$ which properly covers $\Lambda_N(\G)$. Then
we have $$\e_0:=\inf_{u\in \Lambda_N(\G), x\in B^c} \|x-u\|=
\inf\{\|x\| :  x\in B^c- \Lambda_N(\G) \}>0 .$$
By setting $Q=B^c$ in \eqref{phizero},
 we deduce
\begin{align*} \int_{x\in B^c }\phi_0(n_xa_y) \;dx
&\le \omega_0
 \cdot \int_{\|x\| > \e_0 } \left({y \over \|x\|^2+\|y\|^2} \right)^{\delta_\G} dx
\\
&\le \omega_0 y^{2-\delta_\G}
 \cdot \int_{\|w\| > \e_0{y}^{-1} } \left({1 \over \|w\|^2+1} \right)^{\delta_\G} dw
  \\ & = 2\pi  \omega_0 y^{2-\delta_\G}
\cdot \int_{r > \e_0 y^{-1}} \frac{r}{{(r^2+1)}^{\delta_\G} }\; dr\\
&= 2\pi  \omega_0 y^{2-\delta_\G} \left( \frac{\e_0^2}{y^2}+1\right)^{1-\delta_\G}
\\ &\ll y^{\delta_\G}.
\end{align*}

\noindent{{\bf Case II}:
$\infty$ is a bounded parabolic fixed point of rank one.}

We may assume without loss of generality that
$\G\cap N$ is generated by $(1, 0)$, that is, $x\mapsto x+1$, and
$\mathcal F_N=\{(x_1, x_2): 0\le x_1<1, x_2\in \br\}$.
There exists $ T >0$ such that
 $\Lambda (\G)\subset \br \times [-T, T]$
 by \cite[Pf. of Prop. 8.4.3]{Thurstonbook}.

By computing the Busemann function, we deduce from
\eqref{conformal} for $k=(k_1, k_2)\in \G\cap N$ that
 \be\label{transLaw}
 d ({(n_{-k})}_*\nu_o) (u)= \left(\frac{\|u\|^2+1}{\|k-u\| ^2+1}\right)^{\delta_{\G}} d\nu_o (u) .
 \ee

We have by changing orders of integrations of $x_1$ and $u_1$ that
\begin{align}\label{ranone}&
\int_{x\in \mathcal F_N} \phi_0(x,y) dx\\
&=\int_{u\in \br\times [-T, T]}\int_{x\in ([0,1]\times \br) -u}\left( { (\|u\|^2+1)y \over \|x\|^2 + y^2 } \right)^\gd dx d\nu_o(u) \notag\\ &=
\int_{x\in \br^2} \left( { y \over \|x\|^2 + y^2 } \right)^\gd
\int_{u\in [-x_1, 1-x_1] \times [-T, T]}  (\|u\|^2+1) ^\gd d\nu_o(u) dx .\notag
\end{align}

We set $c_1$ and $c_2$ to be, respectively, the infimum and the supremum of
$$e_{x_1}:=\int_{u\in [-x_1, 1-x_1] \times [-T, T]}  (\|u\|^2+1) ^\gd d\nu_o(u)$$
over all $x_1\in \br$.
We claim that $$0<c_1\le c_2<\infty .$$

For any $k_1\in \z$, by changing the variable $u_1\mapsto u_1+k_1$ and recalling
 \eqref{transLaw}, we deduce
\begin{align*}e_{x_1}=& \int_{u\in [-x_1, 1-x_1] \times [-T, T]}  (\|u\|^2+1) ^\gd d\nu_o(u) \\
&=\int_{u_1=k_1-x_1}^{u_1= k_1+1-x_1}
\int_{u_2=-T}^{u_2=T} (\|u-k\|^2+1)^\gd d((n_{-k})_* \nu_o)(u) \\
&=\int_{u_1=k_1-x_1}^{u_1= k_1+1-x_1}
\int_{u_2=-T}^{u_2=T}
(\|u\|^2+1)^\gd
d\nu_o(u).\end{align*}

Choosing $k_1$ so that $x_1\in [k_1-1, k_1)$, we have
 $e_{x_1}\ll \nu_o([0,2]\times [-T, T])$ and hence $c_2<\infty$.
Note that $e_{x_1}\asymp e_{x_1+1}$, where the implied
constant is independent of $x_1\in \br$,
and that $$e_{x_1}+e_{x_1+1}= \int_{u_1=k_1-x_1-1}^{u_1= k_1+1-x_1}
\int_{u_2=-T}^{u_2=T}
(\|u\|^2+1)^\gd
d\nu_o(u)\gg  \nu_o([0,1]\times [-T, T]) .$$

On the other hand,
since the $N\cap \G$-translates of
$[0,1]\times [-T,T]$ cover the support of $\nu_o$ except for $\infty$ and $\nu_o$ is atom-free,
we have $\nu_o([0,1]\times [-T, T])>0$. 
This proves $c_1=\inf e_{x_1} >0$.


We now deduce from \eqref{ranone} and \eqref{whole} that
$$\int_{x\in \mathcal F_N} \phi_0(x,y) dx
\gg c_1
\int_{x\in \br^2} \left( { y \over \|x\|^2 + y^2 } \right)^\gd  dx
\gg y^{2-\delta_\G},
$$
proving (1).

If $B$ is an open subset of $\mathcal F_N=[0,1]\times \br$ which properly
covers $\Lambda_N(\G)$,
 we have $$\e_0:=\inf_{x \in B^c, u\in \Lambda(\G)} |u_2-x_2| >0.$$



Hence
$\{(x, u)\in \br^2\times \Lambda_N(\G): x\in  B^c -u\}$ is contained in the set
$$\{u_1\in \R, -u_1 <x_1 < 1-u_1, |u_2| <T, |x_2| > \e_0\} .$$

Since $$
\{u_1\in\R,\ \  -u_1 < x_1 <1-u_1\} =
\{x_1\in\R, \ \ -x_1 < u_1 < 1-x_1\},
$$
we deduce from \eqref{phizero} by changing the order of integrations that
\begin{align*}
&\int_{x\in B^c} \phi_0(x,y) dx
\\&=
\int_{u\in\Lambda (\G)}
\int_{x \in B^c-u }
\left( { (\|u\|^2+1)y \over \|x\|^2 + y^2 } \right)^\gd
dxd\nu_o(u)\\
&\le \int_{x_1\in \br, |x_2|> \e_0 }
\left( { y \over \|x\|^2 + y^2 } \right)^\gd
\left( \int_{u_1=-x_1}^{u_1=1-x_1}\int_{u_2=-T}^{u_2=T}
(\|u\|^2+1)^\gd
d\nu_o(u)
\right)
dx_2dx_1  \\
&\le c_2\cdot
\int_{x_1\in \br, |x_2|> \e_0 }
\left( { y \over \|x\|^2 + y^2 } \right)^\gd
dx_2dx_1 .
\end{align*}



The $x_1$ integral can be evaluated explicitly and yields:
\beann
\int_{x\in B^c} \phi_0(x,y) dx
&\ll&
{\sqrt\pi \cdot \G(\gd-1/2)\over \G(\gd)}
y^\gd
\int_{|x_2|>\e_0}
\left(
{1\over x_2^2+y^2}\right)
^{\gd-1/2}
dx_2
\\
&\ll&
{\sqrt\pi \cdot \G(\gd-1/2)\over \G(\gd)}
y^\gd
\int_{|x_2|>\e_0}
\left(
{1\over x_2^2}\right)
^{\gd-1/2}
dx_2
\\
&\ll&
y^\gd
.
\eeann

Hence (2) is proved.

\noindent{{\bf Case III}:
$\infty$ is a bounded parabolic fixed point of rank two.}
In this case (2) holds for a trivial reason.
But we still need to show (1). The argument is similar to the case II.
Without loss of generality, we assume that $N\cap \G$ is generated by
$(1,0)$ and $(0,1)$, so that $\mathcal F_N=[0,1]\times [0,1]$.

Similarly to \eqref{ranone}
we have
\begin{align*}\label{rantwo}
 \phi_0^N(a_y) &=
\int_{x\in \br^2} \left( { y \over \|x\|^2 + y^2 } \right)^\gd
\int_{u\in [-x_1, 1-x_1] \times [-x_2, 1-x_2]}  (\|u\|^2+1) ^\gd d\nu_o(u) dx\notag
\end{align*}
By \eqref{whole},
it suffices to show that
$$c_0:=\inf_{x_1, x_2\in \br}
\int_{u\in [-x_1, 1-x_1] \times [-x_2, 1-x_2]}  (\|u\|^2+1) ^\gd d\nu_o(u) >0 .$$

For $(x_1, x_2)\in [k_1-1, k_1)\times [k_2-1, k_2)$,
by changing the variables $u_i\to u_i+k_i$ for $k_i\in \z$ and by using \eqref{transLaw},
we deduce
\begin{multline*} \int_{u_1=-x_1}^{u_1= 1-x_1}\int_{u_2=-x_2}^{u_2=1-x_2}   (\|u\|^2+1) ^\gd d\nu_o(u)
\\ =\int_{u_1=k_1-x_1}^{u_1= k_1+1-x_1}
\int_{u_2=k_2-x_2}^{u_2=k_2+1-x_2}
(\|u\|^2+ 1)^\gd
d\nu_o(u).\end{multline*}
Hence
$$ c_0\gg \inf_{k_i-1\le x_i\le k_i} \nu_o ([k_i-x_i, k_i-x_i +1-x_i])\gg \nu_o([0,1]^2)  >0 .$$
\end{proof}

 \begin{Cor}\label{ppoo} For any $y>0$,
there exist $c_{\phi_0}>0$ and $d_{\phi_0}\ge 0$ such that
$$ \phi_0^N(a_y) =
c_{\phi_0} y^{2-\delta_\G}+ d_{\phi_0} y^{\delta_\G} .$$
If $\infty\notin \Lambda(\G)$, then $d_{\phi_0}=0$.
\end{Cor}

\begin{proof}
Since $\Delta \phi_0 =\delta_{\G}(2-\delta_{\G}) \phi_0$,
it follows that
$$-y^2 \frac{\partial^2}{\partial y^2} \phi_0^N +y
\frac{\partial}{\partial y} \phi_0^N =\delta_{\G}(2-\delta_{\G})\phi_0^N .$$

As both $y^{\delta_{\G}}$ and $y^{2-\delta_{\G}}$ satisfy the above differential equation,
we have $$\phi_0^N(a_y)=c_{\phi_0} y^{2-\delta_{\G}}+d_{\phi_0} y^{\delta_{\G}} $$
for some $c_{\phi_0}, d_{\phi_0} \ge 0$.
Proposition \ref{ppo} (1) implies that $c_{\phi_0}>0$.

In the case when $\mathcal F_N=\c$,
the last claim is proved in \eqref{first}. \end{proof}


\section{Spherical functions and spectral bounds}
We keep the notations set up in section \ref{closedh}. Let $0\le s\le 2$,
and consider the character $\chi_s$ on the subgroup $B:=AMN$ of $G$
defined by $$\chi_s(a_y m n)= y^{s} $$
where $a_y\in A$ is given as before, and $m\in M$ and $n\in N$.



The unitarily induced
representation $(\pi_s:=\op{Ind}_{B}^G \chi_s, V_s)$
 admits a unique $K$-invariant unit vector,
say $v_s$.

By the theory of spherical functions,
$$
f_s(g):=\langle \pi_s(g)v_s, v_s\rangle =\int_K v_s(kg)dk
$$
is the unique bi $K$-invariant function
of $G$
with $f_s(e)=1$ and with $\mathcal C f_s = s (2-s) f_s$ where $\mathcal C$ is the Casimir operator
of $G$.
Moreover, there exist  $c_s>0$ and $\e>0$ such that for all $y$ small
 \begin{equation}\label{sha} f_s(a_y)=c_s
\cdot y^{2-s}  (1+O(y^{\e }) )\end{equation}
by \cite[4.6]{GangolliVaradarajanbook}.

Since the Casimir operator is equal to the Laplace operator $\Delta$ on $K$-invariant functions,
this implies:
\begin{Thm}\label{as} Let $\G <G$ be a discrete subgroup.
 Let $\phi_s\in L^2(\G\ba G)^K\cap C^\infty(\G\ba G)$
satisfy $\Delta \phi_s= s(2-s) \phi_s$ and $\|\phi_s\|_2=1$.
Then there exist $c_s>0$ and $ \e>0$ such that for all small  $0<y<1$,
$$\langle a_y\phi_s, \phi_s\rangle_{L^2(\G\ba G)}=c_s\cdot y^{2-s} (1+O(y^{\e })) .$$
\end{Thm}

In the unitary dual of $G$, the spherical part
consists of the principal series and
the complimentary series. We use the parametrization of
 $s\in\{1 +i\br \}\cup [1, 2]$
so that $s=2$ corresponds to the trivial representation and the vertical line
$ 1 +i\br$
corresponds to the tempered spectrum.
Then the complimentary series is parametrized by $V_s$, $ 1 <  s\le 2$
defined before.

Let $\{X_i\}$ denote an orthonormal basis of the Lie algebra of $K$ with respect to an
$Ad$-invariant scalar product, and define $\omega:=1-\sum X_i^2$.
This is a differential operator in the center of the enveloping algebra of
$\op{Lie}(K)$ and acts as a scalar on each $K$- isotypic component of $V_s$.

\begin{Prop}\label{decay}
Fix $1<s_0 < 2$. Let $(V, \pi)$ be a representation of $G$
which does not weakly contain any  complementary series representation
$V_s$ with parameter $s\ge s_0$.
Then  for any $\e>0$, there exists $c_\e>0$ such that for any smooth vectors $w_1, w_2\in V$,
and $y<1$,
$$|\langle a_y w_1, w_2\rangle |\le c_\e \cdot  y^{2-s_0-\e}
\cdot  \|\omega(w_1)\|\cdot \| \omega (w_2)\|.
 $$
\end{Prop}
\begin{proof} (We refer to \cite{Shalom2000} for the arguments below)
As a $G$-representation, $\pi$ has a Hilbert integral decomposition
$\pi=\int_{z\in \hat G} \oplus ^{m_z} \rho_z d\nu(z)$
where $\hat G$ denotes the unitary dual of $G$, $\rho_z$ is
irreducible and $m_z$ is the multiplicity of $\rho_z$, and $\nu$ is the spectral measure
on $\hat G$. By the assumption on $\pi$, for almost all $z$, $\rho_z$ is either tempered
or isomorphic to $\pi_s$ for $1 \le s\le s_0$.
As $1 <3-s_0 < 2$, there exists
 the complementary series $(V_{3-s_0}, \pi_{3-s_0})$.
We claim that
the tensor product $\rho_z\otimes \pi_{3-s_0}$ is a tempered representation.
Recall that a unitary representation of $G$ is tempered if and only if
there exists a dense subset of vectors whose matrix coefficients are $L^{2+\e}$-integrable
for any $\e>0$.
If $\rho_z$ is tempered,  $\rho_z\otimes \pi_{3-s_0}$ is clearly tempered.
If $\rho_z$ is isomorphic to $\pi_s$ for some $1\le s\le s_0$,
and $v_z$ denotes the spherical vector of $\rho_z$ of norm one,
then
the matrix coefficient $g\mapsto \langle \rho_z(g) v_z, v_z\rangle$
is $L^{2/(2-s_0) +\e}$-integrable for any $\e>0$, by \eqref{sha} together with the fact that
 the Haar measure on $G$ satisfies $d(k_1a_y k_2)\asymp y^{-3} dk_1dy dk_2$ for all $0<y\le 1$.
Since $\rho_z$ is irreducible, it follows that
there exists a dense set of vectors whose matrix coefficients are  $L^{2/(2-s_0) +\e}$-integrable
 for any $\e>0$.
Similarly there exists a dense set of vectors in $V_{3-s_0}$
 whose matrix coefficients are  $L^{2/(s_0-1) +\e}$-integrable
 for any $\e>0$.
Hence by the H\"older inequality,
there exists a dense set of vectors in  $\rho_z \otimes \pi_{3-s_0}$
 whose matrix coefficients are  $L^{2 +\e}$-integrable
 for any $\e>0$,
implying that $\rho_z\otimes \pi_{3-s_0}$ is tempered.



Since $\pi \otimes \pi_{3-s_0} =\int_z \oplus ^{m_z} (\rho_z\otimes \pi_{3-s_0}) d\nu(z)$,
we deduce that $\pi\otimes \pi_{3-s_0}$ is tempered.

We now claim that
 for any $\e>0$, there is a constant $c_\e>0$ such that any $K$-finite unit vectors $w_1$
and $w_2$,
 we have \begin{equation}\label{mattt} |\langle a_y w_1, w_2\rangle |\le c_\e \cdot  y^{2-s_0-\e}
\cdot  \prod_i \sqrt{( \op{dim}\la  K w_i \ra)}.
\end{equation}
Noting that the $K$-span
of $w_i\otimes v_{3-s_0}$ has the same dimension as the $K$-span of $w_i$,
the temperedness of $\pi \otimes \pi_{3-s_0}$ implies that
for any $\e>0$, there exists a constant $c_\e>0$ such that
 \begin{align*}
 \la a_y (w_1\otimes v_{3-s_0}), (w_2\otimes v_{3-s_0})\ra &=\la a_y w_1, w_2\ra \cdot \la a_y v_{3-s_0} , v_{3-s_0} \ra\\ &
\le c_\e \cdot y^{1+\e}\cdot  \prod_i \sqrt{( \op{dim}\la  K w_i \ra)}.
 \end{align*}

As $\langle a_y v_{3-s_0}, v_{3-s_0}\rangle = c\cdot  y^{-1+s_0} (1 +O(y^\e_0))$ for some $\e_0>0$,
 the claim \eqref{mattt}  follows.
 Passing from the above bounds of \eqref{mat}  for
$K$-finite vectors to those for smooth vectors has been detailed in \cite[Pf. of Thm 6]{Maucourant2007}.
In particular, in the case of $G=\SL_2(\c)$, the above degree of Sobolev norm
suffices.
\end{proof}

\begin{Def}\label{sgamma} {\rm For a geometrically finite discrete subgroup $\G$ of $G$
with $\delta_\G >1$, we fix  $1 < s_\G <\delta_\G$ so that
 there is no
eigenvalue of $\Delta$ between $s_\G(2-s_\G)$ and  the base
eigenvalue $\lambda_0=\delta_{\G}(2-\delta_{\G})$ in $L^2(\G\ba G)$.}
\end{Def}

By the theorem of Lax-Phillips \cite{LaxPhillips},
the Laplace spectrum on $L^2(\G\ba G)^K$ has only finitely many eigenvalues outside the tempered spectrum.
Therefore $1< s_\G <\delta_\G$ exists. The maximum difference between $\delta_\G$ and $s_\G$
will be referred to as the spectral gap for $\G$.

Let $\{Z_1, \cdots, Z_6\}$ denote an orthonormal basis of the Lie algebra of $G$.
Let $\G <G$ be a discrete subgroup of $G$.
For $f\in C^\infty (\G\ba G)\cap L^2(\G\ba G)$, we consider the following
 Sobolev norm $\mathcal S_m(f)$:
$$  \mathcal S_m(f)=\max \{ \| Z_{i_1}\cdots Z_{i_\ell} (f) \|_2 :
1\le i_j\le 6 , 0\le \ell\le m\}.$$

\begin{Cor}\label{matco} Let $\G$ be a geometrically finite
discrete subgroup of $G$
with $\delta_\G >1$.
Then for any $\psi_1\in L^2(\G\ba G)\cap C^\infty(\G\ba G)^K$,
 $\psi_2\in  C_c^\infty(\G\ba G)$
 and $0<y<1$,
\begin{align*}
\langle a_y \psi_1, \psi_2\rangle &=\langle \psi_1, \phi_0\rangle
 \langle a_y \phi_0, \psi_2\rangle + O( y^{2-s_\G }\mathcal S_2(\psi_1)
\cdot \mathcal
 S_2(\psi_2) ).
\end{align*}
Here $\phi_0\in L^2(\Gamma\ba G)^K$ is the unique eigenfunction
of $\Delta$ with eigenvalue $\delta_\G(2-\delta_\G)$ with unit $L^2$-norm.
\end{Cor}
\begin{proof}  We have
$$L^2(\G\ba G)=W_{\delta_\G} \oplus V$$ where $W_{\delta_\G}$ is isomorphic to $V_{\delta_\G}$
as a $G$-representation and  $V$ does not contain
any complementary series $V_s$ with parameter $s >s_\G$.
Write $\psi_1=\langle \psi_1, \phi_0\rangle \phi_0+\psi_1^\perp$.
Since $\phi_0$ is the unique $K$-invariant vector in $W_{\delta_\G}$ up to
a constant multiple,
 we have $\psi_1^\perp\in V^K $. Hence
by Proposition \ref{decay}, for any $\e>0$
and
$y\le 1$,
\begin{align*} \langle a_y \psi_1, \psi_2\rangle
&=\langle \psi_1, \phi_0\rangle
 \langle a_y \phi_0, \psi_2\rangle
+
 \langle a_y \psi_1^\perp , \psi_2 \rangle
\\&= \langle \psi_1, \phi_0\rangle
 \langle a_y \phi_0, \psi_2\rangle
+O(y^{2-s_\G} \mathcal S_2 (\psi_2) \cdot \mathcal S_2 (\psi_2))
\end{align*}
since $\mathcal S_2(\psi_i^\perp) \ll \mathcal S(\psi_i)$.


\end{proof}

\section{Equidistribution of expanding closed horospheres with respect to the
 Burger-Roblin measure}\label{lhorof}
Let $\G<\PSL_2(\c)$ be a geometrically finite Kleinian group
with $\delta_\G>1$. Assume that $\G\ba \G N$ is closed.
Let $\mu$ denote the Haar measure on $G=NAK$ given by
$$d\mu(n_xa_y k)=
y^{-3} dxdydk$$ where $dk$ is the probability Haar measure on $K$.
We normalize $\phi_0$ so that $$\int_{\G\ba \bH^3}\phi_0 (x,y)^2 \frac{1}{y^3} dxdy =1 .$$

By Corollary \ref{ppoo}, we have
$$\int_{n_x\in (N\cap \G)\ba N} \phi_0(x, y)\; dx
=c_{\phi_0} y^{2-\delta_\G} +d_{\phi_0}y^{\delta_\G}$$
where $c_{\phi_0}>0$ and $d_{\phi_0}\ge 0$.

In this section, we aim to prove the following theorem.
\begin{Thm}\label{ma1}
For any $\psi\in C^\infty_c(\G\ba G)^K$,
\begin{align*}
\int_{n_x\in ( N\cap \G)\ba N} \psi (n_x a_y)\; dx
&=\langle \psi, \phi_0\rangle\cdot
 c_{\phi_0}  \cdot y^{2-\delta_{\G}}  (1 +O(y^{\frac{2}{7}( \delta_{\G}-s_\G)} ))
\end{align*}
 where the implied constant depends
only on the Sobolev norms of $\psi$, the volume of
  $N(\text{supp}
   (\psi))$ and the volume of an open subset of $(N\cap \G)\ba N$ which properly covers $\Lambda_N(\G)$.
\end{Thm}


Most of this section is devoted to a proof of Theorem \ref{ma1}.
\begin{Def} {\rm For a given $\psi\in C^\infty (\G\ba G)^K$ and
$\eta\in C_c((N\cap \G)\ba N)$, define the function
$\I_\eta(\psi)$ on $G$ by
$$\I_{\eta}(\psi)(a_y):=\int_{n_x\in (N\cap \G)\ba N} \psi(n_x a_y) \eta  (n_x)\; dx .$$}
\end{Def}

We denote by $N^-$ the strictly lower triangular subgroup of $G$:
$$N^-:=\{\begin{pmatrix} 1 & 0\\ x& 1\end{pmatrix}:x\in \c\} .$$

The product map $N \times A\times N^-\times M \to G$ is a diffeomorphism at a neighborhood
of $e$. Let $\nu$ be a smooth measure on $AN^-M$ such that $
dn \otimes \nu ( an^-m)=d\mu$.

Fix a left-invariant Riemannian metric $d$ on $G$ and denote by
  $U_\e$ the ball of radius $\e$ about $e$ in $G$.

\begin{Def} \label{def} {\rm \begin{itemize}
\item We fix a non-negative function $\eta\in C_c^\infty((N\cap \G)\ba N)$ with
 $\eta=1$ on a bounded open subset of $\mathcal F_N$
which properly covers
  $\Lambda_N(\G)$.
  \item Fix $\e_0 >0$ so that
for the $\e_0$-neighborhood $U_{\e_0}$ of $e$, the
multiplication map
$$\supp (\eta) \times (U_{\e_0}\cap AN^-M)
\to \supp (\eta) (U_{\e_0}\cap AN^-M)\subset  \G\ba G$$ is a bijection onto
its image.
\item  For each $\e<\e_0$, let $ r_\e $ be a non-negative smooth function
in $AN^-M$ whose support is contained in
$$ W_\e:=(U_\e \cap A)(U_{\e_0}\cap N^-) (U_{\e_0}\cap M)$$ and $\int_{W_{\e}}  r_\e \; d\nu  =1$.

\item We define the following function $\rho_{\eta, \e}$ on $\G\ba G$ which is $0$ outside $\op{supp}(\eta)U_{\e_0}$ and  for $g=n_xan^-m \in \op{supp}(\eta)
(U_{\e_0}\cap AN^-M)$,
$$\rho_{\eta, \e}(g) := \eta (n_x) \otimes r_\e(an^- m) .$$
            \end{itemize}}
\end{Def}

Recall from Proposition \ref{ppo} that
$$
 \phi_0^N(a_y)=\I_\eta( \phi_0)(a_y) +O(y^\gd) .$$

\begin{Prop}\label{hugh} We have for all small $0<\e\ll \e_0$ and for all $0<y < 1$,
 $$ \phi_0^N(a_y) = \langle a_y\phi_0, \rho_{\eta, \e} \rangle_{L^2(\G\ba G)}
 + O_\eta(  \e\cdot  y^{2-\delta_\G})+O(y^{\delta_\G})  $$
 where the first implied constant depends on the Lipschitz constant for $\eta$.
\end{Prop}
\begin{proof}
Let $h=a_{y_0} n^-_{x} m \in W_\e$.
Then for $n\in N$ and $y>0$,
we have $$nha_y =n a_{y y_0} n^-_{yx} m .$$

As the product map $A\times N\times K\to G$ is a diffeomorphism and
hence a bi-Lipschitz map at a neighborhood of $e$, there exists $\ell >0$ such that
the $\e$-neighborhood of $e$ in $G$
is contained in the product  $A_{\ell \e}N_{\ell\e} K_{\ell\e}$ of $\ell\e$-neighborhoods
for all small $\e>0$.

Therefore we may write
$$n_{yx}^-=  a_{y_1} n_{x_1} k_1\in A_{\ell y \e_0}N_{\ell y\e_0} K$$
and hence
\begin{multline*}
 nha_y= na_{yy_0y_1} n_{x_1}k_1 m \\= n(a_{yy_0y_1} n_{x_1}  a_{yy_0y_1}^{-1}) a_{yy_0y_1}    k_1 m=
n (n_{x_1 yy_0y_1}) a_{yy_0y_1}  k_1 m .
\end{multline*}

As $\phi_0$ is $K$-invariant and $dn$ is $N$-invariant,
\begin{align*} \int_{N\cap \G\ba N} \phi_0(n h a_y) \cdot \eta(n)  dn&
=\int_{N\cap \G\ba N} \phi_0(n (n_{x_1 yy_0y_1}) a_{yy_0y_1} )  \cdot \eta(n) dn\\
&=\int_{N\cap \G\ba N} \phi_0(n a_{yy_0y_1} ) (\eta(n) +O(\e)) dn
\end{align*}
as $\eta(n)-\eta(nn')=O_\eta(\e)$ for all $n\in N$ and $n'\in N \cap U_\e$.
By Corollary \ref{ppoo}, we deduce

\begin{align*} &\int_{N\cap \G\ba N} \phi_0(n h a_y)  \cdot \eta(n) dn\\
&=  \int_{N\cap \G\ba N} \phi_0(n a_{yy_0y_1} ) \eta(n) dn  +O_\eta(\e\phi_0^N(a_{yy_0y_1}))
 \\&= c_{\phi_0} (yy_0y_1)^{2-\delta_{\G}} +d_{\phi_0} (yy_0y_1)^{\delta_{\G}}+
 O_\eta(\e y^{2-\delta_\G}) \\
&= c_{\phi_0} y^{2-\delta_{\G}} (1 +O(1- (y_0y_1)^{2-\delta_{\G}}))+  O_\eta (\e y^{2-\delta_\G})
+O(y^{\delta_\G}) \\
&= c_{\phi_0} y^{2-\delta_{\G}} (1 +O(\e))
     +O(y^{\delta_\G})  \end{align*}
 as $|y_0-1|=O(\e)$ and $|y_1-1|=O(y\e)$.

As $\int r_\e d\nu(h)=1$, we deduce
\begin{align*}
 \langle a_y\phi_0, \rho_{\eta, \e} \rangle& =\int_{W_\e} r_\e(h)
\left(\int_{N\cap \G\ba N}\phi_0(nha_y) \eta(n)  \; dn\right)
\,d\nu(h) \\&
=
c_{\phi_0} y^{2- \delta_{\G}} (1+O_\eta(\e)) +O(y^{\delta_{\G}}) \\ &=
\phi_0^N (a_y) + O_\eta (\e y^{2-\delta_\G} )  +O(y^{\delta_{\G}}).
\end{align*}
\end{proof}




%


\begin{Lem}\label{approx} For $\psi\in C^\infty_c(\G\ba G)^K$,
there exists $\hat \psi\in C^\infty_c(\G\ba G)^K$
such that \begin{enumerate}

\item

for all small $\e>0$ and $h\in U_\e$,
$$|\psi(g)-\psi(gh)|\le \e \cdot \hat \psi(g)\quad \text{ for all $g\in \Gamma\ba G$}. $$
\item $\mathcal S_{m}(\hat  \psi) \ll \mathcal S_{5}(\psi)$ for any $m\in \n$, where
the implied constant depends only on $\supp(\psi)$.          \end{enumerate}
\end{Lem}
\begin{proof}
Fix $\e_0>0$. Let $f_0\in C_c^\infty(\G\ba G)^K$ such that
$f_0(g)=1$ for all $g\in \text{supp}(\psi)U_{\e_0}^{-1}K$ and
$f_0(g)=0$ for all $g\in \Gamma\backslash G -\text{supp}(\psi)U_{2\e_0}^{-1}K$.

Set $C_\psi:=\sup_{g\in \text{supp}(\psi)}\sum_{i=1}^6|X_i(\psi)(g)|$.
Then there exists a constant $c_0\ge 1$ such that
 for all $g\in \Gamma\ba G$ and $h\in U_\e$ for $\e<\e_0$,
$$|\psi (g)-\psi
(gh)|\le \e \cdot  c_0 C_\psi. $$

Hence if we define $\hat \psi\in C_c^\infty(\Gamma\ba G)^K$ by
$\hat \psi (g)=c_0 C_\psi f_0(g)$ for $g\in \Gamma \ba G$,
then (1) holds.

Now by the Sobolev imbedding theorem (cf. \cite[Thm. 2.30]{Aubinbook}),
we have $$C_\psi\le \mathcal S_5(\psi).$$
Since
 $\mathcal S_m(\hat \psi)\ll C_\psi $,
this proves (2).
\end{proof}


\begin{Prop}\label{mat} Let $\psi\in C^\infty(\G\ba G)^K$.
Then for any $0<y<1$ and any small $\e >0$,
$$|\I_{\eta}(\psi)(a_y) -\langle a_y \psi, \rho_{\eta, \e}  \rangle | \ll
(\e+y) \cdot \I_\eta (\hat \psi) (a_{y}) .  $$
\end{Prop}
\begin{proof}
If $an^- m\in W_\e=(U_\e\cap A)(U_{\e_0}\cap N^-)(U_{\e_0}\cap M)$, then
$$(an^- m)a_y= a_ya (a_{y^{-1}} n^- a_y) m \in a_y W_\e$$
since $a_{y^{-1}}$ contracts $N^-$ by conjugation as $0<y<1$.

As $\psi$ is $M$-invariant, for any $h=an^-m \in W_\e$,
there exists an $h'\in (U_\e\cap A)(U_{y\e_0}\cap N^-)$ such that
 $$
 |\psi(na_y )- \psi(n   h a_y) |=|\psi(na_y )- \psi(n  a_y h') |
 \ll  \hat \psi(na_y) (\e +y) .
 $$

Hence
\begin{align*}
|\psi(na_y) -\int_{h\in W_\e} \psi(n h a_y)r_\e (h) d\nu(h) |  \ll \hat \psi(na_y) (\e +y)
.\end{align*}
Therefore
$$
 |\I_{\eta}(\psi)(a_y) -\langle a_y \psi, \rho_{ \e}\rangle_{L^2(\G\ba G)} |
 \ll
(\e+y) \cdot \int_{( N\cap \G)\ba N}  \hat \psi(na_y) \eta(n) dn  .
$$

\end{proof}


\vs
\noindent{\bf Proof of Theorem \ref{ma1}:}
Recall that $\eta, \e_0, \rho_{\eta, \e}=\eta\otimes r_\e$ are as in Def \ref{def}.
For simplicity, we set $\rho_\e=\rho_{\eta, \e}$. Noting that $r_\e$ is essentially
an $\e$-approximation only in the $A$-direction,
we obtain that $\mathcal S_2 (\rho_\e)=O_\eta (\e^{-5/2})$.

We may further assume that
$\eta=1$ on $N(\op{supp}(\psi))$ by Corollary \ref{suppcom}
so that
$$\int_{( N\cap \G)\ba N} \psi (n a_y)\; dn=\I_\eta(\psi) (a_y) .$$

By Proposition \ref{ppo}, we also have
$$  \phi_0^N (a_y)=\I_\eta(\phi_0)(a_y)+O(y^{\delta_{\G}}) .$$

Set $p=5/2$.
Fix $\ell\in \n$ so that
 $$\ell >\frac{(2-\delta_{\G})(p+1)}{(\delta_{\G}-s_\G)} .$$

Setting $\psi_0(g):=\psi (g)$,
we define for $1\le i\le \ell $,
inductively
$$\psi_i (g) := \hat\psi_{i-1}(g) $$
where $\hat\psi_{i-1}$ is given by Lemma \ref{approx}.


Applying
Proposition \ref{mat} to each $\psi_i$, we obtain for $0\le i\le \ell -1$
$$
\I_{\eta} ( \psi_i)(a_y)=
 \langle a_y   \psi_i , \rho_\e\rangle
+O((\e +y) \I_{\eta} ( \psi_{i+1})(a_y) )
$$
and
$$
\I_{\eta }( \psi_\ell)(a_y)=
 \langle a_y   \psi_\ell , \rho_\e\rangle
+O_\eta((\e +y) \mathcal S_2({\psi_\ell }))
$$
where the implied constant in the $O_\eta$ notation depends on $\int \eta \,dn$.

Note that  by Corollary \ref{matco},
we have for each $1\le i\le \ell $,
\begin{align*}
 \langle a_y   \psi_i , \rho_\e\rangle &=
\la \psi_i, \phi_0\ra \la a_y \phi_0, \rho_\e\ra  + O(y^{2-s_\G}
\mathcal S_2 (\rho_\e) \mathcal S_2(\psi_i))\\
&=O(\la a_y \phi_0, \rho_\e\ra  \cdot \|\psi_i\|_2 )
+ O(y^{2-s_\G}
\mathcal S_2 (\rho_\e) \mathcal S_2(\psi_i))
\\ &=O(\la a_y \phi_0, \rho_\e\ra  \cdot \mathcal S_5 (\psi))
+ O(y^{2-s_\G}
\e^{-p} \mathcal S_{5}(\psi)).
\end{align*}

Hence for any $y<\e$,
\begin{align*}
&\I_{\eta} (\psi)(a_y) = \langle a_y\psi, \rho_\e\rangle +
 \sum_{k=1}^{\ell -1} O( \langle a_y\psi_k, \rho_\e\rangle (\e+y) ^k)  +  O_\eta(
\mathcal S_{5} ( \psi) (\e+y) ^\ell )
\\&=  \langle a_y \psi, \rho_\e\rangle
+ O( \langle a_y\phi_0, \rho_\e\rangle \e \mathcal S_{5 }(\psi) )+
O( \e \mathcal S_{5} ( \psi)  y^{2-s_\G} \e^{-p} )+
O(\mathcal S_{5 } ( \psi) \e^\ell )
\\&=
\langle \psi, \phi_0\rangle
 \cdot \langle a_y \phi_0, \rho_\e\rangle + O(
 \langle a_y\phi_0, \rho_\e\rangle \e  )
+O(y^{2-s_\G} \e^{-p}  ) + O(\e^\ell ) \\
&=
\langle \psi, \phi_0\rangle
 \cdot \phi_0^N(a_y) + O(y^{\delta_\G}) +O (\e y^{2-\delta_\G} )
 +O (y^{2-s_\G} \e^{-p}  )  + O (\e^\ell )
 \end{align*}
by Proposition \ref{hugh}, where the implied constants depend on the Sobolev norm
 $\mathcal S_{5}(\psi)$
and $\int \eta \,dn$.

Equating the two error terms $O(\e y^{2-\delta_\G})$ and
 $O(y^{2-s_\G}\e^{-p})$  gives the choice $\e=y^{(\delta_{\G}-s_\G)/(p+1)}$.
By the condition on $\ell$, we then have $\e^{\ell }\ll y^{2-\delta_{\G} +\frac 27 (\delta_{\G}-s_\G)}$.
Hence we deduce:
$$\int_{( N\cap \G)\ba N} \psi (n a_y)\; dn= \I_{\eta}(\psi)(a_y)
=
\langle \psi, \phi_0\rangle\cdot  \phi_0^N(a_y)\cdot
(1 +O(y^{ \frac 27 (\delta_\G-s_\G)}).
$$
Note that the implied constant depends on the Sobolev norm $\mathcal S_{5}(\psi)$ of $\psi$
and the $L^1$-norm $\int \eta \,dn$, which in turn depends only on the volumes
$N(\text{supp}(\psi))$ and any open subset of $(N\cap \G)\ba N$ which properly covers $\Lambda_N(\G)$.

\vs\vs

\noindent{\bf Burger-Roblin measure $\hat \mu$: } In identifying $\partial_\infty(\bH^3)$ with $K/M$,
we may define the following measure $\hat \mu$ on $\op{T}^1(\bH^3)=G/M$:
for $\psi\in C_c(G/M)$,
$$\hat \mu(\psi)=\int_{k\in K}\int_{a_yn_x\in AN} \psi(ka_yn_x) y^{\delta -1} dydxd\nu_o(k)$$
 where we consider the Patterson Sullivan measure $\nu_o$ in section \ref{phizeroo} as a measure on $K$
 via the projection $K\to K/M$: for $f\in C(K)$,
 $\nu_o(f)=\int_{k\in K/M} \int_{m\in
 M} f(km)dm d\nu_o(k)$ for the probability invariant
 measure $dm$ on $M$.

By the conformal property of $\nu_o$,
the measure $\hat \mu$ is left $\G$-invariant and hence
induces a Radon measure on $\op{T}^1(\G\ba \bH^3)$ via the canonical
projection.
\begin{Lem} For a $K$-invariant function
$\psi\in C_c(G)$, we have
\begin{equation}\label{nuhat}
\hat \mu(\psi)=\la \psi, \phi_0\ra .
 \end{equation}
\end{Lem}
\begin{proof}
It is easy to compute that for the base point $o=(0,0,1)\in \bH^3$,
 $$\beta_\infty(a_yn_x o, o)= -\log y$$
for any $0<y\le 1$ and $x\in \c$.
We note that $dg=y^{-1}dydxdk$ for $g=a_yn_xk$
is a Haar measure on $G$.
Therefore, using $\psi$ is $K$-invariant,
\begin{align*}&
\hat \mu(\psi)=\int_{K}\int_{AN} \int_{k_0\in K}\psi(ka_yn_xk_0)d({k_0}) y^{\delta -1} dydx d\nu_o(k)
\\&=\int_{g\in G}\int_{K}\psi(kg) e^{-\delta\beta_\infty(go,o)} d\nu_o(k) dg
\\&=\int_{G}\psi(g) \int_{k\in K/M} e^{-\delta\beta_\infty(k^{-1} go,o)} d\nu_o(k) dg
\\&=\int_{G}\psi(g) \int_{K/M} e^{-\delta\beta_{k(\infty)}(go,o)} d\nu_o(k) dg
=\la \psi, \phi_0\ra
\end{align*}
as $\phi_0(g)=\int_{K/M} e^{-\delta\beta_{k(\infty)}(go,o)} d\nu_o(k)$.
\end{proof}

Generalizing Burger's result \cite{Burger1990},
 Roblin \cite[Thm 6.4]{Roblin2003} proved:
\begin{Thm}\label{roblin}  The measure $\hat \mu$ is, up to constant multiple, the unique Radon measure
on $\op{T}^1(\G\ba \bH^3)$ invariant for the horospherical foliations whose support is
$\hat{\Omega}_\G$ and which gives zero measure to closed horospheres.
\end{Thm}
We call the measure $\hat \mu$ the Burger-Roblin measure.

In Appendix \ref{app}, the following theorem is deduced from Theorem \ref{ma1}.
\begin{Thm}\label{ft}
For $\psi\in C_c(\op{T}^1(\G\ba \bH^3))$,
$$  \int_{n_x\in ( N\cap \G)\ba N} \psi (n_x a_y)\; dx \sim  c_{\phi_0}\cdot \hat\mu(\psi) \cdot y^{2- \delta_{\G}} \quad\text{as $y\to 0$.}$$
\end{Thm}\vs

\section{Orbital counting for a Kleinian group}
Let $\iota: G=\PSL_2(\c)\to \SO_F(\br)$ be a real linear representation where $F$
is a real quadratic form in $4$ variables of signature $(3,1)$. Let $\G<G$ be a geometrically finite torsion-free discrete subgroup with $\delta_\G>1$.
Let $v_0\in \br^4$ be a non-zero (row) vector with $F(v_0)=0$
such that the orbit $v_0\G$ is discrete in $\br^4$.

Since the orthogonal group  $\O_F(\br)$ acts transitively on
the set of non-zero vectors of the cone $F=0$, there exists $g_0^t\in \O_{F}(\br)$ such that
  the stabilizer of $g_0^t v_0 ^t$  is equal to $N^-M$
 where $N^-$ is the strictly lower triangular subgroup.
In fact, $g_0^tv_0^t$ is unique up to homothety.
 Set $$\G_0=g_0^{-1}\G g_0 .$$  As $v_0\G$ is discrete, it follows
that $\G_0 \ba \G_0 NM$ is closed.
Hence by Lemma \ref{stab},
the orbit $\G\ba \G g_0 N $ is closed, equivalently $\G_0\ba \G_0 N$ is closed.

Denote by $\phi_0\in L^2(\G_0\ba \bH^3 ) $ the unique positive eigenfunction of $\Delta$ with eigenvalue
$\delta_\G(2-\delta_\G)$ and of unit $L^2$-norm $\int_{\G_0\ba \bH^3}\phi_0(x,y)^2 y^{-3} dx dy=1$.
By Corollary \ref{ppoo}, we have
$$\phi_0^N(a_y) =c_{\phi_0 }y^{2-\delta_\G} +d_{\phi_0} y^{\delta_\G} $$
where $c_{\phi_0}>0$ and $d_{\phi_0}\ge 0$.
Recall that the Patterson-Sullivan measure $\nu_o$
on $K$ which is normalized so that
\begin{equation*} \phi_0(x,y)=\int_{u\in\Lambda(\G)\setminus \{\infty\}}
 \left({(\|u\|^2+1) y \over \|x-u\|^2+y^2}\right)^{\delta_\G}d\nu_o(u) .
\end{equation*}

\begin{Thm}\label{reduction4}
For any norm $\|\cdot \|$ on $\br^4$, we have, as $T\to \infty$,
$$\# \{v\in v_0\G: \| v\|<T\}\sim  \delta_\G^{-1} \cdot c_{\phi_0} \cdot \left(\int_{k\in K}\|v_0 (g_0k^{-1} g_0^{-1})\|^{-\delta _\G} d\nu_o(k) \right) \cdot T^{\delta_{\G}}
. $$

If $\|\cdot \|$ is $g_0 Kg_0^{-1}$-invariant, then there exists $\e>0$ such that
$$\# \{v\in v_0\G: \| v\|<T\}=  \delta_\G^{-1}\cdot
\phi_0(e) \cdot c_{\phi_0}\cdot
\|v_0 \|^{-\delta _\G} \cdot T^{\delta_{\G}}
 (1+O(T^{-\e}))$$
 where $\e$ depends only on the spectral gap $\delta_\G -s_\G$
 and the implied constant depends only on $\Lambda_N(\G)$.
\end{Thm}




By replacing $\G$ with $\G_0=g_0^{-1}\G g_0$, we may assume henceforth
 that $g_0=e$, and thus
the stabilizer of $v_0$ in $G$ is $NM$.
By Lemma \ref{stab}, the stabilizer of $v_0$ in $\G$ is simply $\G\cap N$.

Note that $N^-M$ fixes $v_0^t$, and that the highest weight $\beta$
 of the (irreducible) representation
of $\iota$ is given by $\beta(a_y)=y^{-1}$.
It follows  that $\iota(a_y) v_0^t =y^{-1}v_0^t$, and hence  $v_0a_y=y^{-1} v_0$.

Set $$B_T:=\{v\in v_0 G: \|v\|<T\}.$$
Define the following function
on $\G\ba G$:
$$
F_T(g):=\sum_{\g\in(N\cap \G)\ba \G}\chi_{B_T} (v_0\g g).
$$

Since  $v_0\G$ is discrete, $F_T(g)$ is well-defined and
$$F_T(e)=
\# \{v\in v_0\G: \|v\|<T\}  .$$

We use the notation: for $\psi\in C_c(\G\ba G)$,
$$\psi^N(a_y):=\int_{(N\cap \G)\ba N} \psi(na_y) \, dn .$$

\begin{Lem}\label{ftweak}
For any $\psi\in C_c(\G\ba G)$ and $T>0$,
$$
\la F_T, \psi \ra = \int_{k\in M\ba K} \int_{y> T^{-1} \|v_0 k\|}
  \psi_k ^N(a_y )  y^{-3}  dydk $$
where $\psi_k(g)= \int_{m\in M} \psi ( g mk )dm$. \end{Lem}
\begin{proof}
Observe:
\begin{align*}
\la F_T, \psi \ra &=
\int_{\G\ba G} \sum_{\g\in (N\cap \G)\ba \G } \chi_{B_T}(v_0\g g)
\psi (g) dg
\\
\nonumber
&=
\int_{g \in (N\cap \G)\ba G} \chi_{B_T}(v_0g)
\psi (g) dg
\\
\nonumber
&=\int_{k\in M\ba K} \int_{y:\|v_0a_y k\|<T}
 \int_{n_xm \in  (N\cap \G)\ba NM} \psi(n_x m a_y k) y^{-3}  \,dx\,  dm\, dy \, dk
\\
\nonumber
&=\int_{k\in M\ba K} \int_{y> T^{-1} \|v_0 k\|}
 \left(\int_{n_x \in  (N\cap \G)\ba N} \int_{m\in M} \psi (n_xa_y mk )\, dm \, dx \right)
y^{-3}  dy\, dk
\\
\nonumber
&=\int_{k\in M\ba K} \int_{y> T^{-1} \|v_0 k\|}
 \left(\int_{n_x \in  (N\cap \G)\ba N}  \psi_k (n_xa_y )  dx \right)
y^{-3}  dy\, dk
\\
\nonumber &=\int_{k\in M\ba K} \int_{y> T^{-1} \|v_0 k\|}
  \psi_k^N (a_y ) \,
y^{-3}  dy\, dk
.
\end{align*}

\end{proof}

Define a function $\xi_{v_0}:K\to \br$
by $$\xi_{v_0}(k)=\|v_0 k\|^{-\delta_\G}.$$
For $\psi\in C_c(\G\ba G)$,
the convolution $\xi_{v_0}*\psi$ is then given by
$$\xi_{v_0}*\psi (g)=\int_{k\in K}\psi(gk) \|v_0k\|^{-\delta_\G}
dk .$$

\begin{Cor} \label{smc}
  For any $\psi\in C_c(\G\ba G)$, we have, as $T\to \infty$,
$$\la F_T, \psi \ra \sim  \delta_\G^{-1} \cdot c_{\phi_0}
 \cdot T^{\delta_{\G}} \cdot \hat\mu(\xi_{v_0}*\psi).
$$
 For $\psi\in C_c^\infty(\G\ba G)^K$ and $\|\cdot \|$ $K$-invariant,
 $$\la F_T, \psi \ra =  \langle  \psi , \phi_0 \rangle \cdot
\delta_\G^{-1} \cdot c_{\phi_0}
 \cdot T^{\delta_{\G}} \cdot  \|v_0 \|^{-\delta_{\G}}  (1+O(T^{-\frac{2}{7} (\delta_\G -s_\G)}))
$$
where $s_\G$ is as in Def. \ref{sgamma}
and the implied constant depends only on the Sobolev norm
of $\psi$, $\Lambda_N(\G)$ and $N(\text{supp}(\psi))$.
\end{Cor}
\begin{proof}
Note that for any $\psi\in C_c(\G\ba G)$ and $k\in K$,
  the function $\psi_k$, defined in Lemma \ref{ftweak}, is $M$-invariant.
Applying Theorem \ref{ft} to $\psi_k$, we obtain that as $y\to 0$,
\begin{equation}\label{errort}
\int_{n_x \in  (N\cap \G)\ba N}  \psi_k (n_xa_y )  dx \sim c_{\phi_0}\cdot
 y^{2-\delta_{\G}}
\cdot \hat\mu(\psi_k) .
 \end{equation}

Hence by applying Lemma \ref{ftweak}, inserting the definition of
$\psi_k$, and evaluating the $y$-integral, we get
\begin{align*} \la F_T, \psi\rangle &\sim c_{\phi_0}\cdot
\int_{M\ba K}\int_{y>T^{-1}\|v_0k\|}y^{-1-\delta_{\G}} \hat \mu(\psi_k)
\; dy\, dk\\
 &= \delta_\G^{-1}\cdot c_{\phi_0}\cdot
T^{\delta_{\G}} \int_{M\ba K}\int_M \hat \mu( \psi (g mk) )\cdot \|v_0 k\|^{-\delta_{\G}} \; dm dk\\
&=\delta_\G^{-1}\cdot
c_{\phi_0}\cdot
T^{\delta_{\G}}\cdot \hat\mu ( \xi_{v_0}* \psi) ,
\end{align*}
proving the first claim.

Now suppose that
both $\psi$ and the norm $\|\cdot \|$
are $K$-invariant.

As $\psi_k=\psi$, by
 Theorem \ref{ma1},
 we can
replace \eqref{errort}
by an asymptotic formula with power savings error term:
$$\int_{n_x \in  (N\cap \G)\ba N}  \psi (n_xa_y )  dx = c_{\phi_0} y^{2-\delta_{\G}}\la \psi, \phi_0  \ra (1+O(y^{\frac{2}{7} (\delta_\G -s_\G)}))$$
 and the implied constant depends on the Sobolev norm
of $\psi$ and  $\Lambda_N(\G)$.

On the other hand,
 $$\xi_{v_0}* \psi=\|v_0\|^{-\delta_\G}\cdot \psi , $$
 and hence
$$\hat\mu (\xi_{v_0}* \psi)
= \|v_0 \|^{-\delta_{\G}} \cdot\la \psi, \phi_0\ra .$$

Therefore
\begin{align*} \la F_T, \psi\rangle &=
c_{\phi_0}
\int_{y>T^{-1}\|v_0\|} y^{-1-\delta_{\G}}\la \psi, \phi_0  \ra
 (1+O(y^{\frac{2}{7} (\delta_\G -s_\G)})) \; dy \\
 &= \delta_\G^{-1}\cdot
 c_{\phi_0}
 \cdot
T^{\delta_{\G}}\cdot
\|v_0 \|^{-\delta_{\G}} (1+O(T^{-\frac{2}{7} (\delta_\G -s_\G)}))
\end{align*}
 where the implied constant depends only on the Sobolev norm
of $\psi$ and the set $\Lambda_N(\G)$.
\end{proof}

\begin{Thm}\label{maincount}
As $T\to \infty$,
$$F_T(e)\sim  \delta_\G^{-1} \cdot c_{\phi_0} \cdot
\left( \int_{K} \|v_0 k^{-1}\|^{-\delta_{\G}} \; d\nu_o(k) \right) \cdot   T ^{\delta_{\G}} .$$
 If $\|\cdot \|$ is $K$-invariant, then for some $\e>0$
 (depending only on the spectral gap $\delta_\G -s_\G$),
$$F_T(e)=  \delta_\G^{-1} \cdot  \phi_0(e) \cdot c_{\phi_0} \cdot
\|v_0 \|^{-\delta_{\G}}  \cdot   T ^{\delta_{\G}} (1+O(T^{-\e})) .$$
where  the implied constant depends only on $\Lambda_N(\G)$.
\end{Thm}
\begin{proof} For all small $\e>0$,
we choose a symmetric $\e$- neighborhood $U_{\e}$ of $e$ in $G$, which injects to $\G\ba G$, such that for all $T\gg 1$ and all $0< \e\ll 1$,
$$ B_T U_\e \subset B_{(1+\e ) T}\quad\text{and}\quad B_{(1-\e)T}
\subset \cap_{u\in U_\e} B_T u
.$$

For $\e>0$, let $\phi_\e\in C_c^\infty(G)$ denote a non-negative function
 supported on $U_\e$ and
 with $\int_{G} \phi_\e\, dg=1$.
We lift $\phi_\e$ to $\G\ba G$ by
averaging over $\G$:
$$\Phi_\e(\G g)=\sum_{\gamma\in \G} \phi_\e(\gamma g) .$$


Then
\begin{equation}\label{ineq1} \langle F_{(1-\e)T}, \Phi_\e\rangle\le
 F_T(e)\le \langle F_{(1+\e) T}, \Phi_\e\rangle
.\end{equation}

Note that
\begin{equation}\label{ftfinal}
\langle F_{(1\pm \e) T}, \Phi_\e\rangle
\sim \delta_\G^{-1}\cdot
c_{\phi_0}
 \cdot  (T(1\pm \e))^{\delta_{\G}} \cdot\hat\mu(\xi_{v_0}* \Phi_\e)
.\end{equation}

Considering the function $R_{v_0}:G\to \br$ given by
$$R_{v_0}(g):= y^{\delta_\G}
\xi_{v_0}(k_0) $$
for $g=a_yn_xk_0\in ANK$,
we have \begin{align*}
\hat\mu(\xi_{v_0}* \Phi_\e)
&=\int_{g\in U_\e} \int_{k_0\in K} \phi_\e(gk_0)\xi_{v_0}(k_0) d(k_0)
d\hat\mu(g)\\
&=\int_{ka_yn_x\in KAN\cap U_\e} \int_{k_0\in K} \phi_\e(ka_yn_x k_0)\xi_{v_0}(k_0) y^{\delta -1} d(k_0)dydx
d\nu_o(k) \\
&=\int_{k\in K}\int_{g\in G} \phi_\e(kg) R(g) dg d\nu_o(k)
\\ &=\int_{g\in U_\e}\phi_\e(g) \int_{k\in K}R(k^{-1} g) d\nu_o(k) dg
\\
&= \int_{k\in K}R(k^{-1})d\nu_o(k) +O(\e) \\
& =
\int_{k\in K/M}\|v_0 k^{-1}\|^{-\delta_\G} d\nu_o(k) + O(\e ).
\end{align*}
 as $\int_G \phi_\e \; dg =1$, where  the implied constant depends only
 on the Lipschitz constant for $R$.

As $\e>0$ is arbitrary, we deduce that
$$F_T(e)\sim  \delta_\G^{-1} \cdot c_{\phi_0} \cdot
\int_{k\in K}\|v_0 k^{-1}\|^{-\delta_\G} d\nu_o(k)
\cdot   T ^{\delta_{\G}} .$$

If $\|\cdot \|$ is $K$-invariant, we may take both $U_\e$ and $\phi_\e$ to be $K$-invariant. Hence by Corollary \ref{smc},
we may replace \eqref{ftfinal} by
\begin{equation*}
\langle F_{(1\pm \e) T}, \Phi_\e\rangle
= \delta_\G^{-1}\cdot c_{\phi_0}\cdot  \|v_0 \|^{-\delta_{\G}}  \cdot  (T(1\pm \e))^{\delta_{\G}} \cdot\la \phi_0, \Phi_\e \ra\cdot  (1+O(T^{-\frac{2}{7}(\delta_\G-s_\G)}))
\end{equation*}
where the implied constant depends on $\mathcal S_{5} (\Phi_\e)=\mathcal S_{5}(\phi_\e)$, and the sets $\Lambda_N(\G)$ and
 $N(U_\e):=\{[n]\in (N\cap \G)\ba N:\G n A\cap U_\e \ne\emptyset\}$.


Since
\begin{align*}
\langle \phi_0, \Phi_\e\rangle& =\phi_0(e) +  O(\e),
\end{align*}
 we have \begin{align*}
\langle F_{(1\pm \e) T}, \Phi_\e\rangle
&= \delta_\G^{-1}\cdot c_{\phi_0} \cdot \|v_0 \|^{-\delta_{\G}}  \cdot  T^{\delta_{\G}} \cdot \phi_0 (e)
+ O(\e T^{\delta_\G}) + O(\e^{-q} T^{\delta_\G -\frac{2}{7}(\delta_\G-s_\G)})
\end{align*}
for some $q>0$ depending on $\mathcal S_{5 }(\phi_\e)$.
Therefore by setting $\e^{1+q} =T^{ -\frac{2}{7}(\delta_\G-s_\G)}$,
$$F_T(e) =
\delta_\G^{-1}\cdot c_{\phi_0}\cdot \|v_0 \|^{-\delta_{\G}}  \cdot  T^{\delta_{\G}} \cdot \phi_0 (e)\cdot
(1+ O( T^{-\e'})) $$
 for $\e'=\frac{2}{7(1+q)} (\delta_\G -s_\G)$, where the implied constant depends only on $\Lambda_N(\G)$.

\end{proof}

\renewcommand{\cA}{\mathsf A}

\section{The Selberg sieve and circles of prime curvature}

\subsection{Selberg's sieve} We first recall the Selberg upper bound sieve. Let $\cA$ denote the finite sequence of real non-negative numbers $\mathsf A=\{a_n\}$, and let $P$ be a finite product of distinct primes. We are interested in an upper bound for the quantity
$$
S(\mathsf A,P):=\sum_{(n,P)=1}a_n.
$$
To estimate $S(\mathsf A,P)$ we need to know how $\cA$ is distributed along certain arithmetic progressions. For $d\nmid P$, define
$$
\cA_d:=\{a_n\in\cA:n\equiv0(d)\}
$$
and set
$
|\cA_d|:=\sum_{n\equiv0(d)}a_n.
$
We record
\begin{Thm}\cite[Theorem 6.4]{IwaniecKowalski} \label{sieveThm} Suppose that
there exists a finite set $S$ of primes
such that $P$ has no prime factor in $S$. Suppose that
 there exist $\X>1$ and a function $g$ on square-free integers  with $0<g(p)<1$ for $p|P$ and $g$ is multiplicative
outside $S$ (i.e., $g(d_1d_2)=g(d_1)g(d_2)$ if $d_1$ and $d_2$ are square-free integers with no factors in $S$)  such that for all $d\nmid P$ square-free,
\be\label{Adecomp}
|\cA_d|= g(d)\X+r_d(\cA).
\ee
Let $h$ be the multiplicative function on square-free integers (outside $S$) given by $h(p)={g(p)\over 1-g(p)}$.
Then for any $D>1$, we have that
$$
S(\cA,P) \le \X \left(\sum_{d<\sqrt D, d|P} h(d)\right)^{-1}
+
\sum_{d<D, d|P}\tau_3(d)\cdot  |r_d(\cA)|
,
$$
where $\tau_3(d)$ denotes the number of representations of $d$ as the product of three natural numbers.
\end{Thm}

\subsection{Executing the sieve}
Recall from section \ref{red} that $Q$ denotes the Descartes quadratic form
and $\mathcal A$ denotes the Apollonian group in $\O_Q({\z})$.
 Fix a primitive integral Apollonian packing $\mathcal P$ with its root quadruple $\xi\in \z^4$.
Then $\mathcal P$ is either bounded or given by the one in Figure \ref{InfPack}.

To execute the sieve, it is important to work with a simply connected group. Hence
we will set $\G_\mathcal A$ to be the preimage of $\SO_{Q}(\br)^\circ \cap \mathcal A$ in the spin double cover $\operatorname{Spin}_{Q}(\br)$ of
$\SO _Q(\br)$.
Recall that $\alpha$ denotes the Hausdorff dimension of the residual set of the packing $\mathcal P$.
As shown in section \ref{reduction}, $\alpha$ is equal to $\delta_{\G_\mathcal A}$, the Hausdorff dimension
of the limit set of $\G_{\mathcal A}$.
As
$\operatorname{Spin}_{Q}(\br)$ is isomorphic to $\SL_2(\c)$,
we have a real linear representation $\iota: \SL_2(\c)\to \SO_Q(\br)$ which
factors through the quotient map $\SL_2(\c) \to \PSL_2(\c)$. By setting $\G$ to be the preimage
of $\G_{\mathcal A}$ under $\iota$, the counting results in the previous section
are all valid for $\G$.



As before, $B_T$ denotes the ball in the cone:
$$B_T:=\{v\in \br^4: Q(v^t)=0,\;\; \|v\|<T\}.$$
Since we are only looking for an upper bound, we may assume $\|\cdot \|$ is
$g_0 Kg_0^{-1}$-invariant, where $g_0\in \O_Q(\br)$
is such that the stabilizer of $g_0^t v_0 ^t$  is equal to $N^-M$.
We fix a small $\e>0$ and let $\phi_\e \in C_c^\infty(\bH^3)=C_c^\infty(G)^K$
 be as in the proof of Theorem \ref{maincount} with $G=\SL_2(\c)$ and $K=\SU(2)$.


\begin{Def}[Weight] {We define the smoothed weight
to each $\g\in\G$, $$
w_T(\g):=\int_{G/K} \chi_{B_T}(\xi  \g g)\phi_\e(g)
\, d\mu(g) .
$$} \end{Def}

Let $f$ be a primitive integral polynomial in $4$ variables.
Consider the sequence $\cA(T)=\{a_n(T)\}$ where
$$
a_n(T):= \sum_{\g\in\op{Stab}_\G(\xi)\ba \G\atop f(\xi\g)=n}w_T(\g).
$$

Clearly,
$$|\cA(T)|=\sum_n a_n(T)=\sum_{\g\in\op{Stab}_\G(\xi)\ba \G}w_T(\g) $$

and
\be\label{Ad}
|\cA_d(T)| =
\sum_{n\equiv0(d)} a_n(T)=
\sum_{\g\in\G\atop f(\xi\g)\equiv0(d)}w_T(\g).
\ee

For any subgroup $\G_0$ of $\G$ with $$\op{Stab}_\G(\xi)=\op{Stab}_{\G_0} (\xi) ,$$
we define
$$
F_T^{\G_0}(g) := \sum_{\g\in  \op{Stab}_\G(\xi)\ba  \G_0}\chi_{B_T}(\xi \g g),
$$
and for each element $\gamma_1\in \G$, we also define a function on $\G_0\ba \bH^3$ by
$$
\Phi_{\e,\g_1}^{\G_0}(g):=\sum_{\g\in\G_0} \phi_\e(\g_1^{-1}\gamma g)
$$
which is an $\e$-approximation to the identity about $[\g_1^{-1}]$ in $\G_0\bk\bH^3$.
The function $\Phi_{\e, e}^{\G_0} $ is simply the lift of $\phi_\e$ to $\G_0\ba \bH^3$ and will be
denoted by $\Phi_\e^{\G_0}$.

For $d\in \z$,
let $\G_\xi(d)$ be the subgroup of $\G$ which stabilizes $\xi \mod d$, i.e.,
$$
\G_{\xi}(d):=\{ \g\in\G: \xi\g\equiv \xi(d) \}.
$$
Note that
$$\op{Stab}_\G(\xi)=\op{Stab}_{\G_\xi (d)} (\xi) .$$

\begin{Lem}\label{ad}
\begin{enumerate}
 \item $ |\cA (T)|=\<F_T^{\G},\Phi_{\e}^{\G}\>
_{L^2(\G\bk\bH^3)};$
\item  for any integer $d$, $$ |\cA_d(T)|=\sum_{\g_1\in\G_\xi(d)\bk\G\atop f(\xi\g_1)\equiv0(d)}\<F_T^{\G_\xi(d)},\Phi_{\e,\g_1}^{\G_\xi(d)}\>
_{L^2(\G_\xi(d)\bk\bH^3)}.$$
\end{enumerate}
\end{Lem}
\begin{proof}
We have
\begin{align*} |\cA(T)|&=\sum_{\g\in  \op{Stab}_\G(\xi)\ba  \G}w_T(\g) \\
&=
\sum_{\g\in  \op{Stab}_\G(\xi)\ba \G}  \int_{G/K} \chi_{B_T}(\xi \gamma g)\phi_\e (g) d\mu(g) \\ &=
\int_{G/K}F_T^\G(g)\phi_\e(g)\; d\mu(g)\\
 &=
\int_{\Gamma\ba  G/K}F_T^\G(g)\Phi_\e^\Gamma(g)\; d\mu(g)
\\ &= \<F_T^\G,\Phi_\e^\G\>_{L^2(\G\bk\bH^3)}
\end{align*}

Expand \eqref{Ad} as
$$
|\cA_d(T)|=\sum_{\g_1\in\G_\xi(d)\bk\G\atop f(\xi\g_1)\equiv0(d)}\sum_{\g\in \op{Stab}_\G(\xi)\ba
\G_\xi(d)} w_T(\g\g_1).
$$
The inner sum is
\beann
\sum_{\g\in \op{Stab}_\G(\xi)\ba \G_\xi(d)} w_T(\g\g_1)
&=&
\sum_{\g\in \op{Stab}_\G(\xi)\ba \G_\xi(d)}
\int_{G/K} \chi_{B_T}(\xi  \g g) \phi_\e(\g_1^{-1}g)\;d\mu(g)
\\
&=&
\int_{G/K} F_T^{\G_\xi(d)}(g)\phi_\e (\g_1^{-1}g)\;d\mu(g)
\\
&=&
\int_{\G_\xi(d)\bk G/K} F_T^{\G_\xi(d)}(g)\Phi_{\e,\gamma_1}^{\G_\xi(d)}(g)
\; d\mu(g).
\eeann

Thus
$$
\sum_{\g\in\op{Stab}_\G(\xi)\ba \G_\xi(d)} w_T(\g\g_1)
=
\<F_T^{\G_\xi(d)},\Phi_{\e,\g_1}^{\G_\xi(d)}\>
_{L^2(\G_\xi(d)\bk\bH^3)}
.$$

\end{proof}

Denote by $\op{Spec}(\G\ba \bH^3)$ the spectrum of the Laplace operator on $L^2(\Gamma\ba \bH^3)$.
As mentioned before, the work of Sullivan, generalizing Patterson's, implies that
 $\op{Spec}(\G\ba \bH^3)\cap [0,1)\ne \emptyset$ in which case $\lambda_0=\alpha(2-\alpha)$ is
the base eigenvalue of $\Delta$. For the principal congruence subgroup $$\G(d):=\{\gamma\in \G: \gamma = I \;\;
(\text{mod } d)\}$$ of $\G$ of level $d$,
we have $$\op{Spec}(\G\ba \bH^3)\subset \op{Spec}(\G(d) \ba \bH^3).$$

The following is obtained by Bourgain, Gamburd and Sarnak in \cite{BourgainGamburdSarnak2008}
and \cite {BourgainGamburdSarnak2008-p}.
\begin{Thm}\label{bgs} Let $\G$ be a Zariski dense subgroup of $\Spin _{Q}({\z})$ with $\delta_{\G}>1$.
 Then
there exist $1\le \theta <\delta_\G$  such that
we have for all square-free integers $d$,
$$\op{Spec}(\G(d) \ba \bH^3)\cap [\theta(2-\theta), \delta_{\G}(2-\delta_\G)]=\{ \delta_{\G}(2-\delta_\G) \}.$$
\end{Thm}

In order to control the error term for $|\cA_d(T)|$,
we need a version of Corollary
\ref{smc}
uniform over all congruence subgroups $\G_{\xi}(d)$ of $\Gamma$.

\begin{Prop}\label{uniformc} There exists $\e_0>0$, uniform over all square-free integers $d$,
 such that for any $\gamma_1\in \G$ and for any congruence
subgroup $\G_{\xi}(d)$ of $\Gamma$, we have
$$\<F_T^{\G_{\xi}(d)},\Phi_{\e, \g_1} ^{\G_{\xi}(d)}\>
_{L^2(\G_{\xi}(d) \bk\bH^3)}
=\frac{c_{\phi_0} \cdot d_\e }{\delta_\G \cdot [\Gamma:\G_{\xi}(d)]} \cdot \|\xi \|^{\alpha} \cdot T^\alpha
+O( \e^{-3}T^{\alpha-\e_0}) $$
for some $d_\e>0$ where the implied constant depends only on $\Lambda_N(\G)$.
\end{Prop}
\begin{proof}
Note that the congruence subgroup $\G (d)$ of
level $d$ is a finite index subgroup of $\G_{\xi}(d)$.
Since $$\op{Spec}(\G_{\xi}(d)\ba \bH^3)\subset \op{Spec}(\G(d)\ba \bH^3)$$
the spectral gap Theorem \ref{bgs} holds for the family $\G_{\xi}(d)$, $d$ square-free, as well.

As we are assuming $\|\cdot \|$ is $g_0K g_0^{-1}$-invariant, by
Corollary \ref{smc}, we have
\begin{equation}\label{laste}\la F_T, \Phi_{\e, \gamma_1}^{\G_{\xi}(d)}  \ra =  \langle  \Phi_{\e, \gamma_1}^{\G_{\xi}(d)} ,
 \phi_0^{\G_{\xi}(d)} \rangle \cdot\delta_{\G_{\xi}(d)}^{-1}\cdot
 c_{\phi_0^{\G_{\xi}(d)}}
 \cdot T^{\delta_{\G}} \cdot \|\xi\|^{\alpha} \cdot (1+O(T^{-\e_0})) ,
\end{equation}
 where $\e_0$ depends only on the spectral gap for $L^2(\G_\xi(d)\ba \bH^3)$ and the implied constant
 depends only on $$\Lambda_N (g_0^{-1} \Gamma_\xi(d) g_0)=\{[n_x]\in ( N\cap g_0^{-1} \G_\xi(d)g_0) \ba N : x\in \Lambda(g_0^{-1} \G_\xi(d)g_0)\} .$$
As $\G_{\xi}(d)$ is a subgroup of finite index in $\G$,
$\delta_{\G_{\xi}(d)}=\delta_\G$ and
$$ \Lambda(g_0^{-1} \G_\xi(d) g_0 ) =\Lambda(g_0^{-1}\G g_0),$$ and moreover
by the definition of $\G_{\xi}(d)$, we have
$$\op{Stab}_\G (\xi) =\op{Stab}_{\G_\xi(d)} (\xi),$$ implying
$$
N\cap g_0^{-1} \G_\xi(d)g_0 =\ N\cap g_0^{-1} \G g_0 .$$
Hence $$\Lambda_N (g_0^{-1} \Gamma_\xi(d)g_0)=\Lambda_N(g_0^{-1} \G g_0),$$ yielding that
 the implied constant in \eqref{laste} is independent of $d$.

By Theorem \ref{bgs}, $\e_0$ can be taken to be uniform over all $d$.
If $\Gamma_0<\Gamma$ is a subgroup of finite index, the base eigenfunction in $L^2(\Gamma_0\ba \bH^3)$
with the unit $L^2$-norm  is given by $$\phi_0^{\Gamma_0}:=\frac{1}{\sqrt{[\Gamma:\Gamma_0]}}\phi_0$$
with $\phi_0=\phi_0^\Gamma$.
Therefore $$
\langle  \Phi_{\e, \gamma_1}^{\G_{\xi}(d)} ,
 \phi_0^{\G_{\xi}(d)} \rangle\cdot c_{\phi_0^{\G_{\xi}(d)}}
= \langle  \Phi_{\e, e},
 \phi_0 \rangle\cdot c_{\phi_0} \cdot \frac{1}{[\Gamma:\G_{\xi}(d)]}.
 $$
This finishes the proof as $\delta_\G=\alpha$, by setting $d_\e=\langle  \Phi_{\e, e},
 \phi_0 \rangle$.
\end{proof}

Setting $$\X= \delta_\G^{-1}\cdot
c_{\phi_0} \cdot d_\e \cdot \|\xi \|^{\alpha} \cdot T^\alpha ,$$
the following corollary is immediate from Lemma \ref{ad} and Proposition \ref{uniformc}.
\begin{Cor}\label{sl}
There exists $\e_0>0$ uniform over all square-free integers $d$ such that
$$|\cA_d(T)| = \Or^0_f(d)\left({1\over [\G:\G_\xi(d)]}\X+O(\X^{1-\e_0})\right),$$ where
$$
\Or^0_f(d) = \sum_{\g\in\G_\xi(d)\bk\G\atop f(\xi\g)\equiv0(d)}1.
$$
\end{Cor}



\subsection{Proof of Theorem \ref{thmTwo}}
We set
$$f_1(x_1,x_2,x_3,x_4)=x_1\quad \text{and }\quad f_2(x_1,x_2,x_3,x_4)=x_1x_2 .$$

For $d$ square-free and $i=1,2$, set $$g_i(d)=\Or^0_{f_i}(d)/[\G:\G_\xi(d)] .$$
\begin{Prop}\label{g} There exists a finite set $S$ of primes such that: 
\begin{enumerate}\item for any square-free integer $d=d_1d_2$ with no prime factors in
 $ S$ and for each $i=1,2$,
$$g_i(d_1d_2)=g_i(d_1)\cdot g_i(d_2);$$
\item for any prime $p$ outside $S$,
$$g_1(p)\in(0,1)\quad\text{ and }\quad g_1(p)={ p}^{-1}+O(p^{-3/2}) .$$
\item for any prime $p$ outside $S$,
$$g_2(p)\in(0,1)\quad\text{ and }\quad g_2(p)= 2{ p}^{-1}+O(p^{-3/2}) .$$
\end{enumerate}
\end{Prop}
\begin{proof}
According to
 the theorem of Matthews, Vaserstein and Weisfeiler  \cite{MatthewsVasersteinWeisfeiler1984},
there exists a finite set of primes $S$ so that
\begin{itemize}\item for all primes $p$ outside $S$, $\G$ projects onto
 $G(\F_p)$
\item  for $d=p_1\cdots p_t$ square-free with $p_i\notin S$,
the diagonal reduction
$$\G\to G(\z/d\z) \to G(\F_{p_1})\times \cdots G(\F_{p_t})$$ is surjective.
\end{itemize}
Enlarge $S$ 
so that
$G(\F_p)$'s have no common composition factors for different $p$'s outside $S$.
This is possible because $G=\op{Spin}(Q)$ can be realized as $\SL_2$ over
$\q[\sqrt{-1}]$.
Hence there exists a finite set $S$ of primes such that for $p$ outside $S$,
 \begin{equation} G(\F_p)= \begin{cases} \SL_2(\F_p)\times\SL_2(\F_p)
 & \text{for $p\equiv1(4)$ }\\
 \SL_2(\F_{p^2})&\text{ for $p\equiv3(4)$.}\end{cases}\end{equation}

It then follows from Goursat's lemma \cite[p.75]{LangAlgebra} that
$\G$ surjects onto $G(\z/d_1\z)\times G(\z/d_2\z)$ for any square-free $d_1$ and $d_2$ with
no prime factors in $S$.
This implies that for $d=d_1d_2$ square-free and without any prime factors in $S$,
the orbit of $\xi$ mod $d$, say $\mathcal O(d)$, is equal to $\mathcal O(d_1)\times \mathcal O(d_2)$
in $(\z /d_1\z)^4\times (\z/d_2\z)^4= (\z/d\z)^4$.
It also follows that  $\mathcal O^0(d)$ is equal to $\mathcal O^0(d_1)\times \mathcal O^0(d_2)$.
Therefore $g(d)=g(d_1)g(d_2)$ as desired.

Denote by $V$ the cone defined by $Q=0$ minus the origin, i.e.,
$$V=\{(x_1, x_2, x_3 ,x_4)\ne 0: \sum_{i=1}^4 2 x_i^2 -(\sum_{i=1}^4 x_i)^2=0\} .$$

Note that
$$W_1:=\{x\in V: f_1(x)=0\}=\{(0, x_2,x_3, x_4)\ne 0: \sum_{i=2}^4 2 x_i^2 -(\sum_{i=2}^4 x_i)^2=0\} .$$

Since both quadratic forms $$Q(x_1,x_2,x_3,x_4)=
\sum_{i=1}^4 2 x_i^2 -(\sum_{i=1}^4 x_i)^2\;\;\text{and}\;\;
Q(0,x_2,x_3,x_4)=\sum_{i=2}^4 2 x_i^2 -(\sum_{i=2}^4 x_i)^2$$ are absolutely irreducible,
we have  by \cite[Thm. 1.2.B]{BorevichShafarevich},
$$\# V(\F_p)=p^3+O(p^{5/2}) ,\quad \# W_1 (\F_p)=p^2+O(p^{3/2}).$$

Since $V$ is a homogeneous space of $G$ with a connected stabilizer,
by \cite[Prop 3.22]{PlatonovRapinchuk1994},
 $$\mathcal O(p)=V(\F_p),\;\;
\text{and hence }\mathcal O^0_{f_1}(p)=W_1(\F_p) .$$
Therefore we deduce  $g_1(p)=\frac{\# \Or ^0_{f_1}(p)}{\# \Or (p) }= p^{-1}+ O(p^{-3/2})$, proving (1).

Now $W_2:=\{x\in V: f_2(x)=0\}$ is the union of
two quadrics given by $V\cap \{x_1=0\}$ and $V\cap \{x_2=0\}$.
Hence $$\# W_2(\F_p)=2 p^2 +O(p^{3/2}).$$
This yields that $g(p)=2p^{-1}+O(p^{-3/2})$.
\end{proof}

We are now ready to prove Theorem \ref{thmTwo}.
First consider $f_1(x_1,x_2,x_3,x_4)=x_1$ so that $\cA(T)=\{a_n(T)\}$ where
$$
a_n(T):= \sum_{\g\in\op{Stab}_\G(\xi)\ba \G\atop f_1(\xi\g)=n}w_T(\g)
$$
 is a smoothed count for the number of vectors $(x_1,x_2,x_3,x_4)$ in the orbit $\xi \mathcal A^t$ of max
norm bounded above by $T$ and $x_1=n$.

By Lemma \ref{sl} and as $\# \mathcal O^0_{f_1}(d) $ is multiplicative with $\#\mathcal O^0_{f_1}
(p)=p^2+O(p^{3/2})$,
the quantity $ r_d(\cA)$ in the decomposition \eqref{Adecomp} satisfies
for some $\e_0>0$, $$
r_d(\cA)\ll d^2 T^{\alpha-\e_0}.
$$
Thus for any $\vep_1>0$,
$$
\sum_{d<D, d|P}\tau_3(d) |r_d(\cA)|
\ll_{\vep_1}
D^{3+\vep_1}T^{\alpha-\e_0},
$$
which is $\ll T^\alpha/\log T$ for $D=T^{\e_0/4}$, say. The key here is that $D$ can be taken as large as a fixed power of $T$. Let $P$ be the product of all primes $p<D=T^{\e_0/4}$ outside of the bad set $S$.

As $h$ is a multiplicative function defined by
$h(p)=\frac{g_1(p)}{1-g_1(p)}$ for $p$ primes with $g_1$ in Proposition \ref{g},
we deduce that (cf. \cite[6.6]{IwaniecKowalski})
$$
\sum_{d<\sqrt D,d|P} h(d) \gg \log D\gg \log T,
$$
and Theorem \ref{sieveThm} gives
$$ S(\cA(T),P)\ll T^\alpha/\log T.
$$
Therefore
\begin{align*} &\#\{(x_1, x_2, x_3, x_4) \in \xi \mathcal A ^t:
 \max_{1\le i\le 4} |x_i| <T,\,\, (x_1, \prod_{p<T^{\e_0/4}} p )=1\} \\
&\ll
 S(\cA((1+\e) T),P) \ll T^\alpha/\log T .\end{align*}
Hence
$$\#\{(x_1, x_2, x_3, x_4) \in\xi \mathcal A ^t:
 \max_{1\le i\le 4} |x_i| <T,\,\, x_1: \,\text{prime}\} \ll T^\alpha/\log T .$$

Since this argument is symmetric in the $x_i$'s,
we have $$\#\{(x_1, x_2, x_3, x_4) \in \xi \mathcal A ^t:
 \max_{1\le i\le 4} |x_i|=\text{prime at most T} \} \ll T^\alpha/\log T .$$
By Lemma \ref{reduction}, this proves
$$\pi^{\mathcal P}(T)\ll  \frac{T^{\alpha}}{\log T} .$$

In order to prove
\begin{equation}\label{final}\pi^{\mathcal P}_2(T)\ll  \frac{T^{\alpha}}{(\log T)^2} ,\end{equation}
we proceed the same way with the polynomial $f_2(x_1, x_2, x_3, x_4)=x_1x_2$
and with the sequence $\cA(T)=\{a_n(T)\}$ where
$$
a_n(T):= \sum_{\g\in\op{Stab}_\G(\xi)\ba \G\atop f_2(\xi\g)=n} w_T(\g).
$$
 is a smoothed count for the number of vectors $(x_1,x_2,x_3,x_4)$ in the orbit $\xi \mathcal A^t$ of max
norm bounded above by $T$ and $x_1x_2=n$.

Note that $g_2(p)=2p^{-1}+O(p^{-3/2})$ by Proposition \ref{g},
and that $$
\sum_{d<\sqrt D,d|P} h(d) \gg (\log D)^2\gg (\log T)^2
$$
for $h(p)=\frac{g_2(p)}{1-g_2(p)}$ for $p$ primes.

Therefore
 Theorem \ref{sieveThm} gives
$$ S(\cA(T),P)\ll T^\alpha/(\log T)^2
$$
which implies
$$\#\{(x_1, x_2, x_3, x_4) \in \xi \mathcal A ^t:
 \max_{1\le i\le 4} |x_i|<T, x_1, x_2:\text{ primes} \} \ll T^\alpha/{(\log T)^2} .$$
Again by the symmetric property of $x_i$'s,
we have
$$\#\{x \in \xi \mathcal A ^t: \max_{1\le i\le 4} |x_i|<T, x_i, x_j :\text{primes for some $i\ne j$}
 \}\\
 \ll T^\alpha/{(\log T)^2} .$$

By Lemma \ref{reduction},
this proves \begin{align*}\pi^{\mathcal P}_2(T)   \ll T^\alpha/{(\log T)^2} .\end{align*}

\newpage
\renewcommand{\thesection}{A}

\section{Appendix: Non-accumulation of expanding closed horospheres on singular tubes
 (by Hee Oh and Nimish Shah)}
\label{app}
In this appendix, we deduce
Theorem \ref{fttt} from Theorem \ref{ma1}:
Recall that $\G <G=\PSL_2(\c)$ is a torsion free discrete geometrically finite group
with $\delta_\G >1$ and that $\phi_0\in L^2(\G\ba \bH^3)$ denotes the positive
base eigenfunction
of $\Delta$ of eigenvalue $\delta_\G (2-\delta_\G)$ and of norm $\int_{\G\ba \bH^3}
\phi_0^2(g)\, d\mu(g)=1$.

We continue the notations $N, a_y, N^-$, etc., from section \ref{lhorof}.
We assume that $\G\ba \G N$ is closed.
By Corollary \ref{ppoo}, for some $c_{\phi_0}>0$ and $d_{\phi_0}\ge 0$,
$$\int_{\G \ba \G N}\phi_0(na_y)dn =c_{\phi_0}y^{2-\delta_\G} +d_{\phi_0}y^{\delta_\G} .$$

We also recall the Burger-Roblin measure $\hat \mu$ defined
in Theorem \ref{ft} which is normalized so that
$\hat \mu (\phi_0)=1$.

\begin{Thm}\label{fttt}
For $\psi\in C_c(\op{T}^1(\G\ba \bH^3))$,
$$  \int_{n_x\in ( N\cap \G)\ba N} \psi (n_x a_y)\; dx \sim  c_{\phi_0}\cdot \hat\mu(\psi) \cdot y^{2- \delta_{\G}} \quad\text{as $y\to 0$.}$$
\end{Thm}\vs
\begin{proof} The idea of proof is motivated by the proof of Lemma 2.1 in
\cite{RatnerDuke}.
For each $0<y\le 1$, define the measure $\mu_y$ on $\G\ba G/M=\op{T}^1(\G\ba \bH^3)$
by
\begin{align*}
\mu_y(\psi)
&=c_{\phi_0} ^{-1} y^{-2+\delta_{\G}}\int_{(N\cap \G)\ba N} \psi(\G \ba \G [n_x a_y]) \,dx
\end{align*}
for $\psi\in C_c(\G\ba G/M)$.

Consider the family
$$
\mathcal M:=\{\mu_y: 0<y<1\}.
$$

We claim that $\mathcal M$ is
relatively compact
in the set of locally finite Borel measures on $\G\ba G/M $ with respect to the weak$^*$-topology. For any
compact subset $C\subset \G\ba G/M$,
let $\psi$ be a $K$-invariant smooth non-negative function of compact support which is one
over $C$.
Then $$\mu_y(C)\le \mu_y(\psi) .$$
As $\mu_y(\psi) \to  \hat \mu (\psi)=\la \psi, \phi_0\ra $ by Theorem \ref{ma1},
 the claim follows.

It now suffices to show that every accumulation point of $\mathcal M$ is equal to $\hat\mu$.
Let $\mu_0$ be an accumulation point of $\mathcal M$, which is clearly  $N$-invariant.

We denote by $\mathcal E_P$ the set
of vectors $v\in \op{T}^1(\G\ba \bH^3)$ the horospheres determined by which
 is closed. In the identification of $ \op{T}^1(\G\ba \bH^3)$
with $ \G \ba G/M$, the set $\mathcal E_P$ corresponds to the
image under the projection $\G\ba G\to \G\ba G/M$
of the sets $\G\ba \G g NA$ for $\G\ba \G gN$ closed.

 Fix a bounded parabolic fixed point $\xi_0$ of $\G$. If
$\xi_0=g_0(\infty)$ for $g_0\in \PSL_2(\c)$, then a cusp, say,
$D(\xi_0)$, at $\xi_0$ is the image of $\cup_{y>y_1} g_0 N a_y $
 for some $y_1\gg 1$ under the projection
$\pi:G\to \G\ba G/M$.

There exists $c_0>1$ such that for any $z=\G\ba \G g_0n a_y\in D
(\xi_0)$,
\begin{equation}\label{1984} c_0^{-1} y^{r-\delta} \le
\phi_0(z)\le c_0  y^{r-\delta} \end{equation}
 where $r\in \{1,2\}$ is the rank
of $\xi_0$ (see \cite[Sec. 5]{Sullivan1984} as well as the proof
of \cite[Lem. 3.5]{CorletteIozzi}).



Noting that $D_0=D_0(\xi_0)\subset \mathcal E_P$, we first claim
that for any weak limit, say, $\mu_0$, of $\mathcal M$,
\be\label{shah} \mu_0(D_0)=0 .\ee

Let $Q$ be a relatively compact open subset of $D_0$. For
$y_\e>1$, setting $D_\e:=D_0a_{y_\e}\subset D_0$,
  we have
  that \be\label{des} \int_{D_\e K} \phi_0^2 \, d\mu(g)
\ll y_\e^{r-2\delta} .\ee For the part of $D_\e K$ inside the unit
neighborhood of the convex core of $\G$, this estimate follows
from
 the proof of \cite[Lem.
4.2]{CorletteIozzi}. When $r=1$, the integral of $\phi_0^2$ over
the part of $D_\e K$ outside the unit neighborhood of the convex
core of $\G$ is comparable to $$\int_{\log y_\e}^\infty
\int_{x=e^t}^\infty e^{-2\delta t} dx dt \asymp
y_\e^{1-2\delta}.$$

Fixing a neighborhood $U$ of $e$ in
$N^-M$ such that $D_0U\subset D_0K$, the set $D_\e U
a_{y_{\e}}^{-1}$ is a neighborhood of $Q$. Therefore if
$\mu_{y_i}$ weakly converges to $\mu_0$,
\begin{align*}
c_{\phi_0}\cdot  \mu_0(\phi_0\cdot \chi_{Q})& \le \lim_{y_i \to 0}
\frac{1}{y_i^{2-\delta}}\int_{\G\ba \G N}
(\phi_0\cdot \chi_{D_\e U a_{y_\e}^{-1}})(na_{y_i})\, dn \\
&=\lim_{y_i \to 0} \frac{1}{y_i^{2-\delta}}\int_{\G\ba \G N}
\phi_0(na_{y_i}) \cdot \chi_{D_\e U}(na_{y_i}a_{y_\e})\, dn .
 \end{align*}

As $y_\e>1$ and $U\subset N^-M$,
$$a_{y_\e}U a_{y_\e^{-1}}\subset U .$$
Hence if $na_{yy_\e}\in D_\e U$, then  $$na_y\in D_\e
a_{y_\e^{-1}} (a_{y_\e}U a_{y_\e^{-1}})\subset D_\e a_{y_\e^{-1}}
U \subset D_0U .$$ Moreover if $na_{y y_\e}\in \G g_0 n'a_{y'} MU$
for some $n'\in N$ and $y'>y_1$,
\begin{equation}\label{comp} na_y\in \G g_0n'a_{y'}a_{y_\e}^{-1} (a_{y_\e}
MU a_{y_\e}^{-1}) \subset g_0n'a_{y'y_\e^{-1}} MU .\end{equation}

Using the formula $$\phi_0(g)=\int_{\xi\in \Lambda(\G)} e^{-\delta
\beta_{\xi}(go,o)} d\nu_0(\xi)$$ for $g\in G$ (see \eqref{PSS}),
it is easy to check that
 $$ e^{-\delta
d(uo, o)} \phi_0(g)\le \phi_0(gu)\le e^{\delta d(uo, o)}
\phi_0(g)$$ for any $g, u\in G$.
 Therefore we can
deduce from \eqref{1984} and \eqref{comp} that for some constant
$c'>1$,
$$\phi_0(na_y)\le c' \cdot y_\e^{\delta-r}\cdot\phi_0( na_ya_{y_\e}) $$
for all $na_ya_{y_\e}\in D_\e U$.
 Hence
\begin{align*}
& c_{\phi_0}\cdot \mu_0(\phi_0\cdot \chi_{Q})\\
 & \le c_{\phi_0} c' {y_\e}^{2-r}
\lim_{y_i\to 0}\frac{1}{(y_iy_\e)^{2-\delta}}
\int_{\G\ba \G N} (\phi_0\cdot \chi_{D_\e U}) (\G\ba \G na_{y_i}a_{y_\e}) \, dn
\\
&\le
c_{\phi_0} c' y_\e^{2-r} \int_{\G\ba \bH^3} (\phi_0^2 \cdot \chi_{D_\e K})(g)\, d\mu(g)
\quad\text{by Theorem \ref{ma1}}\\
&\ll y_\e^{2-2\delta} \quad\text{by \eqref{des}}.
\end{align*}

Therefore
$$\mu_0(\phi_0\cdot \chi_{Q})\ll y_{\e}^{2-2\delta} .$$
As $\delta >1$, $y_{\e}>1$ is arbitrary, and $\min_{g\in Q}
\phi_0(g) >0$, we have $\mu_0(Q)=0$. This proves the claim
\eqref{shah}.

We now claim that $\mu_0(\mathcal E_P)=0$ for any weak limit
$\mu_0$: $\mu_{y_i}\to \mu_0$.
Suppose not. Since there are only finitely many cusps,
 there exist relatively compact open subset $Q\subset \mathcal E_P$
 and a bounded parabolic fixed point $\xi_0 \in\Lambda(\G)$ such that $\mu_0(Q)>0$ and
its image $a_y(Q)=Qa_y $ under the geodesic flow converges to $\xi_0$ as $y\to \infty$.
Fix $y_0>1$ such that
$$Qa_{y_0}\subset D_0(\xi_0) .$$
By passing to a subsequence, we may assume
$\mu_{y_iy_0}$ is convergent with a weak limit, say, $\mu_0'$.
Since $Qa_{y_0}\subset D_0(\xi_0)$,
by \eqref{shah},
we have $$\mu_0'(Qa_{y_0})=0 .$$
Therefore for any $\e>0$,
there exists a neighborhood $U_\e\subset N^-M$ of $e$
such that
\be\label{ssmal} \mu_0'(Qa_{y_0} U_\e)< y_0^{-2+\delta} \e .\ee

Noting that $Qa_{y_0}U_\e a_{y_0}^{-1}$ is a neighborhood of $Q$,
we have
\begin{align*}c_{\phi_0}\cdot \mu_0(Q)&\le\lim_{y_i\to 0}\frac{1}{y_i^{2-\delta}}\int_{\G\ba \G N}
 \chi_{Qa_{y_0}U_\e a_{y_0}^{-1}}(na_{y_i}) dn\\ &=
y_0^{2-\delta} \lim_{y_i\to 0}\frac{1}{(y_0y_i)^{2-\delta}}\int_{\G\ba \G N}
 \chi_{Qa_{y_0}U_\e}(na_{y_i}a_{y_0}) dn\\&
=y_0^{2-\delta} \mu_0' (Qa_{y_0}U_\e)
\\&\le \e\quad\text{by \eqref{ssmal}}.
\end{align*}

Since $\e>0$ is arbitrary,
$\mu_0(Q)=0$.
This proves
$$\mu_0(\mathcal E_P)=0 .$$

We deduce from Theorem \ref{roblin} that
$$\mu_0=\alpha_1\hat \mu$$ for some $\alpha_1\ge 0$.
On the other hand, by Theorem \ref{ma1},
 $$\mu_0 (\phi_0) =\hat \mu(\phi_0)=1 .$$
It follows that $\alpha_1=1$.
\end{proof}\vs\vs


\begin{thebibliography}{10}

\bibitem{Aubinbook}
Thierry Aubin.
\newblock {\em Nonlinear analysis on manifolds. {M}onge-{A}mp\`ere equations},
  volume 252 of {\em Grundlehren der Mathematischen Wissenschaften [Fundamental
  Principles of Mathematical Sciences]}.
\newblock Springer-Verlag, New York, 1982.

\bibitem{BeardonMaskit1974}
Alan~F. Beardon and Bernard Maskit.
\newblock Limit points of {K}leinian groups and finite sided fundamental
  polyhedra.
\newblock {\em Acta Math.}, 132:1--12, 1974.

\bibitem{BorevichShafarevich}
A.~I. Borevich and I.~R. Shafarevich.
\newblock {\em Number theory}.
\newblock Translated from the Russian by Newcomb Greenleaf. Pure and Applied
  Mathematics, Vol. 20. Academic Press, New York, 1966.

\bibitem{BourgainGamburdSarnak2008}
Jean Bourgain, Alex Gamburd, and Peter Sarnak.
\newblock Affine linear sieve, expanders and sum-product, 2008.
\newblock To appear in Inventiones.

\bibitem{BourgainGamburdSarnak2008-p}
Jean Bourgain, Alex Gamburd, and Peter Sarnak.
\newblock Generalization of {S}elberg's theorem and {S}elberg's sieve, 2009.
\newblock Preprint.

\bibitem{Boyd1982}
David~W. Boyd.
\newblock The sequence of radii of the {A}pollonian packing.
\newblock {\em Math. Comp.}, 39(159):249--254, 1982.

\bibitem{Burger1990}
Marc Burger.
\newblock Horocycle flow on geometrically finite surfaces.
\newblock {\em Duke Math. J.}, 61(3):779--803, 1990.

\bibitem{CorletteIozzi}
Kevin Corlette and Alessandra Iozzi.
\newblock Limit sets of discrete groups of isometries of exotic hyperbolic
  spaces.
\newblock {\em Trans. Amer. Math. Soc.}, 351(4):1507--1530, 1999.

\bibitem{Coxeter1968}
H.~S.~M. Coxeter.
\newblock The problem of {A}pollonius.
\newblock {\em Amer. Math. Monthly}, 75:5--15, 1968.

\bibitem{Dalbo2000}
F.~Dal'bo.
\newblock Topologie du feuilletage fortement stable.
\newblock {\em Ann. Inst. Fourier (Grenoble)}, 50(3):981--993, 2000.

\bibitem{dani1981}
S.~G. Dani.
\newblock Invariant measures and minimal sets of horospherical flows.
\newblock {\em Invent. Math.}, 64(2):357--385, 1981.

\bibitem{EskinMcMullen1993}
Alex Eskin and C.~T. McMullen.
\newblock Mixing, counting, and equidistribution in {L}ie groups.
\newblock {\em Duke Math. J.}, 71(1):181--209, 1993.

\bibitem{EFuchs}
E.~Fuchs.
\newblock 2010.
\newblock Ph. D. Thesis, Princeton University.

\bibitem{GangolliVaradarajanbook}
Ramesh Gangolli and V.~S. Varadarajan.
\newblock {\em Harmonic analysis of spherical functions on real reductive
  groups}, volume 101 of {\em Ergebnisse der Mathematik und ihrer Grenzgebiete
  [Results in Mathematics and Related Areas]}.
\newblock Springer-Verlag, Berlin, 1988.

\bibitem{GrahamLagariasMallowsWilksYanI-n}
Ronald~L. Graham, Jeffrey~C. Lagarias, Colin~L. Mallows, Allan~R. Wilks, and
  Catherine~H. Yan.
\newblock Apollonian circle packings: number theory.
\newblock {\em J. Number Theory}, 100(1):1--45, 2003.

\bibitem{GrahamLagariasMallowsWilksYanI}
Ronald~L. Graham, Jeffrey~C. Lagarias, Colin~L. Mallows, Allan~R. Wilks, and
  Catherine~H. Yan.
\newblock Apollonian circle packings: geometry and group theory. {I}. {T}he
  {A}pollonian group.
\newblock {\em Discrete Comput. Geom.}, 34(4):547--585, 2005.

\bibitem{GuivarchRaugi1986}
Y.~Guivarc'h and A.~Raugi.
\newblock Products of random matrices: convergence theorems.
\newblock In {\em Random matrices and their applications ({B}runswick, {M}aine,
  1984)}, volume~50 of {\em Contemp. Math.}, pages 31--54. Amer. Math. Soc.,
  Providence, RI, 1986.

\bibitem{Hirst1967}
K.~E. Hirst.
\newblock The {A}pollonian packing of circles.
\newblock {\em J. London Math. Soc.}, 42:281--291, 1967.

\bibitem{IwaniecKowalski}
Henryk Iwaniec and Emmanuel Kowalski.
\newblock {\em Analytic number theory}, volume~53 of {\em American Mathematical
  Society Colloquium Publications}.
\newblock American Mathematical Society, Providence, RI, 2004.

\bibitem{Kapovichbook}
Michael Kapovich.
\newblock {\em Hyperbolic manifolds and discrete groups}, volume 183 of {\em
  Progress in Mathematics}.
\newblock Birkh\"auser Boston Inc., Boston, MA, 2001.

\bibitem{Kontorovich2007}
Alex Kontorovich.
\newblock The hyperbolic lattice point count in infinite volume with
  applications to sieves.
\newblock {\em Duke Math. J.}, 149(1):1--36, 2009.

\bibitem{KontorovichOh2008}
Alex Kontorovich and Hee Oh.
\newblock Almost prime {P}ythagorean triples in thin orbits.
\newblock {\em Preprint, 2009}.

\bibitem{LangAlgebra}
Serge Lang.
\newblock {\em Algebra}, volume 211 of {\em Graduate Texts in Mathematics}.
\newblock Springer-Verlag, New York, third edition, 2002.

\bibitem{LaxPhillips}
Peter~D. Lax and Ralph~S. Phillips.
\newblock The asymptotic distribution of lattice points in {E}uclidean and
  non-{E}uclidean spaces.
\newblock {\em J. Funct. Anal.}, 46(3):280--350, 1982.

\bibitem{Margulisthesis}
Gregory Margulis.
\newblock {\em On some aspects of the theory of {A}nosov systems}.
\newblock Springer Monographs in Mathematics. Springer-Verlag, Berlin, 2004.
\newblock With a survey by Richard Sharp: Periodic orbits of hyperbolic flows,
  Translated from the Russian by Valentina Vladimirovna Szulikowska.

\bibitem{Maskit1988}
Bernard Maskit.
\newblock {\em Kleinian groups}, volume 287 of {\em Grundlehren der
  Mathematischen Wissenschaften [Fundamental Principles of Mathematical
  Sciences]}.
\newblock Springer-Verlag, Berlin, 1988.

\bibitem{MatthewsVasersteinWeisfeiler1984}
C.~Matthews, L.~Vaserstein, and B.~Weisfeiler.
\newblock Congruence properties of {Z}ariski-dense subgroups.
\newblock {\em Proc. London Math. Soc}, 48:514--532, 1984.

\bibitem{Maucourant2007}
F.~Maucourant.
\newblock Homogeneous asymptotic limits of {H}aar measures of semisimple linear
  groups and their lattices.
\newblock {\em Duke Math. J.}, 136(2):357--399, 2007.

\bibitem{McMullen1998}
C.~T. McMullen.
\newblock Hausdorff dimension and conformal dynamics. {III}. {C}omputation of
  dimension.
\newblock {\em Amer. J. Math.}, 120(4):691--721, 1998.

\bibitem{MohammadiGolsefidy}
Amir Mohammadi and Alireza Salehi~Golsefidy.
\newblock Translates of horospherical measures and counting problems.
\newblock {\em Preprint}, 2008.

\bibitem{OhICM}
Hee Oh.
\newblock Dynamics on {G}eometrically finite hyperbolic manifolds with
  applications to {A}pollonian circle packings and beyond.
\newblock {\em Proc. of ICM (Hyderabad, 2010)}.

\bibitem{OhShahGFH}
Hee Oh and Nimish Shah.
\newblock Equidistribution and counting for orbits of geometrically finite
  hyperbolic groups.
\newblock {\em Preprint}.

\bibitem{OhShahcircle}
Hee Oh and Nimish Shah.
\newblock The asymptotic distribution of circles in the orbits of {K}leinian
  groups.
\newblock {\em Preprint}.


\bibitem{Patterson1976}
S.J. Patterson.
\newblock The limit set of a {F}uchsian group.
\newblock {\em Acta Mathematica}, 136:241--273, 1976.

\bibitem{PlatonovRapinchuk1994}
Vladimir Platonov and Andrei Rapinchuk.
\newblock {\em Algebraic groups and number theory}, volume 139 of {\em Pure and
  Applied Mathematics}.
\newblock Academic Press Inc., Boston, MA, 1994.
\newblock Translated from the 1991 Russian original by Rachel Rowen.

\bibitem{Ratner1991}
Marina Ratner.
\newblock On {R}aghunathan's measure conjecture.
\newblock {\em Ann. of Math. (2)}, 134(3):545--607, 1991.

\bibitem{RatnerDuke}
Marina Ratner.
\newblock Raghunathan's topological conjecture and distributions of unipotent
  flows.
\newblock {\em Duke Math. J.}, 63(1):235--280, 1991.

\bibitem{Roblin2003}
Thomas Roblin.
\newblock Ergodicit\'e et \'equidistribution en courbure n\'egative.
\newblock {\em M\'em. Soc. Math. Fr. (N.S.)}, (95):vi+96, 2003.

\bibitem{Sarnak1981}
Peter Sarnak.
\newblock Asymptotic behavior of periodic orbits of the horocycle flow and
  eisenstein series.
\newblock {\em Comm. Pure Appl. Math.}, 34(6):719--739, 1981.

\bibitem{SarnakToLagarias}
Peter Sarnak.
\newblock Letter to {J}. {L}agarias, 2007.
\newblock available at www.math.princeton.edu/{~}sarnak.

\bibitem{Schapira2005}
Barbara Schapira.
\newblock Equidistribution of the horocycles of a geometrically finite surface.
\newblock {\em Int. Math. Res. Not.}, (40):2447--2471, 2005.

\bibitem{Shalom2000}
Yehuda Shalom.
\newblock Rigidity, unitary representations of semisimple groups, and
  fundamental groups of manifolds with rank one transformation group.
\newblock {\em Ann. of Math. (2)}, 152(1):113--182, 2000.

\bibitem{Soddy1937}
F.~Soddy.
\newblock The bowl of integers and the hexlet.
\newblock {\em Nature}, 139:77--79, 1936.

\bibitem{Soddy1936}
F.~Soddy.
\newblock The kiss precise.
\newblock {\em Nature}, 137:1021, 1937.

\bibitem{Sullivan1979}
Dennis Sullivan.
\newblock The density at infinity of a discrete group of hyperbolic motions.
\newblock {\em Inst. Hautes \'Etudes Sci. Publ. Math.}, (50):171--202, 1979.

\bibitem{Sullivan1984}
Dennis Sullivan.
\newblock Entropy, {H}ausdorff measures old and new, and limit sets of
  geometrically finite {K}leinian groups.
\newblock {\em Acta Math.}, 153(3-4):259--277, 1984.

\bibitem{Thurstonbook}
William Thurston.
\newblock {\em The Geometry and Topology of Three-Manifolds}.
\newblock available at www.msri.org/publications/books. Electronic
  version-March 2002.

\bibitem{Wilker1977}
J.~B. Wilker.
\newblock Sizing up a solid packing.
\newblock {\em Period. Math. Hungar.}, 8(2):117--134, 1977.

\end{thebibliography}
\end{document}